\documentclass[10pt, reqno]{amsart} 
\usepackage{enumerate} 
\usepackage{amsmath,amsfonts, amscd, bbm, eucal, latexsym, amssymb, extarrows, fixltx2e,mathrsfs, rotating}
\usepackage[utf8]{inputenc} 
\usepackage[T1]{fontenc}
\usepackage{ae,aecompl}
\usepackage{dsfont} 
\usepackage{pxfonts}
\usepackage{microtype}
\usepackage{braket}

\input xy
 \xyoption{all}

 \usepackage[colorlinks=true, pdfstartview=FitV,
linkcolor=cyan, citecolor=red, urlcolor=blue]{hyperref}

\newif\ifPDF
\ifx\pdfoutput\undefined 
        \ifx\pdfversion\undefined
                \PDFfalse
        \fi
\else 
        \ifnum\pdfoutput > 0
                \PDFtrue
        \else
                \PDFfalse
        \fi
\fi

\makeatletter
\newcommand{\xRightarrow}[2][]{\ext@arrow 0359\Rightarrowfill@{#1}{#2}}
\makeatother

\theoremstyle{plain}
				
\newtheorem{theorem}{Theorem}[section]				
\newtheorem{proposition}[theorem]{Proposition}		
\newtheorem{corollary}[theorem]{Corollary}
\newtheorem{lemma}[theorem]{Lemma}

\theoremstyle{definition}

\newtheorem{definition}[theorem]{Definition}

\newtheorem{remark}[theorem]{Remark}
\newtheorem{example}[theorem]{Example}

\newcommand{\ov}{\overline}

\newcommand{\CBbb}{\mathbb C}

\newcommand{\GBbb}{\mathbb G}

\newcommand{\QBbb}{\mathbb Q}
\newcommand{\RBbb}{\mathbb R}

\newcommand{\ZBbb}{\mathbb Z}

\newcommand{\Acal}{\mathcal A}

\newcommand{\Ccal}{\mathcal C}
\newcommand{\Dcal}{\mathcal D}

\newcommand{\Hcal}{\mathcal H}

\newcommand{\Lcal}{\mathcal L}
\newcommand{\Mcal}{\mathcal M}

\newcommand{\Ocal}{\mathcal O}
\newcommand{\Pcal}{\mathcal P}

\newcommand{\Scal}{\mathcal S}
\newcommand{\Tcal}{\mathcal T}
\newcommand{\Ucal}{\mathcal U}
\newcommand{\Vcal}{\mathcal V}

\newcommand{\Xcal}{\mathcal X}

\newcommand{\pfrak}{\mathfrak p}

\DeclareMathOperator{\Hom}{Hom}

\DeclareMathOperator{\tr}{tr}

\DeclareMathOperator{\Div}{div}
\DeclareMathOperator{\ord}{ord}

\DeclareMathOperator{\PICRIG}{PICRIG}
\DeclareMathOperator{\PIC}{PIC}

\DeclareMathOperator{\LOG}{LOG}
\DeclareMathOperator{\Gm}{\mathbb{G}_{m}}
\DeclareMathOperator{\Spec}{Spec}
\DeclareMathOperator{\Real}{Re}
\DeclareMathOperator{\Imag}{Im}

\newcommand{\dbar}{\bar\partial}

\newcommand{\lra}{\longrightarrow}

\newcommand{\Pic}{{\rm Pic}}

\newcommand{\GM}{{\rm GM}}

\newcommand{\isorightarrow}{\xrightarrow{
   \,\smash{\raisebox{-0.5ex}{\ensuremath{\sim}}}\,}}
   
\begin{document}

\title[Flat line bundles and the Cappell-Miller torsion in Arakelov geometry]{Flat line bundles and the Cappell-Miller torsion in Arakelov geometry}

\author[Freixas i Montplet]{Gerard Freixas i Montplet}

\address{CNRS -- Institut de Math\'ematiques de Jussieu, 4 Place Jussieu,
   75005 Paris, France}
\email{gerard.freixas@imj-prg.fr}
\thanks{G. F. supported in part by ANR grant ANR-12-BS01-0002.}

\author[Wentworth]{Richard A. Wentworth}

\address{Department of Mathematics,
   University of Maryland,
   College Park, MD 20742, USA}
\email{raw@umd.edu}
\thanks{R.W. supported in part by NSF grant DMS-1406513.}
\thanks{ The authors also acknowledge support from NSF grants DMS-1107452, -1107263, -1107367 ``RNMS: GEometric structures And Representation varieties'' (the GEAR Network).
}

\begin{abstract}
 In this paper, we extend Deligne's functorial Riemann-Roch isomorphism for hermitian holomorphic line bundles on Riemann surfaces  to the case of  flat, not necessarily unitary connections.  The Quillen metric and $\star$-product of Gillet-Soul\'e is replaced  with complex valued logarithms. On the determinant of cohomology side, the idea goes back to Fay's holomorphic extension of determinants of Dolbeault laplacians, and it is  shown here to be  equivalent to the holomorphic Cappell-Miller torsion. On the Deligne pairing side, the logarithm is a refinement of the intersection connections considered in \cite{FreixasWentworth:15}. The construction naturally leads to an Arakelov theory for flat line bundles on arithmetic surfaces and produces arithmetic intersection numbers valued in $\CBbb/\pi i\,  \ZBbb$. In this context we prove an arithmetic Riemann-Roch theorem. This realizes a program proposed by Cappell-Miller to show that the holomorphic torsion exhibits properties similar to those of the Quillen metric proved by Bismut, Gillet and Soul\'e. Finally, we give examples that clarify the kind of invariants that the formalism captures; namely,  periods of differential forms.
\end{abstract}

\maketitle

\thispagestyle{empty}

\section{Introduction}
Arithmetic intersection theory was initiated by Arakelov \cite{Arakelov} in an attempt to approach the Mordell conjecture on rational points of projective curves over number fields by mimicking the successful arguments of the function field case. The new insight was the realization that an intersection theory on arithmetic surfaces could be defined by adding some archimedean information to divisors. This archimedean datum consists of the so-called Green's functions that arise from smooth hermitian metrics on holomorphic line bundles. The use of a metric structure is also natural for diophantine purposes, as one may want to measure the size of integral sections of a line bundle on an arithmetic surface. 

Arakelov's foundational work was  complemented by Faltings, who proved among other things   the first version of an arithmetric Riemann-Roch type formula \cite{Faltings}. Later, in a long collaboration starting with \cite{GS:AIT}, Gillet and Soul\'e vastly extended the theory both to higher dimensions and to more general structures on the archimedean side. Their point of view is an elaboration of the ideas of Arakelov and is cast as a suitable ``completion'' of the usual Chow groups of classical intersection theory over a Dedekind domain. Their formalism includes arithmetic analogues of characteristic classes of hermitian holomorphic vector bundles
\cite{GS:ACC1, GS:ACC2}. This led them to develop and prove a general Grothendieck-Riemann-Roch type theorem in this setting \cite{GS:ARR}. A key ingredient is the  \emph{analytic torsion} of the Dolbeault complex associated to a hermitian holomorphic bundle
over a compact K\"ahler manifold. Their proof requires deep properties of the analytic torsion due to Bismut and collaborators \cite{Bismut-Freed-1, Bismut-Freed-2, BGS1, BGS2, BGS3, Bismut-Lebeau}. In \cite{Deligne:87}, Deligne proposed a program to lift the Grothendieck-Riemann-Roch theorem to a functorial isomorphism between line bundles that  becomes an isometry when vector bundles are endowed with suitable metrics. In the case of families of curves this goal was achieved. He establishes a canonical isometry between the determinant of cohomology of a hermitian vector bundle with the Quillen metric, and some hermitian intersection bundles, involving in particular the so-called \emph{Deligne pairings} of line bundles. 

In our previous work \cite{FreixasWentworth:15}, we produced natural connections on Deligne pairings of line bundles with flat relative  connections on families of compact Riemann surfaces. These were called \emph{intersection connections}, and they recover  Deligne's constructions in the case where the relative connections are the  Chern connections for a hermitian structure.  As in the case of Deligne's formulation, intersection connections are functorial, and via the Chern-Weil expression they realize a natural cohomological relationship for Deligne pairings. Moreover, we showed that in the case of a trivial family of curves, \emph{i.e.} a single Riemann surface and a holomorphic family of flat line bundles on it, we could interpret Fay's holomorphic extension of analytic torsion for flat unitary line bundles \cite{Fay} as the construction of a Quillen type  holomorphic connection on the determinant of cohomology, and as a statement that the Deligne-Riemann-Roch type isomorphism is flat with respect to these connections. The contents of \cite{FreixasWentworth:15} are summarized in Section \ref{section:preliminar} below. 

The principal results of the present paper are the following.
\begin{itemize}
\item We extend the flatness of the Deligne isomorphism to nontrivial families of smooth projective curves. The proof makes use of the idea of a \emph{logarithm} for a line bundle with connection.
\item The holomorphic extension of analytic torsion naturally defines an example of a logarithm which we call the \emph{Quillen logarithm}. We show that the Quillen logarithm coincides with the torsion invariant defined by Cappell-Miller in \cite{CM}.
\item We initiate an arithmetic intersection theory where the archimedean data consists of  flat, not necessarily unitary, connections.
\end{itemize}
Below we describe each of these items in more detail.

\subsection{Logarithms}
The results in \cite{FreixasWentworth:15} on intersection and Quillen connections are vacuous for a single Riemann surface and a single flat holomorphic line bundle, since there are no interesting  connections over points! To proceed further, and especially with applications to Arakelov theory in mind, we establish
  ``integrated'' versions of the intersection and holomorphic Quillen connections that are nontrivial even when the parameter space is zero dimensional. The nature of such an object is what we have referred to above as a logarithm of a line bundle $\Lcal\to S$ over a smooth variety $S$. This takes the place  of the logarithm of a hermitian metric in the classical situation. More precisely, a logarithm is an equivariant map $\LOG: \Lcal^\times\to \CBbb/2\pi i\, \ZBbb$. It has an associated connection which generalizes the Chern connection of a hermitian metric, but which is not necessarily unitary for some hermitian structure.
Although the notion of a logarithm is equivalent simply to  a trivialization of the $\GBbb_m$-torsor $\Lcal^\times$, it nevertheless plays an important role in the archimedean part of the arithmetic intersection product, as we explain below.

The logarithm provides a refinement of the relationship  between intersection connections, the holomorphic extension of  analytic torsion, and the Deligne isomorphism. More precisely, let $(X,p)$ be a compact Riemann surface with a base point, $\ov{X}$ the conjugate Riemann surface, and $\Lcal\to X$, $\Lcal^{c}\to\ov{X}$ rigidified (at $p$) flat complex line bundles with respective holonomies $\chi$ and $\chi^{-1}$, for some character $\chi\colon\pi_{1}(X,p)\to\CBbb^{\times}$. Applied to these data, Deligne's canonical (up to sign) isomorphism for $\Lcal$ and $\Lcal^{c}$ gives
 \begin{equation} \label{eqn:deligne}
 	\Dcal\colon \left\{\lambda(\Lcal-\Ocal_{X})\otimes_{\mathbb{C}}\lambda(\Lcal^{c}-\Ocal_{\overline{X}})\right\}^{\otimes 2}
	\isorightarrow\langle\Lcal,\Lcal\otimes\omega_{X}^{-1}\rangle\otimes_{\mathbb{C}}\langle\Lcal^{c},\Lcal^{c}\otimes\omega_{\overline{X}}^{-1}\rangle
 \end{equation}
 where $\lambda$ denotes the (virtual) determinant of cohomology for the induced holomorphic structure, and $\langle\ ,\ \rangle$ denotes the Deligne pairing (see Section \ref{section:preliminar} below for a review of Deligne's isomorphism). Choosing a hermitian metric on $T_{X}$, one can define a holomorphic extension of analytic torsion as in \cite{Fay}, and this gives rise to a natural \emph{Quillen logarithm} $\LOG_Q$ on the left hand side of
\eqref{eqn:deligne}, and an associated generalization of the Quillen connection. On the other hand, we shall show in Section \ref{section:int-log} that the intersection connection of \cite{FreixasWentworth:15} can be integrated  to an \emph{intersection logarithm} $\LOG_{int}$ on the right hand side of \eqref{eqn:deligne}.  The first main result  is the following (see Section \ref{section:LOG-det-coh} below, especially Theorem \ref{theorem:Iso-Deligne} and Corollary \ref{corollary:iso-Deligne}):

 \begin{theorem}\label{theorem:iso-Deligne}
 Deligne's isomorphism \eqref{eqn:deligne} is compatible with $\LOG_Q$ and $\LOG_{int}$, modulo $\pi i\,  \ZBbb$. That is,
 \begin{equation} \label{deligne-log}
 \LOG_{Q}=\LOG_{int}\circ\Dcal  \mod \pi i\,  \ZBbb
 \end{equation}
 Moreover, in families the Deligne isomorphism is flat with respect to the Quillen and intersection connections.
 \end{theorem}
The proof we give relies on our previous work. The idea is to deform the line bundles to the universal family over the Betti moduli space $M_{B}(X)=\Hom(\pi_{1}(X,p),\CBbb^{\times})$, over which we have previously proven the compatibility of Deligne's isomorphism with holomorphic Quillen and intersection connections. In \cite[Sec.\ 5.3]{FreixasWentworth:15}, these connections where shown to be flat over $M_{B}(X)$. Since the Quillen and intersection logarithms are primitives for the logarithms, the logarithms must therefore coincide up to a constant. The constant is fixed by evaluation on unitary characters, for which Deligne's isomorphism is an isometry. The ambiguity of sign in the isomorphism \eqref{eqn:deligne} is responsible for taking the values in the equality in the theorem modulo $\pi i\,  \ZBbb$, instead of $2\pi i\, \ZBbb$. 

\subsection{The Quillen-Cappell-Miller logarithm}
In recent years, several authors have developed complex valued analogues of analytic torsion for flat line bundles  \cite{Kim:07,  BravermanKappeler:08, McIntyreTeo:08}. In the holomorphic case, this is due to Cappell-Miller \cite{CM}. With the notation above, the Cappell-Miller holomorphic torsion can be seen as a trivialization of $\lambda(\Lcal)\otimes_{\CBbb}\lambda(\Lcal^{c})$ that depends on the hermitian metric on $T_{X}$ and the connections on the line bundles. Hence, it gives raise to a logarithm, which we temporarily call the Cappell-Miller logarithm. In \cite{CM}, the authors ask whether their torsion has similar properties to the holomorphic torsion, as in the work of Bismut-Gillet-Soul\'e and Gillet-Soul\'e. Our second main result is that this is indeed the case for Riemann surfaces and line bundles. In fact, we prove the following (see Section \ref{section:CM} and Theorem \ref{thm:Quillen-CM} below):

\begin{theorem}\label{theorem:CM-Quillen}
The Cappell-Miller logarithm and the Quillen logarithm coincide.
\end{theorem}
In light of  the theorem, we may also call $\LOG_{Q}$ the \emph{Quillen-Cappell-Miller logarithm}. The idea of the proof is again by deformation to the universal case over $M_{B}(X)$. We prove that Cappell-Miller's construction can be done in familes and that it provides a \emph{holomorphic} trivialization of the ``universal'' determinant of cohomology on $M_{B}(X)$. The strategy is analogous to the observation in \cite{BGS3}, according to which Bismut-Freed's construction of the determinant of cohomology and Quillen's metric are compatible with the holomorphic structure of the Knudsen-Mumford determinant \cite{KM}. However, these authors work with hermitian vector bundles, and their reasoning is particular to the theory of self-adjoint Laplace type operators.  As our operators are not self-adjoint, this argument does not directly apply. Nevertheless, the operators we consider are still conjugate, \`a la Gromov, to a fixed self-adjoint laplacian on a fixed domain. This presentation exhibits a holomorphic dependence with respect to parameters in $M_{B}(X)$. In this context, Kato's theory of analytic perturbations of closed operators \cite[Chap.\ VII]{Kato} turns out to be well-suited, and provides the necessary alternative arguments to those in  \cite{BGS3}. Once this is completed, we obtain two holomorphic logarithms on the universal determinant of cohomology that agree on the unitary locus. By a standard argument this implies that they must coincide everywhere. \emph{A posteriori}, we remark that the analogue of the curvature theorems of Bismut-Freed and Bismut-Gillet-Soul\'e for the Cappell-Miller torsion is empty, since we prove the latter gives rise to a flat Quillen type connection in the family situation. This makes our functorial approach essential in order to establish nontrivial finer properties of the Cappell-Miller torsion.

\subsection{The Arithmetic-Riemann-Roch theorem}
The third aim of this paper is to use the results above to initiate an Arakelov theory for flat line bundles on arithmetic surfaces (Section \ref{section:AIT}). The quest for such a theory was made more conceivable by Burgos' cohomological approach to Arakelov geometry, which interprets Green currents as objects in some truncated Deligne \emph{real} cohomology \cite{Burgos}. This evolved into the abstract formalism of Burgos-Kramer-K\"uhn \cite{BKK}, allowing one to introduce \emph{integral} Deligne cohomology instead. Despite these developments, to our knowledge, the attempts so far have been unsuccessful. It turns out that the intersection logarithm is the key in the construction of an arithmetic intersection pairing for flat line bundles. At the archimedean places, the nature of our tools forces us to work simultaneously with a Riemann surface and its conjugate, and pairs of flat line bundles with opposite holonomies. 
We find an analogue of this apparatus in the arithmetic setting 
which we call a \emph{conjugate pair} $\Lcal^\sharp$ of line bundles with connections (see Definition \ref{def:conjugate-pair}). 
Through Deligne's pairing and the intersection logarithm, we attach to conjugate pairs $\Lcal^{\sharp}$ and $\Mcal^{\sharp}$
 an object $\langle\Lcal^{\sharp},\Mcal^{\sharp}\rangle$, which consists of a line bundle over $\Spec\Ocal_{K}$ together with the data of intersection logarithms at the archimedean places. For such an object there is a variant of the arithmetic degree in classical Arakelov geometry, denoted  $\deg^{\sharp}$, which takes values in $\CBbb/\pi i\,  \ZBbb$ instead of $\RBbb$. The construction also applies to mixed situations; for instance, to a rigidified conjugate pair $\Lcal^{\sharp}$ and a hermitian line bundle $\ov{\Mcal}$. When the dualizing sheaf $\omega_{\Xcal/S}$ is equipped with a smooth hermitian metric, we can define $\lambda(\Lcal^{\sharp})_{Q}$, the determinant of cohomology of $\Lcal^{\sharp}$ with the Quillen-Cappell-Miller logarithms at the archimedean places. Using this formalism, we prove an arithmetic Riemann-Roch type theorem for these enhanced line bundles (Theorem \ref{theorem:arithmetic-RR} below):

\begin{theorem}[{\sc Arithmetic Riemann-Roch}]     \label{thm:arr}
Let $\Xcal\rightarrow S=\Spec\Ocal_K$ be an arithmetic surface with a section $\sigma:S\rightarrow\Xcal$.
Suppose the relative dualizing sheaf $\omega_{\Xcal/S}$ is endowed  with a smooth hermitian metric. Let $\Lcal^{\sharp}$ be a rigidified conjugate pair of line bundles with connections. Endow the determinant of cohomology of $\Lcal^{\sharp}$ with the Quillen-Cappell-Miller logarithm. Then the following equality holds in $\CBbb/\pi i\,  \ZBbb$.
\begin{align}
	\begin{split}
	12\deg^{\sharp}\lambda(\Lcal^{\sharp})_{Q}-2\delta
	&=2(\overline{\omega}_{\Xcal/S},\overline{\omega}_{\Xcal/S})
	+6(\Lcal^{\sharp},\Lcal^{\sharp})
	-6(\Lcal^{\sharp},\overline{\omega}_{\Xcal/S})\\
&\qquad	-(4g-4)[K:\QBbb]
	\left(\frac{\zeta'(-1)}{\zeta(-1)}+\frac{1}{2}\right),
	\end{split}
\end{align}
where $\delta=\sum_{\pfrak}n_{\pfrak}\log (N\pfrak)$ is the ``Artin conductor'' measuring the bad reduction of $\Xcal\rightarrow\Spec\Ocal_K$. If $K$ does not admit any real embeddings then the equality lifts to  $\CBbb/2\pi i\, \ZBbb$.
\end{theorem}
In the theorem it is possible to avoid the rigidification of $\Lcal^{\sharp}$ along the section $\sigma$, at the cost of taking values in $\CBbb/\pi i\,  \ZBbb[1/h_{K}]$, where $h_{K}$ is the class number of $K$. The particular choice of section is not relevant: none of the quantities computed by the formula depends upon it. However, the existence of a section is needed for the construction. A variant of the formalism (including an  arithmetic Riemann-Roch formula) consists in introducing conjugate pairs of arithmetic surfaces and line bundles. This makes sense and can be useful when $K$ is a CM field. The arithmetic intersection numbers are then valued in $\CBbb/2\pi i\,\ZBbb$.


\tableofcontents

\section{Deligne-Riemann-Roch and intersection connections}\label{section:preliminar}
In this section we briefly review those results from our previous work \cite{FreixasWentworth:15} that are relevant for the present article. Let $\pi\colon\Xcal\to S$ be a smooth and proper morphism of quasi-projective and smooth complex varieties, with connected fibers of dimension 1. Let $\Lcal$ and $\Mcal$ be two holomorphic line bundles on $\Xcal$. The Deligne pairing of $\Lcal$ and $\Mcal$ is a holomorphic line bundle $\langle \Lcal,\Mcal\rangle$ on $S$, that can be presented in terms of generators and relations. Locally on $S$ (\emph{i.e.} possibly after replacing $S$ by an open subset), the line bundle is trivialized by symbols $\langle\ell, m\rangle$, where $\Div \ell$ and $\Div m$ are disjoint, finite and \'etale\footnote{Under the most general assumptions (l.c.i. flat morphisms between schemes), it only makes sense to require flatness of the divisors. In our setting (smooth morphisms of smooth varieties over $\CBbb$), a Bertini type argument shows we can take them to be \'etale \cite[Lemma 2.8]{FreixasWentworth:15}.} over an open subset of $S$ (for simplicity, we say that $\ell$ and $m$ are in \emph{relative general position}).
 Relations, inducing the glueing and cocycle conditions, are given by
\begin{displaymath}
	\langle f\ell,m\rangle=N_{\Div m/S}(f)\langle\ell,m\rangle,
\end{displaymath} 
whenever $f$ is a meromorphic function such that both symbols are defined, as well as a symmetric relation in the other ``variable''. Here, $N_{\Div m/S}(f)$ denotes the norm of $f$ along the divisor of $m$. It is multiplicative with respect to addition of divisors, and it is equal to the usual norm on functions for finite, flat divisors over the base. The construction is consistent, thanks to the Weil reciprocity law: for two meromorphic functions $f$ and $g$ whose divisors are in relative general position, we have
\begin{displaymath}
	N_{\Div f/S}(g)=N_{\Div g/S}(f).
\end{displaymath}
The Deligne pairing can be constructed both in the analytic and the algebraic categories, and it is compatible with the analytification functor. This is why we omit specifying the topology. The Deligne pairing is compatible with base change and has natural functorial
properties in $\Lcal$ and $\Mcal$.

Let $\nabla \colon \Lcal\to \Lcal\otimes\Omega^{1}_{\Xcal/S}$ be a relative holomorphic connection, and assume for the time being that $\Mcal$ has relative degree 0. We showed that there exists a $\Ccal^{\infty}_{\Xcal}$ connection $\widetilde{\nabla}\colon \Lcal\to \Lcal\otimes\Acal^{1}_{\Xcal}$, \emph{compatible} with the holomorphic structure on $\Lcal$ (this is $\widetilde{\nabla}^{0,1}=\ov{\partial}_{\Lcal}$), such that the following rule determines a well defined compatible connection on $\langle \Lcal,\Mcal\rangle$:
\begin{displaymath}
	\nabla_{tr}\langle\ell, m\rangle=\langle\ell,m\rangle\otimes\tr_{\Div m/S}\left(\frac{\widetilde{\nabla}\ell}{\ell}\right).
\end{displaymath}
Notice that it makes sense to take the trace of the differential form $\widetilde{\nabla}\ell/\ell$ along $\Div m$, since the latter is finite \'etale over the base, and the divisors of the sections are disjoint. The existence of $\widetilde{\nabla}$ is not obvious, since the rule just defined encodes a nontrivial reciprocity law, that we call (WR):
\begin{displaymath}
	\tr_{\Div f/S}\left(\frac{\widetilde{\nabla}\ell}{\ell}\right)=\tr_{\Div\ell/S}\left(\frac{df}{f}\right),
\end{displaymath}
whenever $f$ is a meromorphic function and the divisors of $f$ and $\ell$ are in relative general position. The construction of $\widetilde{\nabla}$ can be made to be compatible with base change, and then  it is  unique up to $\Gamma(\Xcal,\pi^{-1}\Acal^{1,0}_{S})$. Furthermore, if $\sigma\colon S\to\Xcal$ is a section and $\Lcal$ is trivialized along $\sigma$, one can isolate a particular extension $\widetilde{\nabla}$ that restricts to the exterior differentiation on $S$ along $\sigma$ (through the trivialization of $\Lcal$). Then the connection $\nabla_{tr}$ can be extended to $\Mcal$ of any relative degree, without ambiguity. We call $\widetilde{\nabla}$ a (or the) canonical extension of $\nabla$, and $\nabla_{tr}$ a trace connection.

Trace connections are manifestly not symmetric, since they do not require any connection on $\Mcal$. Let $\widetilde{\nabla}^{\prime}\colon \Mcal\to \Mcal\otimes\Acal_{\Xcal}^{1}$ be a smooth compatible connection on $\Mcal$ and let $\nabla_{tr}$ be a trace connection on $\langle\Lcal,\Mcal\rangle$. If the relative degree of $\Mcal$ is not zero, we tacitly assume that $\Lcal$ is rigidified along a given section. The trace connection $\nabla_{tr}$ can then be completed to a connection that ``sees'' $\widetilde{\nabla}^{\prime}$:
\begin{displaymath}
	\frac{\nabla_{int}\langle\ell,m\rangle}{\langle\ell,m\rangle}=\frac{\nabla_{tr}\langle\ell,m\rangle}{\langle\ell,m\rangle}
	+\frac{i}{2\pi}\pi_{\ast}\left(\frac{\widetilde{\nabla}^{\prime}m}{m}\wedge F_{\widetilde{\nabla}}\right),
\end{displaymath}
where $F_{\widetilde{\nabla}}$ is the curvature of the canonical extension $\widetilde{\nabla}$ on $\Lcal$. Assume now that $\widetilde{\nabla}^{\prime}$ is a canonical extension of a relative holomorphic connection $\nabla^{\prime}\colon\Mcal\to\Mcal\otimes\Omega_{\Xcal/S}^{1}$. Then the intersection connection is compatible with the obvious symmetry of the Deligne pairing. These constructions carry over to the case when the relative connections only have a smooth dependence on the horizontal directions, but are still holomorphic on fibers. The intersection connection reduces to the trace connection if $\widetilde{\nabla}^{\prime}$ is the Chern connection of a smooth hermitian metric on $\Mcal$, flat on fibers. Finally, the trace connection coincides with the Chern connection of the metrized Deligne pairing in case $\widetilde{\nabla}$ is a Chern connection, flat on fibers, as well.

%
 Let us denote $\lambda(\Lcal)$ for the determinant of the cohomology of $\Lcal$, that is
\begin{displaymath}
	\lambda(\Lcal)=\det R\pi_{\ast}(\Lcal).
\end{displaymath}
The determinant of $R\pi_{\ast}(\Lcal)$ makes sense, since it is a perfect complex and so the theory of Knudsen-Mumford \cite{KM} applies. It can be extended,  multiplicatively, to virtual objects, namely formal sums of line bundles with integer coefficients. Deligne \cite{Deligne:87} proves the existence of an isomorphism
\begin{displaymath}
	\Dcal\colon\lambda(\Lcal-\Ocal)^{\otimes 2}\isorightarrow\langle\Lcal,\Lcal\otimes_{\Xcal/S}\rangle,
\end{displaymath}
where $\omega_{\Xcal/S}$ is the relative cotangent bundle of $\pi$. The isomorphism is compatible with base change and is functorial in $\Lcal$. It is unique  with this properties, up to sign. It can be combined with Mumford's canonical (up to sign) and functorial isomorphism \cite{Mumford}, which in the language of Deligne's pairings reads
\begin{displaymath}
	\lambda(\Ocal)^{\otimes 12}\isorightarrow\langle\omega_{\Xcal/S},\omega_{\Xcal/S}\rangle.
\end{displaymath}
Hence, we have a canonical (up to sign) isomorphism
\begin{displaymath}
	\Dcal^{\prime}\colon\lambda(\Lcal)^{\otimes 12}\isorightarrow\langle\omega_{\Xcal/S},\omega_{\Xcal/S}\rangle\otimes\langle\Lcal,\Lcal\otimes_{\Xcal/S}\rangle^{\otimes 6},
\end{displaymath}
which is again compatible with base change and functorial in $\Lcal$. The latter is also usually called Deligne's isomorphism.

When the line bundles $\Lcal$ and $\omega_{\Xcal/S}$ are endowed with smooth hermitian metrics, all the line bundles on $S$ involved in Deligne's isomorphism inherit hermitian metrics. On the Deligne pairings, the construction is the metrized counterpart of the intersection connection definition, and it will not be recalled here. It amounts to the $\star$-product of Green currents introduced by Gillet-Soul\'e in arithmetic intersection theory. The determinant of cohomology can be equipped with the so-called Quillen metric, whose Chern connection is compatible with the Quillen connection of Bismut-Freed \cite{Bismut-Freed-1, Bismut-Freed-2}. The Deligne isomorphism is, up to an overall topological constant, an isometry for these metrics. The value of the constant can be pinned down, for instance by using the arithmetic Riemann-Roch theorem of Gillet-Soul\'e \cite{GS:ARR}. We refer the reader to the survey articles of Soul\'e \cite{Soule} and Bost \cite{Bost}, where all these constructions and facts are summarized. Because the Deligne isomorphism is an isometry in the metrized case, it is in particular parallel for the corresponding Chern connections.  

 One of the aims of \cite{FreixasWentworth:15} is to elucidate to what extent Deligne's isometry, and more precisely its Chern connection version, carries over in the case of relative, flat, compatible connections on $\Lcal$ that are not necessarily unitary for some hermitian structure. In \cite[Sec.\ 5]{FreixasWentworth:15}
  (see especially Theorem 5.10 and Remark 5.11 therein) we discuss and solve this question in the particular case of trivial fibrations. 
  The present article shows that this particular case actually implies the most general one. 
  More precisely, fix $X$ a compact Riemann surface, a base point $p\in X$, and a hermitian metric on 
   $\omega_{X}$. 
   As parameter space we take $S=M_{B}(X)$, the affine variety of characters $\chi\colon\pi_{1}(X,p)\to\CBbb^{\times}$. Let $\Xcal=X\times M_{B}(X)$, which is fibered over $M_{B}(X)$ by the second projection. Over $\Xcal$, there is a universal line bundle with relative connection, $(\Lcal,\nabla)$, whose holonomy at a given $\chi\in M_{B}(X)$ is $\chi$ itself. We also need to introduce the conjugate Riemann surface $\ov{X}$ (reverse the complex structure), with same base point and same character variety $M_{B}(X)$. We put $\Xcal^{c}=\ov{X}\times M_{B}(X)$. There is a universal line bundle with relative connection $(\Lcal^{c},\nabla^{c})$, whose holonomy at a given $\chi\in M_{B}(X)$ is now $\chi^{-1}$. Note that if $\chi$ is unitary, then $\Lcal^{c}_{\chi}$ is the holomorphic line bundle on $\ov{X}$ conjugate to $\Lcal_{\chi}$, but this is not the case for general $\chi$. By using a variant of the holomorphic analytic torsion introduced by Fay \cite{Fay}, later used by Hitchin \cite{Hitchin}, we endowed the product of determinants of cohomologies, $\lambda(\Lcal)\otimes_{\CBbb}\lambda(\Lcal^{c})$, with a holomorphic and flat connection on $M_{B}(X)$. We then showed that this Quillen type connection corresponds to the tensor product of intersection connections on the Deligne pairings, through the tensor product of Deligne's isomorphisms for $\Lcal$ and $\Lcal^{c}$. 
%

\section{Logarithms and Deligne Pairings}
\subsection{Logarithms and connections on holomorphic line bundles}
Let $S$ be a connected complex analytic manifold and $\Lcal\to S$ a $\Ccal^{\infty}_{S}$ complex line bundle. 
To simplify the presentation, the same notation will be used
 when $\Lcal$ is understood to have the structure of a holomorphic line bundle.  
 Also, no notational distinction will be made between a holomorphic line bundle and the associated invertible  sheaf of $\Ocal_S$ modules. Finally, 
denote by $\Lcal^{\times}$ the $\Gm$ torsor (or principal bundle) given by the complement of the zero section in the total space of $\Lcal$. 

 Here we introduce the notion of smooth logarithm for $\Lcal$.  For a holomorphic bundle there is also notion of holomorphic logarithm, and whenever we talk about holomorphic logarithm it will be implicit that $\Lcal$ has a holomorphic structure. 
The discussion of logarithms and connections is given in terms of $d\log$ deRham complexes. The reader will notice that this is an additive reformulation of the notion of trivialization (see Remark \ref{remark:why-logs} below).
\begin{definition}
A \emph{smooth (resp.\ holomorphic) logarithm} for $\Lcal$ is a map
\begin{displaymath}
	\LOG:\Lcal^{\times}\longrightarrow\mathbb{C}/2\pi i\, \mathbb{Z}
\end{displaymath}
satisfying: $\LOG(\lambda\cdot e)=\log\lambda+\LOG(e)$, for  $\lambda\in \Gm$ and $e\in\Lcal^{\times}$, and such that the well-defined $\mathbb{C}^{\times}$-valued function $\exp\circ\LOG$ is smooth (resp.\ holomorphic) with respect to the natural structure of smooth (resp. complex analytic) manifold on $\Lcal^{\times}$.
\end{definition}
\begin{remark}\label{remark:why-logs}
Clearly, a logarithm  is a reformulation  of the choice of a trivialization.
The reason for working with logarithms in this paper is to provide a simplification of some formulas, a direct relationship with connections, as well as a context that is well-suited for the arithmetic discussion later on. Indeed, in classical Arakelov geometry the corresponding avatar of a smooth $\LOG$  is the ordinary logarithm  of a smooth hermitian metric, and more generally the notion of Green current for a cycle.
\end{remark}
A logarithm $\LOG$ can be reduced modulo $\pi i\,  \ZBbb$. We will write $\overline{\LOG}$ for the reduction of $\LOG$. By construction, the reduction of a logarithm modulo $\pi i\,  \ZBbb$ factors through $\Lcal^{\times}/\lbrace\pm 1\rbrace$:
\begin{displaymath}
	\xymatrix{
		\Lcal^{\times}\ar@{->>}[d]\ar[r]^{\hspace{-0.5cm}\LOG}	&\CBbb/2\pi i\, \ZBbb\ar@{->>}[d]\\
		\Lcal^{\times}/\lbrace\pm 1\rbrace\ar[r]^{\hspace{-0.25cm} \overline{\LOG} }	&\CBbb/\pi i\,  \ZBbb\ .
	}
\end{displaymath}
Though perhaps not apparent at this moment, the necessity for this reduction will appear at several points below (notably because of the sign ambiguity in Deligne's isomorphism).

Given a smooth (resp.\ holomorphic) logarithm $\LOG$, we can locally lift it to a well defined $\mathbb{C}$-valued smooth (resp.\ holomorphic) function. Therefore, the differential  $d\LOG$ is a well-defined differential form on $\Lcal^\times$.
A smooth $\LOG$ is holomorphic exactly when $d\LOG$ is holomorphic.
%
We can attach to a smooth logarithm a smooth and flat connection $\nabla_{\LOG}$ on $\Lcal$, determined by the rule
\begin{equation}\label{eq:1}
	\frac{\nabla_{\LOG}e}{e}=e^\ast (d\LOG) ,
\end{equation}
where $e: S^\circ\subset S\to \Lcal^\times$ is a local frame. 
 If $\Lcal$ is holomorphic the connection
 is compatible exactly when $\LOG$ is holomorphic, as we immediately see by taking the $(0,1)$ part of \eqref{eq:1} for $e$ holomorphic. The existence of a smooth (resp.\ holomorphic) $\LOG$ on $\Lcal$ is related to the existence of a flat smooth (resp.\ holomorphic connection). 
 
 We will say that a connection $\nabla$ on $\Lcal$ is \emph{associated to a logarithm}, if $\nabla=\nabla_{\LOG}$ for some logarithm $\LOG$ on $\Lcal$. For the sake of clarity, it is worth elaborating on this notion from a cohomological point of view. Let us focus on the holomorphic case, which is the relevant one in the present work (the smooth case is dealt with similarly). Introduce the holomorphic log-deRham complex:
\begin{displaymath}
	\Omega_{S}^{\times}:\Ocal_{S}^{\times}\xrightarrow{\ d\log\ }\Omega_{S}^{1}\overset{d}{\longrightarrow}\Omega_{S}^{2}\longrightarrow\cdots
\end{displaymath}
There is an exact sequence of sheaves of abelian groups
\begin{displaymath}
	0\longrightarrow \tau_{\geq 1}\Omega_{S}^{\bullet}\longrightarrow\Omega_{S}^{\times}\longrightarrow\Ocal_{S}^{\times}\longrightarrow 1\ ,
\end{displaymath}
where $\Omega_{S}^{\bullet}$ is the holomorphic deRham complex and $\tau_{\geq i}$ stands for the filtration b\^ete  of a complex from degree $i$ on. From the hypercohomology long exact sequence, we derive a short exact sequence of groups
\begin{displaymath}
	H^{0}(S,\Omega_{S}^{1})\lra \mathbb{H}^{1}(S,\Omega_{S}^{\times})\lra\Pic(S)\ .
\end{displaymath}
The middle group $\mathbb{H}^{1}(S,\Omega_{S}^{\times})$ classifies isomorphism classes of holomorphic line bundles on $S$ with  flat holomorphic connections. The vector space $H^{0}(S,\Omega_{S}^{1})$ maps to the holomorphic connections on the trivial line bundle. The map to $\Pic(S)$ is just forgetting the connection. We will write by $[\Lcal,\nabla]$ the class in $\mathbb{H}^{1}(S,\Omega_{S}^{\times})$ of $\Lcal$ with a holomorphic connection $\nabla$.
\begin{proposition}
There exists a holomorphic $\LOG$ for $\Lcal$ if, and only if, there exists a holomorphic connection $\nabla$ on $\Lcal$ with $[\Lcal,\nabla]=0$. In this case, the connection $\nabla$ is associated to a $\LOG$.
\end{proposition}
\begin{proof}
We compute the hypercohomology group $\mathbb{H}^{1}(S,\Omega_{S}^{\times})$ with a \v{C}ech resolution. Let $\Ucal=\lbrace U_{i}\rbrace_i$ be an open covering of $S$ by suitable open subsets, and such that $\Lcal$ admits a local holomorphic trivialization $e_{i}$ on $U_{i}$. Elements of $\mathbb{H}^{1}(S,\Omega_{S}^{\times})$ can be represented by couples $(\lbrace\omega_{i}\rbrace,\lbrace f_{ij}\rbrace)$ in
\begin{displaymath}
	\Ccal^{0}(\Ucal,\Omega_{S}^{1})\oplus\Ccal^{1}(\Ucal,\Ocal_{S}^{\times}),
\end{displaymath}
subject to the cocycle relation
\begin{displaymath}
	d\omega_{i}=0,\quad \omega_{i}-\omega_{j}=d\log (f_{ij}),\quad f_{ij}f_{jk}f_{ki}=1.
\end{displaymath}
Coboundaries are of the form
\begin{displaymath}
	\omega_{i}=d\log(f_i),\quad f_{ij}=f_{i}/f_{j}.
\end{displaymath}
Let $\LOG$ be a holomorphic logarithm on $\Lcal$. Then the attached flat connection $\nabla_{\LOG}$ has trivial class. Indeed, we put 
\begin{displaymath}
	f_{i}=\exp\LOG(e_i),\quad\omega_{i}=\frac{\nabla_{\LOG}e_i}{e_i}=d\log(f_i).
\end{displaymath}
Conversely, let $\nabla$ be a holomorphic connection on $\Lcal$ with vanishing class. We put
\begin{displaymath}
	\omega_{i}=\frac{\nabla e_{i}}{e_i},\quad e_{i}=f_{ij}e_{j}.
\end{displaymath}
The cocycle $(\lbrace\omega_{i}\rbrace,\lbrace f_{ij}\rbrace)$ is trivial. We can thus find units $f_{i}\in\Gamma(U_i,\Ocal_{S}^{\times})$ with
\begin{displaymath}
	\omega_{i}=d\log(f_i),\quad f_{ij}=f_i/f_j.
\end{displaymath}
Then, we can define a holomorphic logarithm $\LOG$ by imposing
\begin{displaymath}
	\LOG(e_{i})=\log(f_i)\mod 2\pi i\, \mathbb{Z},
\end{displaymath}
and extending trivially under the $\Gm$ action. For $\LOG$ to be well-defined, it is enough to observe that on overlaps we have, by definition
\begin{displaymath}
	\LOG(e_i)=\log(f_{ij})+\LOG(e_j),
\end{displaymath}
which is compatible with the $\Gm$ action. By construction, $\nabla=\nabla_{\LOG}$. As a complement, notice that if $\tilde{f}_i$ is another choice of functions, then necessarily $\tilde{f}_i=\lambda_i f_i$, for some nonvanishing constant $\lambda_i$. Moreover, $\lambda_{i}=\lambda_{j}$ because of the condition $\tilde{f}_i/\tilde{f}_j=f_{ij}=f_{i}/f_{j}$. Therefore, the change in $\LOG$ is just by a constant, as was to be expected.
\end{proof}
Actually, the proof of the proposition also gives:
\begin{corollary}
A holomorphic connection $\nabla$ on $\Lcal$ is associated to a $\LOG$ if, and only if, $[\Lcal,\nabla]=0$. In this case, the associated $\LOG$ is unique up to a constant.
\end{corollary}

\subsection{Construction of naive logarithms}\label{section:LOG}
Let $\pi:\Xcal\rightarrow S$ be a smooth proper morphism of smooth quasi-projective complex varieties with connected fibers of relative dimension one. We assume given a fixed  section $\sigma:S\rightarrow\Xcal$ and $\Lcal$, $\Mcal$ holomorphic line bundles on $\Xcal$. We require $\Lcal$ comes with a rigidification (\emph{i.e.} a choice of trivialization) along $\sigma$. We consider relative connections
\begin{displaymath}
 	\nabla_{\Xcal/S}^{\Lcal}:\Lcal\rightarrow\Lcal\otimes\Acal^{1}_{\Xcal/S},\quad \nabla_{\Xcal/S}^{\Mcal}:\Mcal\rightarrow\Mcal\otimes\Acal^{1}_{\Xcal/S},
\end{displaymath}	
compatible with the holomorphic structures. Hence, the $(0,1)$ projection $(\nabla_{\Xcal/S}^{\Lcal})^{(0,1)}=\overline{\partial}_{\Lcal,}$, the relative Dolbeault operator on $\Lcal$, and similarly for $\Mcal$. We suppose that $\nabla_{\Xcal/S}^{\Lcal}$ is flat, but make no assumption on $\nabla_{\Xcal/S}^{\Mcal}$ for the time being. The connection on $\Lcal$ can be thought as a \emph{smooth family} (with respect to $S$) of holomorphic connections on $\Lcal$ restricted to fibers. Below we use this data to construct a smooth logarithm map on the Deligne pairing of $\Lcal$ and $\Mcal$:
\begin{displaymath}
	\LOG_{na}:\langle\Lcal,\Mcal\rangle^{\times}\longrightarrow\mathbb{C}/2\pi i\, \mathbb{Z}.
\end{displaymath}
and we compute its associated connection. This logarithm is defined so as to give a direct relationship with the intersection connection on $\langle\Lcal,\Mcal\rangle$. We anticipate, however,  some problems with this construction:
\begin{enumerate}
	\item it is only defined locally on contractible open subsets of $S$;
	\item it depends on auxiliary data which  prevents  an extension to the whole of $S$;
	\item while it depends on the connection $\nabla^{\Mcal}_{\Xcal/S}$, it nearly depends only on the holomorphic structure of $\Lcal$ on fibers (see Remark \ref{remark:dependence-holom} below for the precise meaning of this assertion) -- in particular, it cannot be compatible with the symmetry of Deligne pairings. 
\end{enumerate}
For these various reasons, we shall call it a \emph{naive logarithm}. 

Let $\nu_{\Lcal}:S\rightarrow H^{1}_{dR}(\Xcal/S)/R^{1}\pi_{\ast}(2\pi i\, \mathbb{Z})$ be the smooth classifying map of $(\Lcal,\nabla_{\Xcal/S}^{\Lcal})$. This map does not depend on the rigidification. Locally on contractible open subsets $S^{\circ}$ of $S$, we can lift $\nu_{\Lcal}$ to a smooth section of $H^{1}_{dR}(\Xcal/S)$, that we write $\tilde{\nu}$. We work over a fixed $S^{\circ}$ and make a choice of lifting $\tilde{\nu}$. We take the universal cover $\widetilde{\Xcal}\rightarrow\Xcal\mid_{S^{\circ}}$. Let $\ell$, $m$ be meromorphic sections of $\Lcal$ and $\Mcal$, whose divisors are finite and \'etale over $S^{\circ}$ (finite, flat and unramified). Using the rigidification $\sigma$ and $\nabla_{\Xcal/S}^{\Lcal}$ and a local lifting $\tilde{\sigma}$ to $\widetilde{\Xcal}$, the section $\ell$ and can be uniquely lifted to $\widetilde{\Xcal}$, as a meromorphic function on fibers, transforming under some character under the action of the fundamental group (the character depends on the fiber) and taking the value $1$ along $\tilde{\sigma}$. We denote this lift $\tilde{\ell}$. Precisely, if $\gamma\in\pi_{1}(\Xcal_s,\sigma(s))$, the transformation of $\tilde{\ell}$ on $\widetilde{\Xcal}_{s}$ with respect to translation by $\gamma$ is
\begin{equation}\label{eq:tr-law}
	\tilde{\ell}(\gamma z)=\exp\left(\int_{\gamma}\tilde{\nu}\mid_{\Xcal_s}\right)\tilde{\ell}(z),\quad z\in\widetilde{\Xcal}_s.
\end{equation}
Notice that the dependence of $\tilde{\ell}$ relative to the base $S^{\circ}$ is only $\Ccal^{\infty}$, because the connections were only assumed to depend smoothly on the horizontal directions. We declare
\begin{equation}\label{eq:7}
	\LOG_{na}(\langle\ell,m\rangle)=\log(\tilde{\ell}(\widetilde{\Div m}))-\int_{\tilde{\sigma}}^{\widetilde{\Div m}}\tilde{\nu}
	-\frac{i}{2\pi}\pi_{\ast}(\frac{\nabla m}{m}\wedge\tilde{\nu})\mod 2\pi i\, \mathbb{Z}\ .
\end{equation}
The index $\emph{na}$ stands for \emph{naive}. Let us clarify the construction:
\begin{enumerate}
\item The integrals are computed fiberwise. We notice that in the last term, there is no need for a global extension of the connection $\nabla_{\Xcal/S}^{\Mcal}$.
\item $\tilde{\nu}$ is to be understood as a differential form $\eta(z,s)$ on $\Xcal$, which is harmonic for fixed $s$ and represents $\tilde{\nu}$ fiberwise.\footnote{The construction of $\eta$ requires the intermediate choice of a metric on $T\Xcal\mid_{S^{\circ}}$ and the use of elliptic theory applied to Kodaira laplacians, with smooth dependence on a parameter. The vertical projection of $\eta$ does not depend on the choice of metric. This is particular to the case of curves.} There is an ambiguity in this representative: it is only unique up to $\pi^{-1}\Acal^{1}_{S^{\circ}}$. This does not affect the integrals, since they are computed on fibers. Therefore, we can choose to rigidify $\eta(z,s)$ by imposing it vanishes along the section $\sigma$.

\item the notation $\widetilde{\Div m}$ indicates a lift of $\Div m$ to the universal cover. Hence, if on a given fiber $\Xcal_{s}$ we have $\Div m=\sum_i n_i P_i$ (finite sum), then $\widetilde{\Div m}=\sum_i n_i\tilde{P}_i$, where the $\tilde{P}_i$ are choices of preimages of $P_i$ in the universal cover $\widetilde{\Xcal}_{s}$. With this understood, the first two terms in the definition of $\LOG_{na}$ expand to
\begin{displaymath}
	\sum_{i}n_{i}\log(\tilde{\ell}(\tilde{P}_{i}))-\sum_{i}n_{i}\int_{\tilde{\sigma}}^{\tilde{P}_i}\tilde{\nu}.
\end{displaymath}
The integration path from $\tilde{\sigma}$ to $\tilde{P}_{i}$ is taken in $\widetilde{\Xcal}_{s}$. This expression does not depend on the choice of liftings $\tilde{\sigma}$ and $\tilde{P}_i$, modulo $2\pi i\, \ZBbb$. For instance, if $P$ and $\gamma P$ are points in $\widetilde{\Xcal}_{s}$ differing by the action of $\gamma\in\pi_{1}(\Xcal_{s},\sigma(s))$, then
\begin{displaymath}
	\begin{split}
		\log(\tilde{\ell}(\gamma P))-\int_{\tilde{\sigma}(s)}^{\gamma P}\tilde{\nu}&=\int_{\gamma}\tilde{\nu}+\log(\tilde{\ell}(P))
		-\int_{\tilde{\sigma}(s)}^{P}\tilde{\nu}-\int_{P}^{\gamma P}\tilde{\nu} \mod 2\pi i\, \mathbb{Z}\\
		&=\log(\tilde{\ell}(P))-\int_{\tilde{\sigma}(s)}^{P}\tilde{\nu}\mod 2\pi i\, \mathbb{Z}\ .
	\end{split}
\end{displaymath}
And if we change the lifting $\tilde{\sigma}$ to $\sigma^{\ast}=\gamma\tilde{\sigma}$, then the new lifting of $\ell$ is $\ell^{\ast}$ with
\begin{displaymath}
	\ell^{\ast}(z)=\exp\left(-\int_{\gamma}\tilde{\nu}\right)\tilde{\ell}(z),
\end{displaymath}
and from this relation it follows the independence of the lift $\tilde{\sigma}$ modulo $2\pi i\, \ZBbb$. 
\item There are several facts that can be checked similarly to our previous work \cite[Sec. 3 and 4]{FreixasWentworth:15}. For instance, the compatibility to the relations defining the Deligne pairing, most notably under the change $f\mapsto fm$ ($f$ a rational function), follows from various reciprocity laws for differential forms, plus the observation
\begin{displaymath}
	\pi_{\ast}\left(\frac{df}{f}\wedge\tilde{\nu}\right)=\partial\pi_{\ast}(\log|f|^{2}\cdot\tilde{\nu})=0.
\end{displaymath}
Here we used that $\tilde{\nu}$ as above is fiberwise $\partial$-closed (by harmonicity), and also that $\tilde{\nu}$ is a 1-form while $\pi$ reduces types by $(1,1)$. 
\end{enumerate}
With this understood, we conclude that $\LOG_{na}$ is a smooth logarithm for $\langle\Lcal,\Mcal\rangle_{\mid S^{\circ}}$. Let us start exploring the dependence of the naive logarithm on the connections.
\begin{lemma}\label{lemma:change-nu-theta}
Let $\theta$ be a differential 1-form on $\Xcal_{\mid S^{\circ}}$, holomorphic on fibers. Assume that $\nabla_{\Xcal/S}^{\Mcal}$ is either flat or the relative Chern connection of a smooth hermitian metric on $\Mcal$. Then the definition of $\LOG_{na}$ is invariant under the change $\nabla_{\Xcal/S}^{\Lcal}\mapsto\nabla_{\Xcal/S}^{\Lcal}+\theta$.
\end{lemma}
\begin{proof}
Notice the change of connections translates into changing $\tilde{\nu}$ by $\tilde{\nu}+\theta$. For the new connection, the lift $\tilde{\ell}'$ compares to $\tilde{\ell}$ by
\begin{displaymath}
	\tilde{\ell}'(z)=\exp\left(\int_{\tilde{\sigma}}^{z}\theta\right)\tilde{\ell}(z),\quad z\in\widetilde{\Xcal}_{s}.
\end{displaymath}
Therefore
\begin{displaymath}
	\log(\tilde{\ell}'(z))-\int_{\tilde{\sigma}(s)}^{z}(\tilde{\nu}+\theta)
	=\log(\tilde{\ell}(z))-\int_{\tilde{\sigma}(s)}^{z}\tilde{\nu}\mod 2\pi i\, \ZBbb.
\end{displaymath}
This settles the first two terms. For the last term in $\LOG_{na}$, we first suppose $\nabla_{\Xcal/S}^{\Mcal}$ is flat. Hence it is holomorphic on fibers, and we have
\begin{displaymath}
	\int_{\Xcal_{s}}\frac{\nabla m}{m}\wedge(\tilde{\nu}+\theta)=\int_{\Xcal_{s}}\frac{\nabla m}{m}\wedge\tilde{\nu},
\end{displaymath}
for type reasons: both $\nabla m/m$ and $\theta$ are of type $(1,0)$ on fibers. If $\nabla_{\Xcal/S}^{\Mcal}$ is the Chern connection of a smooth hermitian metric $\|\cdot\|$ on $\Mcal$, then the last term actually vanishes! Indeed,
\begin{align*}
	\int_{\Xcal_{s}}\frac{\nabla m}{m}\wedge\tilde{\nu}&=\int_{\Xcal_{s}}\partial\log\|m\|^{2}\wedge\widetilde{\nu} 
	=\int_{\Xcal_{s}}\partial\log\|m\|^{2}\wedge\widetilde{\nu}'' \\
	&=\int_{\Xcal_{s}}\partial(\log\|m\|^{2}\cdot\widetilde{\nu}'') =\int_{\Xcal_{s}}d(\log\|m\|^{2}\cdot\widetilde{\nu}'') \\
	&=0\ .
\end{align*}
We again  used that the vertical representatives of $\tilde{\nu}$ are harmonic and that the singular differential form $\log\|m\|^{2}\cdot\widetilde{\nu}''$ has no residues on $\Xcal_s$. This concludes the proof.
\end{proof}
\begin{remark}\label{remark:dependence-holom}
The content of the lemma is that for these particular connections on $\Mcal$, $\LOG_{na}$  nearly depends only on the (relative) Chern connection on $\Lcal$. The only subtle point is that for this to be entirely true, we would need the invariance of $\LOG_{na}$ under the additional transformation $\tilde{\nu}\mapsto\tilde{\nu}+\theta$, for $\theta$ a horizontal section of $R^{1}\pi_{\ast}(2\pi i\, \ZBbb)\mid_{S^{\circ}}$. This is however not the case! This issue will be addressed by considering the conjugate family at the same time. The resulting logarithm will then depend on the full connection $\nabla_{\Xcal/S}^{\Lcal}$.
\end{remark}
Let us now focus on the case when $\Mcal$ is endowed with a Chern connection. 
\begin{lemma}\label{lemma:chern-on-M}
Assume $\nabla^{\Mcal}_{\Xcal/S}$ is the relative Chern connection of a smooth hermitian metric on $\Mcal$. Endow the line bundle $\Mcal\otimes\Ocal(-(\deg\Mcal)\sigma)$, of relative degree 0, with a relative flat unitary connection. Finally, equip $\sigma^{\ast}(\Lcal)$ with the holomorphic logarithm induced by the rigidification $\sigma^{\ast}(\Lcal)\isorightarrow\Ocal_{S}$. Then, the isomorphism of Deligne pairings
\begin{displaymath}
	\langle\Lcal,\Mcal\rangle\isorightarrow\langle\Lcal,\Mcal\otimes\Ocal(-(\deg\Mcal)\sigma)\rangle\otimes\sigma^{\ast}(\Lcal)^{\otimes\deg\Mcal}
\end{displaymath}
is compatible with the respective logarithms. In particular, the naive logarithm on $\langle\Lcal,\Mcal\rangle$ does not depend on the particular choice of Chern connection $\nabla^{\Mcal}_{\Xcal/S}$.
\end{lemma}
\begin{proof}
For the Deligne pairing on the left hand side, as we already saw in the proof of Lemma \ref{lemma:change-nu-theta}, we have
\begin{displaymath}
	\LOG_{na}(\langle\ell,m\rangle)=\log(\tilde{\ell}(\widetilde{\Div m}))-\int_{\tilde{\sigma}}^{\widetilde{\Div m}}\tilde{\nu},
\end{displaymath}
because $\nabla^{\Mcal}_{\Xcal/S}$ is a Chern connection. Assume now that $\ell$ does not have a pole or a zero along $\sigma$. Then, by the very construction of $\tilde{\ell}$, we have on the one hand
\begin{displaymath}
	\log(\tilde{\ell}(\widetilde{\Div m}))=\log(\tilde{\ell}(\widetilde{\Div m}-(\deg\Mcal)\tilde{\sigma}))+(\deg\Mcal)\log(\sigma^{\ast}\ell),
\end{displaymath}
while on the other hand it is obvious that
\begin{displaymath}
	\int_{\tilde{\sigma}}^{\widetilde{\Div m}}\tilde{\nu}=\int_{\tilde{\sigma}}^{\widetilde{\Div m}-(\deg\Mcal)\tilde{\sigma}}\tilde{\nu}.
\end{displaymath}
The lemma follows from these observations.
\end{proof}

\subsection{The connection attached to a naive logarithm}
We maintain the notations so far. We wish to compute the connection associated to $\LOG_{na}$, this is $d\LOG_{na}$. This requires differentiation of functions on $S^{\circ}$ of the form
\begin{displaymath}
	\int_{\tilde{\sigma}}^{\widetilde{\Div m}}\tilde{\nu},\quad\pi_{\ast}\left(\frac{\nabla m}{m}\wedge\tilde{\nu}\right).
\end{displaymath}
For instance, in the first integral we have to deal with the horizontal variation of both $\widetilde{\Div m}$ and $\tilde{\nu}$ (we are allowed to suppose that $\widetilde{\Div m}$ is given by sections, after possibly changing $S^{\circ}$ by some open cover). The path of integration, in fibers, from $\tilde{\sigma}$ to $\widetilde{\Div m}$ can be seen as a smooth family of currents on fibers. We need to explain how to differentiate these. For lack of an appropriate reference, we elaborate on this question below.

\subsubsection{Families of currents on cohomology classes and differentiation} A family of currents on $S$ relative to a smooth and proper morphism $\pi\colon\Xcal\to S$, of degree $d$, is a section of the sheaf of $\Ccal^{\infty}_{S}$-modules $\pi_{\ast}(\Dcal_{\Xcal/\Scal}^{d})$, where $\Dcal_{\Xcal/\Scal}^{d}$ is the (sheafified) $\Ccal^{\infty}_{S}$-linear topological dual of the sheaf $\Acal^{n-d}_{\Xcal/S,0}$ of smooth relative differentials with compact support. In other words, by definition, $\Dcal_{\Xcal/\Scal}^{d}$ is the subsheaf of the sheaf of currents $\Dcal_{\Xcal}^{d}$ on $\Xcal$ which are $\pi^{-1}\Ccal^{\infty}_{S}$-linear (compatibly with multiplication of currents by smooth functions) and vanish on $(\Acal_{\Xcal}^{p}\wedge\pi^{\ast}\Acal^{n-d-p}_{S})\cap\Acal^{n-d}_{\Xcal,0}$ for $0\leq p<n-d$ (the index $0$ indicates compact support). The space of sections of $\Dcal_{\Xcal/S}^{d}(U)$ over an open $U\subset\Xcal$ is a closed subspace of $\Dcal_{\Xcal}^{d}(U)$. Because $\pi$ is proper, we have a pairing
\begin{displaymath}
	\pi_{\ast}(\Dcal_{\Xcal/\Scal}^{d})\times\pi_{\ast}(\Acal^{n-d}_{\Xcal/S,0})\longrightarrow\Ccal^{\infty}_{S,0}.
\end{displaymath}
Write $T_{\bullet}\in\pi_{\ast}(\Dcal_{\Xcal/\Scal}^{d})(U)$ for a smooth family of currents, viewed as an association $s\mapsto T_{s}$, where $T_{s}$ is a current on $\Xcal_{s}$, smoothly depending on $s\in U$. This can be rigorously formulated as follows. Let $\theta$ be a degree $n-d$ differential form on a given fiber $\Xcal_{s_{0}}$. Because $\pi$ is submersive, it can be trivialized in a neighborhood of $s_{0}$. Let $\varphi$ be a smooth function on $S$, with compact support in a neighborhood of $s_{0}$, and taking the value 1 in a neighborhood of $s_{0}$. Using the trivialization of $\pi$ and the function $\varphi$, one extends $\theta$ to a compactly supported form $\widetilde{\theta}$ on $\Xcal$, with support in a neighborhood of the fiber $\Xcal_{s_{0}}$. Then we put $T_{s_{0}}(\theta)=T_{\bullet}(\widetilde{\theta})(s_{0})$. The construction does not depend on any choices. Indeed, let $\widetilde{\theta}$ and $\widetilde{\theta}'$ be two such extensions, depending on local trivializations and choices of compactly supported functions $\varphi$ and $\varphi'$ on $S$, as before. Then, one can write $\widetilde{\theta}=\widetilde{\theta}'+\sum_{i=1}^{\dim S}\rho_{i}\omega_{i}$, where the $\rho_{i}$ are smooth functions on $S$, vanishing at $s_{0}$, and the $\omega_{i}$ are smooth differential forms on $\Xcal$ with compact supports. But because $T_{\bullet}$ is $\Ccal^{\infty}_{S}$-linear, compatible with multiplication of currents by smooth functions, we have $T_{\bullet}(\rho_{i}\omega_{i})=\rho_{i} T_{\bullet}(\omega_{i})$, which vanishes at $s_{0}=0$. Thus, $T_{\bullet}(\widetilde{\theta})(s_{0})=T_{\bullet}(\widetilde{\theta}')(s_{0})$. Furthermore, if $s\mapsto \theta_{s}$ is a smooth family of differential forms on fibers $\Xcal_{s}$, then $s\mapsto T_{s}(\theta_{s})$ is a smooth function. Both points of view, the sheaf theoretic one and $s\mapsto T_{s}$, are easily seen to be equivalent. We will confuse them from now on.

A smooth family of currents can be differentiated with respect to the parameter space $S$. It gives raise to a smooth family of currents with values in 1-differential forms, given by a pairing
\begin{displaymath}
	\pi_{\ast}(\Dcal_{\Xcal/\Scal}^{d})\times\pi_{\ast}(\Acal^{n-d}_{\Xcal/S,0})\longrightarrow\Acal^{1}_{S,0}.
\end{displaymath}
Locally on $S$, we can trivialize $\Acal^{1}_{S}$ as a sheaf of $\Ccal^{\infty}_{S}$-modules (by taking a basis of smooth vector fields) and see such objects as vectors of smooth families of currents. This is legitimate, since our relative currents are $\Ccal^{\infty}_{S}$-linear, and any two (local) basis of vector fields on $S$ differ by a matrix of $\Ccal^{\infty}_{S}$ coefficients. We can now iterate this procedure, and talk about families of currents with values in differential forms of any degree, and differentiate them. The differential of a family of currents $T_{\bullet}$ is denoted $d_{S}T_{\bullet}$. One checks $d_{S}^{2}=0$.  

Assume now that the morphism $\pi$ is of relative dimension 1, as is the case in this article. Then we can extend families of currents to relative cohomology classes. We begin with $T_{\bullet}$ a smooth family of currents of degree $n-1$, with values in differential forms of degree $d$. To simplifiy the notations, we assume $T_{\bullet}$ is defined over the whole $S$. Let $\theta$ be a smooth section of $H^{1}_{dR}(\Xcal/S)$. We define a differential form $T_{\bullet}(\theta)$ on $S$, by using harmonic representatives: relative to a contractible $S^{\circ}$, we represent $\theta$ by a smooth family of differential forms $\eta(z,s)$ on fibers $\Xcal_s$, which are harmonic for fixed $s\in S^{\circ}$. Then, on $S^{\circ}$ we put
\begin{displaymath}
	T_{\bullet}(\theta)\mid_{S^{\circ}}=(T_{\bullet})\mid_{S^{\circ}}(\eta)\in\Acal^{d}_{S}(S_{0}).
\end{displaymath}
Here the construction is best understood in the interpretation $s\mapsto T_{s}$ of smooth families of currents. Because the harmonic representative $\eta$ is unique modulo $\pi^{-1}\Acal^{1}_{S^{\circ}}$ and $T_{\bullet}$ is a relative current, the expression $T_{\bullet}(\theta)\mid_{S^{\circ}}$ is well defined and can be globalized to the whole $S$. We write the resulting differential form $T_{\bullet}(\theta)$. 

Let $\nabla_{\GM}\colon H^{1}_{dR}(\Xcal/S)\rightarrow H^{1}_{dR}(\Xcal/S)\otimes\Acal^{1}_{S}$ be the Gauss-Manin connection. With the previous notation for $T_{\bullet}$ and $\theta$, we can also define $T_{\bullet}(\nabla_{\GM}\theta)$, by the following prescription. On contractible $S^{\circ}$ we write
\begin{displaymath}
	\nabla_{\GM}\theta=\sum_{i}\theta_{i}\otimes \beta_i,
\end{displaymath}
where the $\theta_i$ are flat sections of $H^{1}(\Xcal/S)\bigr|_{S^{\circ}}$, and the $\beta_i$ are smooth 1-forms on $S^{\circ}$. Then, because $T_\bullet$ is $\Ccal_{S}^{\infty}$-linear, the expression
\begin{displaymath}
	T_{\bullet}(\nabla_{\GM}\theta):=\sum_{i}T_{\bullet}(\theta_i)\wedge \beta_i.
\end{displaymath}
 is independent of choices made.  Hence, it is well-defined and extends to $S$.

With these conventions, the following differentiation rule is easily checked:
\begin{equation}\label{eq:current_diff_2}
	d_{S}(T_{\bullet}(\theta))=(d_{S} T_{\bullet})(\theta)+(-1)^{d}T_{\bullet}(\nabla_{\GM}\theta).
\end{equation}
To do so, one computes the Gauss-Manin connection by locally trivializing the family and applying Stokes' theorem (this is the so-called Cartan-Lie formula \cite[Sec. 9.2.2]{Voisin}). Alternatively, the equation is an easy consequence of the construction of the canonical extension of a relative flat connection in \cite[Sec. 4]{FreixasWentworth:15}, in this case applied to the relative connection determined by $\theta$ (an auxiliary local choice of a section of $\pi$ is needed). Notice that the rule is indeed compatible with the expected property $d_{S}^2=0$ on families of currents, as we see by applying $d_{S}$ to \eqref{eq:current_diff_2} and recalling that $\nabla_{\GM}^2=0$. 
The main examples of currents that will fit into this framework are currents of integration against differential forms or families of paths, as in the proposition below.

\subsubsection{Differentiation of naive logarithms} The next statement provides an illustration of differentiation of currents on cohomology classes in the context of naive logarithms.

\begin{proposition}\label{prop:LOG-int}
Let $(\Lcal,\nabla_{\Xcal/S}^{\Lcal})$, $(\Mcal,\nabla_{\Xcal/S}^{\Mcal})$, $\tilde{\nu}$ be as above. Suppose given a global extension $\nabla^{\Mcal}:\Mcal\rightarrow\Mcal\otimes\Acal^{1}_{\Xcal/\CBbb}$ of $\nabla_{\Xcal/S}^{\Mcal}$, compatible with the holomorphic structure. Then
\begin{displaymath}
	d\LOG_{na}\langle\ell,m\rangle=\frac{\nabla_{\langle\Lcal,\Mcal\rangle}^{int}\langle\ell,m\rangle}{\langle\ell,m\rangle}
	-\frac{i}{2\pi}\pi_{\ast}(F_{\nabla^{\Mcal}}\wedge\tilde{\nu}),
\end{displaymath}
where $F_{\nabla^{\Mcal}}$ is the curvature of $\nabla^{\Mcal}$.
\end{proposition}
\begin{proof}
We collect the following identities. First, since we suppose that $\widetilde{\nu}$ is rigidified, \emph{i.e.} vanishes, along $\sigma$, the differentiation law \eqref{eq:current_diff_2} gives
\begin{equation}\label{eq:2}
	d\int_{\tilde{\sigma}}^{\widetilde{\Div m}}\tilde{\nu}=\tr_{\Div m/S^{\circ}}(\tilde{\nu})+\int_{\tilde{\sigma}}^{\widetilde{\Div m}}\nabla_{\GM}\nu_{\Lcal}.
\end{equation}
This still holds even for nonrigidified $\tilde{\nu}$. Similarly, we have
\begin{equation}\label{eq:3}
	\frac{i}{2\pi}d\pi_{\ast}\left(\frac{\nabla m}{m}\wedge\tilde{\nu}\right)=\frac{i}{2\pi}\pi_{\ast}(F_{\nabla^{\Mcal}}\wedge\tilde{\nu})
	-\tr_{\Div m/S^{\circ}}(\tilde{\nu})+\frac{i}{2\pi}\pi_{\ast}(\frac{\nabla m}{m}\wedge\nabla_{\GM}\nu_{\Lcal}),
\end{equation}
where we used the Poincar\'e-Lelong equation of currents
\begin{displaymath}
	\frac{i}{2\pi}d\left[\frac{\nabla m}{m}\right]+\delta_{\Div m}=\frac{i}{2\pi}F_{\nabla^{\Mcal}}.
\end{displaymath}
Now we observe that, in terms of the curvature of the canonical extension $\nabla^{\Lcal}$ of $\nabla_{\Xcal/S}^{\Lcal}$, we have
\begin{equation}\label{eq:4}
	\frac{i}{2\pi}\pi_{\ast}(\frac{\nabla m}{m}\wedge\nabla_{\GM}\nu_{\Lcal})
	=-\frac{i}{2\pi}\pi_{\ast}\left(\frac{\nabla m}{m}\wedge F_{\nabla^{\Lcal}}\right).
\end{equation}
We recall the definition of the intersection connection:
\begin{equation}\label{eq:5}
	\frac{\nabla_{\langle\Lcal,\Mcal\rangle}^{int}\langle\ell,m\rangle}{\langle\ell,m\rangle}=
	\frac{i}{2\pi}\pi_{\ast}\left(\frac{\nabla m}{m}\wedge F_{\nabla^{\Lcal}}\right)
	+\tr_{\Div m/S^{\circ}}\left(\frac{\nabla^{\Lcal}\ell}{\ell}\right)
\end{equation}
and that by the very definition of the canonical extension $\nabla^{\Lcal}$
\begin{equation}\label{eq:6}
	\tr_{\Div m/S^{\circ}}\left(\frac{\nabla^{\Lcal}\ell}{\ell}\right)=\tr_{\widetilde{\Div m}/S^{\circ}}\left(\frac{d\tilde{\ell}}{\tilde{\ell}}\right)
	-\int_{\tilde{\sigma}}^{\widetilde{\Div m}}\nabla_{\GM}\nu_{\Lcal}.
\end{equation}
Putting together equations \eqref{eq:2}--\eqref{eq:4}, and taking into account the defining equations \eqref{eq:5}--\eqref{eq:6}, we conclude with the assertion of the theorem.

\end{proof}
\begin{remark}
Recall that the curvature of the intersection connection above in terms of the curvatures of $\nabla^{\Lcal}$ (the canonical extension of $\nabla_{\Xcal/S}^{\Lcal}$) and $\nabla^{\Mcal}$ is given by
\begin{displaymath}
	F_{\langle\Lcal,\Mcal\rangle^{int}}=\pi_{\ast}(F_{\nabla^{\Lcal}}\wedge F_{\nabla^{\Mcal}}).
\end{displaymath}
This was proven in \cite[Prop. 3.16]{FreixasWentworth:15}. This formula is consistent with:
\begin{displaymath}
	\begin{split}
		d\frac{\nabla_{\langle\Lcal,\Mcal\rangle}^{int}\langle\ell,m\rangle}{\langle\ell,m\rangle}&=
	d^{2}\LOG_{na}\langle\ell,m\rangle
	+\frac{i}{2\pi}d\pi_{\ast}(F_{\nabla^{\Mcal}}\wedge\tilde{\nu})\\
	&=0-\frac{i}{2\pi}\pi_{\ast}\left(F_{\nabla^{\Mcal}}\wedge\nabla_{\GM}\nu_{\Lcal}\right)\\
	&=\frac{i}{2\pi}\pi_{\ast}\left(F_{\nabla^{\Lcal}}\wedge F_{\nabla^{\Mcal}}\right).
\end{split}
\end{displaymath}	
Observe the sign in the second line, according to \eqref{eq:current_diff_2}.
\end{remark}

\subsection{Dependence of naive logarithms on liftings}
Continuing with the notation of Section \ref{section:LOG}, we now study the dependence of the construction of $\LOG_{na}$ for Deligne pairings $\langle\Lcal,\Mcal\rangle$ on the lifting $\tilde{\nu}$ of $\nu_{\Lcal}$. 
Let $\theta$ be a flat section of $R^{1}\pi_{\ast}(2\pi i\, \mathbb{Z})$ on the contractible open subset $S^{\circ}$ of $S$. We think of $\theta$ as a smooth family of cohomology classes with periods in $2\pi i\, \mathbb{Z}$. As usual, on fibers we will confuse the notation $\theta$ with its harmonic representative. We wish to study the change of $\LOG_{na}$ under the transformation $\tilde{\nu}\mapsto\tilde{\nu}+\theta$. A first remark is that given a meromorphic section $\ell$ of $\Lcal$, the lifting $\tilde{\ell}$ does not depend on the choice of $\theta$. Therefore, we are led to study the change of the expression
\begin{displaymath}
	\int_{\tilde{\sigma}}^{\widetilde{\Div m}}\tilde{\nu}+\frac{i}{2\pi}\pi_{\ast}\left(\frac{\nabla m}{m}\wedge\tilde{\nu}\right);
\end{displaymath}
that is, the factor
\begin{equation}\label{eq:100}
	\int_{\tilde{\sigma}}^{\widetilde{\Div m}}\theta+\frac{i}{2\pi}\pi_{\ast}\left(\frac{\nabla m}{m}\wedge\theta\right).
\end{equation}
Observe that a change of representatives in $\tilde{\sigma}$ or $\widetilde{\Div m}$ does not affect this factor modulo $2\pi i\, \ZBbb$, because $\theta$ has periods in $2\pi i\, \ZBbb$. Since the expression is a function on $S^{\circ}$, we can reduce to the case when the base $S$ is a point, and thus work over a single Riemann surface $X$. 

There is no general answer for the question  posed above unless we make some additional assumptions on $\nabla^{\Mcal}$. The first case to consider is when $\nabla^{\Mcal}$ is the Chern connection of a smooth hermitian metric on $\Mcal$. Then, we already saw during the proof of Lemma \ref{lemma:change-nu-theta} that
\begin{displaymath}
	\int_{X}\frac{\nabla m}{m}\wedge\tilde{\nu}=\int_{X}d(\log\|m\|^{2}\tilde{\nu})=0.
\end{displaymath}
We are then reduced to
\begin{displaymath}
	\int_{\tilde{\sigma}}^{\widetilde{\Div m}}\theta.
\end{displaymath}
This quantity does not vanish in general. In this case, the lack of invariance under the tranformation $\tilde{\nu}\mapsto\tilde{\nu}+\theta$ will be addressed later in Section \ref{subsection:conjugate-families} by introducing the conjugate datum. 

The second relevant case is when $\nabla^{\Mcal}$ is holomorphic. Let $\vartheta$ be a harmonic differential form whose class in $M_{dR}(X)=H^{1}(X,\CBbb)/H^{1}(X,2\pi i\, \ZBbb)$ corresponds to the connection $\nabla^{\Mcal}$. Then, the associated Chern connection corresponds to $\vartheta''-\overline{\vartheta}''$, and we have the comparison
\begin{equation}\label{eq:11}
	\nabla^{\Mcal}=\nabla^{\Mcal}_{ch}+\vartheta'+\overline{\vartheta}''.
\end{equation}
 Also, because $\theta$ has purely imaginary periods, we have a decomposition
\begin{equation}\label{eq:12}
	\theta=\theta''-\overline{\theta}''.
\end{equation}
These relations will be used in the proof of the following statement.
\begin{proposition}[{\sc Refined Poincar\'e-Lelong equation}]\label{prop:refined-PL}
Assume $\nabla^{\Mcal}$ is holomorphic and choose a harmonic one form $\vartheta$ representing the class of $\nabla^{\Mcal}$ in $M_{dR}(X)$. Let $\theta$ be a harmonic one form with periods in $2\pi i\, \ZBbb$. Then
\begin{equation}\label{eq:refining-PL}
	\int_{\tilde{\sigma}}^{\widetilde{\Div m}}\theta+\frac{i}{2\pi}\int_{X}\frac{\nabla m}{m}\wedge\theta=
	\frac{i}{2\pi }\int_{X}\vartheta\wedge\theta \quad\mod 2\pi i\, \mathbb{Z}.
\end{equation}
\end{proposition}
\begin{remark}
Before giving the proof, let us observe that this relation is a refinement of the Poincar\'e-Lelong equation applied to $\theta$. Indeed, in a family situation, we may differentiate \eqref{eq:refining-PL} following the differentiation rules for currents \eqref{eq:current_diff_2} (the indeterminacy of $2\pi i\, \mathbb{Z}$ is locally constant and thus killed by differentiation). In a family situation $\theta$ is necessarily flat for the Gauss-Manin connection: $\nabla_{\GM}\theta=0$. We obtain
\begin{displaymath}
	\tr_{\Div m}(\theta)+\pi_{\ast}\left(d\left[\frac{\nabla m}{m}\right]\wedge\theta\right)=\frac{i}{2\pi}\pi_{\ast}(F_{\nabla^{\Mcal}}\wedge\theta).
\end{displaymath}
This is a relative version of the Poincar\'e-Lelong equation applied to $\theta$!
\end{remark}
\begin{proof}[Proof of Proposition \ref{prop:refined-PL}]
The statement is the conjunction of various reciprocity laws. They involve the boundary of a fundamental domain delimited by (liftings of) simple curves $\alpha_{i}$ and $\beta_i$ whose homology classes provide with a symplectic basis of $H_{1}(X,\mathbb{Z})$, with 
 intersection matrix 
$$\left(\begin{array}{cc}
	0	&+1_{g}\\
	-1_{g}	&0
\end{array}\right).$$ We can take $\widetilde{\Div m}$ in the chosen fundamental domain based at $\tilde{\sigma}$, because we already justified \eqref{eq:100} does not depend on representatives. Applying the reciprocity formula to $\theta'$ and the meromorphic differential form $\nabla^{\Mcal}_{ch}m/m=\partial\log\|m\|^{2}$ (where $\|\cdot\|$ stands for a flat metric on $\Mcal$), we have
\begin{equation}\label{eq:13}
	\int_{\tilde{\sigma}}^{\widetilde{\Div m}}\theta'=\frac{1}{2\pi i\, }\sum_{j}\int_{\alpha_j}\theta'\int_{\beta_j}\partial\log\|m\|^{2}
	-\int_{\beta_j}\theta'\int_{\alpha_j}\partial\log\|m\|^{2}.
\end{equation}
Now we take into account that $\theta'=-\overline{\theta}''$, and conjugate the previous expression to obtain
\begin{displaymath}
	-\int_{\tilde{\sigma}}^{\Div\tilde{m}}\theta''=\frac{1}{2\pi i\, }\sum_{j}\int_{\alpha_j}\theta''\int_{\beta_j}\overline{\partial}\log\|m\|^{2}
	-\int_{\beta_j}\theta''\int_{\alpha_j}\overline{\partial}\log\|m\|^{2}.
\end{displaymath}
But observe that for a closed curve $\gamma$ disjoint from the divisor of $m$, we have by Stokes' theorem,
\begin{displaymath}
	\int_{\gamma}d\log\|m\|^{2}=0,
\end{displaymath}
and therefore
\begin{displaymath}
	\int_{\gamma}\partial\log\|m\|^{2}=-\int_{\gamma}\overline{\partial}\log\|m\|^{2}.
\end{displaymath}
We thus derive
\begin{equation}\label{eq:14}
	\int_{\tilde{\sigma}}^{\widetilde{\Div m}}\theta''=\frac{1}{2\pi i\, }\sum_{j}\int_{\alpha_j}\theta''\int_{\beta_j}\partial\log\|m\|^{2}
	-\int_{\beta_j}\theta''\int_{\alpha_j}\partial\log\|m\|^{2}.
\end{equation}
Equations \eqref{eq:13}--\eqref{eq:14} together lead to
\begin{equation}\label{eq:15}
	\int_{\tilde{\sigma}}^{\Div\tilde{m}}\theta=\frac{1}{2\pi i\, }\sum_{j}\int_{\alpha_j}\theta\int_{\beta_j}\partial\log\|m\|^{2}
	-\int_{\beta_j}\theta\int_{\alpha_j}\partial\log\|m\|^{2}.
\end{equation}
But now, modulo $2\pi i\, \mathbb{Z}$, we have
\begin{equation}\label{eq:16}
	\int_{\gamma}\partial\log\|m\|^{2}=\int_{\gamma}(\vartheta''-\overline{\vartheta}'').
\end{equation}
Because the periods of $\theta$ are in $2\pi i\, \mathbb{Z}$, eqs.\ \eqref{eq:15}--\eqref{eq:16} summarize to
\begin{displaymath}
	\int_{\tilde{\sigma}}^{\widetilde{\Div m}}\theta=\frac{1}{2\pi i\, }\sum_{j}\int_{\alpha_j}\theta\int_{\beta_j}(\vartheta''-\overline{\vartheta}'')-\int_{\beta_j}\theta\int_{\alpha_j}(\vartheta''-\overline{\vartheta}'')
\end{displaymath}
modulo $2\pi i\, \mathbb{Z}$. The last combination of periods can be expressed in terms of integration over the whole $X$, and we conclude:
\begin{equation}\label{eq:17}
	\int_{\tilde{\sigma}}^{\widetilde{\Div m}}\theta=\frac{1}{2\pi i\, }\int_{X}\theta\wedge(\vartheta''-\overline{\vartheta}'').
\end{equation}
Let's now treat the second integral:
\begin{displaymath}
	\int_{X}\frac{\nabla m}{m}\wedge\theta=\int_{X}\frac{\nabla_{ch}m}{m}\wedge \theta+\int_{X}(\vartheta'+\overline{\vartheta}'')\wedge\theta.
\end{displaymath}
The first integral on the right hand side vanishes:
\begin{displaymath}
	\int_{X}\frac{\nabla_{ch}m}{m}\wedge \theta=\int_{X}d(\log\|m\|^{2}\theta'')=0,
\end{displaymath}
where we use that $\theta''$ is closed and that the differential form $\log\|m\|^2 \theta''$ has no residues. Hence we arrive at
\begin{equation}\label{eq:18}
	\frac{i}{2\pi}\int_{X}\frac{\nabla m}{m}\wedge\theta
	=\frac{i}{2\pi}\int_{X}(\vartheta'+\overline{\vartheta}'')\wedge\theta
	=\frac{1}{2\pi i\, }\int_{X}\theta\wedge (\vartheta'+\overline{\vartheta}'').
\end{equation}
We sum \eqref{eq:17} and \eqref{eq:18} to obtain
\begin{displaymath}
	\int_{\sigma}^{\widetilde{\Div m}}\theta+\frac{i}{2\pi}\int_{X}\frac{\nabla m}{m}\wedge\theta=
	\frac{1}{2\pi i\, }\int_{X}\theta\wedge\vartheta
\end{displaymath}
modulo $2\pi i\, \mathbb{Z}$, as was to be shown.
\end{proof}
\begin{remark}
The integral
\begin{displaymath}
	\frac{i}{2\pi }\int_{X}\vartheta\wedge\theta\quad\mod 2\pi i\, \mathbb{Z}
\end{displaymath}
depends only on the class of $\vartheta$ modulo the lattice $H^{1}(X,2\pi i\, \mathbb{Z})$, or equivalently on the point $[\nabla^{\Mcal}]$ in $M_{dR}(X)$.
\end{remark}

\section{The Intersection Logarithm}\label{section:int-log}

\subsection{Intersection logarithms in conjugate families}\label{subsection:conjugate-families}
As previously, $X$ is a compact Riemann surface and $\sigma\in X$ a fixed base point. We regard $X$ as a smooth projective algebraic curve over $\CBbb$, and then we write $X\rightarrow\Spec\CBbb$ for the structure map. Let $\overline{X}$ be the conjugate Riemann surface to $X$. As a differentiable manifold $\overline{X}$ coincides with $X$, and in particular $\pi_{1}(X,\sigma)=\pi_{1}(\overline{X},\overline{\sigma})$. The almost complex structures  of $X$ and $\overline{X}$ are related by $\overline J=-J$, and hence the orientations are opposite to each other. We can also canonically realize $\overline{X}$ as the complex analytic manifold associated to the base change of $X\rightarrow\Spec\CBbb$ by the conjugation $\CBbb\rightarrow\CBbb$. 

Introduce $(\Lcal^{c}, \nabla^{\Lcal,c})$ the canonically rigidified (at $\overline{\sigma}$) holomorphic line bundle with connection attached to the differential form $-\tilde{\nu}$, regarded as a differential form on $\overline{X}$. Equivalently, if $\chi\colon\pi_{1}(X,\sigma)\rightarrow\CBbb^{\times}$ is the holonomy character of $\nabla^{\Lcal}$, then $(\Lcal^{c},\nabla^{\Lcal,c})$ is the flat holomorphic line bundle on $\overline{X}$ with holonomy character $\chi^{-1}$. We say that $(\Lcal,\nabla^{\Lcal})$ and $(\Lcal^{c},\nabla^{\Lcal,c})$ is a \emph{conjugate pair}. We emphasize that this terminology does not refer to the complex structure. As rank $1$ local systems, these bundles are mutually complex conjugate exactly when the character $\chi$ is unitary.

For the connection $\nabla^{\Mcal}$, from now on we focus on two cases:
\begin{enumerate}
	\item[\textbullet] {\bf $\nabla^{\Mcal}$ is a Chern connection (not necessarily flat).} In this case, $\Mcal^{c}$ denotes the complex conjugate line bundle to $\Mcal$ on $\overline{X}$. We let $\nabla^{\Mcal,c}$ be the conjugate of the connection $\nabla^{\Mcal}$.
	\item[\textbullet] {\bf $\nabla^{\Mcal}$ is flat.} Then we assume that $\Mcal$ is rigidified at $\sigma$. Then $(\Mcal^{c},\nabla^{\Mcal,c})$ is the flat holomorphic line bundle on $\overline{X}$, canonically rigidified at $\overline{\sigma}$, with inverse holonomy character to $(\Mcal,\nabla^{\Mcal})$.
\end{enumerate}
There is an intersection between these two situations: the flat unitary case. The conventions defining $(\Mcal^{c},\nabla^{\Mcal,c})$ are consistent. By these we mean both are mutually isomorphic: there is a unique isomorphism respecting the connections and rigidifications.
In either case, we write $\LOG^{c}_{na}$ for the corresponding naive logarithm for $\langle\Lcal^{c},\Mcal^{c}\rangle$. 
\begin{proposition}\label{prop:indep-repr}
The sum of logarithms $\LOG_{na}$ and $\LOG^{c}_{na}$, for $\langle\Lcal,\Mcal\rangle$ and $\langle\Lcal^{c},\Mcal^{c}\rangle$, defines a logarithm for
\begin{displaymath}
	\langle\Lcal,\Mcal\rangle\otimes_{\mathbb{C}}\langle\Lcal^{c},\Mcal^{c}\rangle,
\end{displaymath} 
that only depends on the point $[\nabla^{\Lcal}]$ in $M_{dR}(X)$, the rigidifications, and on $\nabla^{\Mcal}$. If $\nabla^{\Mcal}$ is flat, then the dependence on $\nabla^{\Mcal}$  factors through $M_{dR}(X)$ as well.
\end{proposition}
\begin{proof}
Let $\theta$ be a harmonic 1-form with periods in $2\pi i\, \ZBbb$. We consider the change of $\LOG_{na}$ and $\LOG_{na}^{c}$ under the transformation $\tilde{\nu}\mapsto\tilde{\nu}+\theta$, and observe that they compensate each other.

We start with the Chern connection case on $\Mcal$. Let $m$ be a meromorphic section of $\Mcal$. It defines a complex conjugate meromorphic section $m^{c}$ of $\Mcal^{c}$. On $X=\overline{X}$, the divisors $\Div m$ and $\Div m^{c}$ are equal. We saw that the change in $\LOG_{na}(\langle\ell,m\rangle)$ under $\tilde{\nu}\mapsto\tilde{\nu}+\theta$ is reduced to
\begin{displaymath}
	\int_{\tilde{\sigma}}^{\widetilde{\Div m}}\theta.
\end{displaymath}
The change in $\LOG_{na}^{c}(\langle\ell', m^{c}\rangle)$ will be
\begin{displaymath}
	\int_{\tilde{\sigma}}^{\widetilde{\Div m^{c}}}(-\theta).
\end{displaymath}
But now, independently of the liftings $\widetilde{\Div m}$ and $\widetilde{\Div m^{c}}$ in $\widetilde{X}=\widetilde{\overline{X}}$, we have
\begin{displaymath}
	\int_{\tilde{\sigma}}^{\widetilde{\Div m}}\theta+\int_{\tilde{\sigma}}^{\widetilde{\Div m^{c}}}(-\theta)=0\quad\mod 2\pi i\, \ZBbb.
\end{displaymath}
More generally, we can change $m^{c}$ by a meromorphic function. For if $f$ is meromorphic on $\overline{X}$, we have
\begin{displaymath}
	\int_{\tilde{\sigma}}^{\widetilde{\Div f}}\theta=0,
\end{displaymath}
precisely by Proposition \ref{prop:refined-PL} applied to the trivial line bundle in place of $\Mcal$. Hence, $m^{c}$ may be taken to be any meromorphic section of $\Mcal^{c}$. All in all, we see that $\LOG_{na}+\LOG_{na}^{c}$ is invariant under $\tilde{\nu}\mapsto\tilde{\nu}+\theta$.

Now for the flat connection case on $\Mcal$. We introduce a harmonic representative $\vartheta$ of the class of $\nabla^{\Mcal}$ in $M_{dR}(X)$. Then $\nabla^{\Mcal,c}$ admits $-\vartheta$ as a harmonic representative in $M_{dR}(\overline{X})$. After Proposition \ref{prop:refined-PL}, for any meromorphic section $m$ of $\Mcal$ on $X$, we have
\begin{equation}\label{eq:101}
	\int_{\tilde{\sigma}}^{\widetilde{\Div m}}\theta+\frac{i}{2\pi}\int_{X}\frac{\nabla m}{m}\wedge\theta=
	\frac{1}{2\pi i\, }\int_{X}\theta\wedge\vartheta\quad\mod 2\pi i\, \ZBbb.
\end{equation}
And if $m^{c}$ is a meromorphic section of $\Mcal^{c}$ on $\overline{X}$, we analogously find
\begin{equation}\label{eq:102}
	\int_{\overline{\sigma}}^{\widetilde{\Div m^{c}}}(-\theta)+\frac{i}{2\pi}\int_{\overline{X}}\frac{\nabla m^{c}}{m^{c}}\wedge\theta=
	\frac{1}{2\pi i\, }\int_{\overline{X}}(-\theta)\wedge(-\vartheta)\quad\mod 2\pi i\, \ZBbb.
\end{equation}
We take into account that $\overline{X}$ has the opposite orientation to $X$, so that
\begin{displaymath}
	\frac{1}{2\pi i\, }\int_{\overline{X}}(-\theta)\wedge(-\vartheta)=-\frac{1}{2\pi i\, }\int_{X}\theta\wedge\vartheta.
\end{displaymath}
Hence, the change in the sum of logarithms is \eqref{eq:101}+\eqref{eq:102}=0. Notice from the formulas defining the logarithms, that the dependence on $\nabla^{\Mcal}$ trivially factors through $M_{dR}(X)$. The statement follows.
\end{proof}
\begin{remark}
From the definition of the naive logarithm, it is automatic that there is no need to rigidify $\Mcal$. However, $\Mcal^{c}$ is rigidified by construction.
\end{remark}

Let us examine the variation of $\LOG_{na}+\LOG^{c}_{na}$ in a family. Because the construction we did of logarithms is a pointwise one (they are functions), the proposition extends to the family situation. We consider $\pi\colon\Xcal\rightarrow S$, a section $\sigma$, and its conjugate family $\overline{\pi}\colon\overline{\Xcal}\rightarrow\overline{S}$ with conjugate section $\overline{\sigma}$. The conjugate family is obtained by changing the holomorphic structure on $\Xcal$ and $S$ to the opposite one. This induces the corresponding change of holomorphic structure and orientation on the fibers. Let $(\Lcal, \nabla^{\Lcal}_{\Xcal/S})$ and $(\Mcal,\nabla^{\Mcal}_{\Xcal/S})$ be line bundles with relative compatible connections on $\Xcal$. We suppose $\nabla_{\Xcal/S}^{\Lcal}$ is flat, and $\nabla_{\Xcal/S}^{\Mcal}$ is either flat or the Chern connection associated to a smooth hermitian metric on $\Mcal$. 

When both connections are flat, we have the smooth classifying sections $\nu_{\Lcal}$ and $\nu_{\Mcal}$ of $H^{1}_{dR}(\Xcal/S)/R^{1}\pi_{\ast}(2\pi i\, \mathbb{Z})$. We then assume that on $\overline{\Xcal}$ we have rigidified line bundles with relative flat connections $(\Lcal^{c}, \nabla^{\Lcal,c}_{\overline{\Xcal}/\overline{S}})$ and $(\Mcal^{c},\nabla^{\Mcal,c}_{\overline{\Xcal}/\overline{S}})$, corresponding to the smooth sections $-\nu_{\Lcal}$ and $-\nu_{\Mcal}$ of 
\begin{displaymath}
	H^{1}_{dR}(\overline{\Xcal}/\overline{S})/R^{1}\overline{\pi}_{\ast}(2\pi i\, \mathbb{Z})=H^{1}_{dR}(\Xcal/S)/R^{1}\pi_{\ast}(2\pi i\, \mathbb{Z})
\end{displaymath}
(as differentiable manifolds). The existence is not always guaranteed, but below we deal with relevant situations when it is. The local construction of Section \ref{section:LOG} produces local naive logarithms $\LOG_{na}$ and $\LOG^{c}_{na}$, by taking local liftings $\tilde{\nu}$ and $-\tilde{\nu}$ for $\nu_{\Lcal}$ and $-\nu_{\Lcal}$, and using the canonical extensions of $\nabla^{\Mcal}_{\Xcal/S}$ and $\nabla^{\Mcal,c}_{\Xcal/S}$. Proposition \ref{prop:indep-repr} ensures that the \emph{a priori} locally defined combination $\LOG_{an}+\LOG^{c}_{an}$ on the smooth line bundle
\begin{displaymath}
	\langle\Lcal,\Mcal\rangle\otimes_{\Ccal^{\infty}_{S}}\langle\Lcal^{c},\Mcal^{c}\rangle,
\end{displaymath}
actually globalizes to a well defined logarithm, that we call \emph{intersection logarithm}: $$\LOG_{int}:=\LOG_{na}+\LOG^{c}_{na}.$$ 

When $\nabla_{\Xcal/S}^{\Mcal}$ is the relative Chern connection attached to a smooth hermitian metric on $\Mcal$, we take $\Mcal^{c}$ to be the conjugate line bundle $\overline{\Mcal}$ on $\overline{\Xcal}$, with its conjugate Chern connection $\nabla_{\Xcal/S}^{\Mcal,c}$. For $\Lcal$, as above we assume the existence of a rigidified $(\Lcal^{c},\nabla^{\Lcal,c})$, with classifying map $-\nu_{\Lcal}$. Again, by Proposition \ref{prop:indep-repr} the locally defined $\LOG_{an}+\LOG_{an}^{c}$ extends to a global logarithm that we also denote $\LOG_{int}$.

We summarize the main features of $\LOG_{int}$.
\begin{proposition}\label{prop:non-rig-LOG}
$\hbox{}$
\begin{enumerate}
\item When all connections are flat, the construction of $\LOG_{int}$ does not depend on the section $\sigma$ and rigidifications.
\item In general, the smooth connection attached to $\LOG_{int}$ is the tensor product of intersection connections.
\end{enumerate}
\end{proposition}
\begin{proof}
We begin with the case when both connections are flat. The first item can be checked pointwise. Let us examine the terms in the definition of $\LOG_{na}$, $\LOG^{c}_{na}$ and $\LOG_{int}$. Suppose we fix another base point $\sigma'$ (and lifting $\tilde{\sigma}'$) and another rigidification. Let $\tilde{\ell}$ and $\tilde{\ell}'$ be equivariant meromorphic functions with character $\chi$, lifting the same meromorphic section of $\Lcal$. Then, for some $\lambda\in\CBbb^{\times}$, we have $\tilde{\ell}'=\lambda\tilde{\ell}$. Therefore, evaluating multiplicatively over a degree 0 divisor $D$ (say in a fundamental domain), we see
that
$\tilde{\ell}'(D)=\tilde{\ell}(D)$. The same happens for $\Lcal^{c}$. Also, in $\LOG_{na}$ we have the change
\begin{displaymath}
	\int_{\tilde{\sigma}}^{z}\tilde{\nu}=\int_{\tilde{\sigma}}^{\tilde{\sigma}'}\tilde{\nu}+\int_{\sigma'}^{z}\tilde{\nu}.
\end{displaymath}
The evaluation at a divisor is defined to be additive. Therefore, for a divisor $D$ we find
\begin{displaymath}
	\int_{\tilde{\sigma}}^{D}\tilde{\nu}=(\deg D)\int_{\tilde{\sigma}}^{\tilde{\sigma}'}\tilde{\nu}+\int_{\tilde{\sigma}'}^{D}\tilde{\nu}.
\end{displaymath}
We are concerned with the case $D=\widetilde{\Div m}$, when $\deg D=0$. This shows the independence of this term of the base point. The same argument applies to $\Lcal^{c}$. Finally, there is nothing to say about the remaining terms in the definition of $\LOG_{na}$ and $\LOG^{c}_{na}$, since they only depend on the vertical connections $\nabla^{\Mcal}$ and $\nabla^{\Mcal,c}$ (as we see pointwise) and $\tilde{\nu}$, and hence do not depend on base points nor rigidifications. The dependence on the choice of $\tilde{\nu}$ modulo $R^{1}\pi_{\ast}(2\pi i\, \ZBbb)$ was already addressed (Proposition \ref{prop:indep-repr}). We conclude that $\LOG_{int}$ does not depend on $\sigma$ and the rigidifications.

For the second item, it is enough to observe that
\begin{displaymath}
	\overline{\pi}_{\ast}(F_{\nabla^{\Mcal,c}}\wedge\tilde{\nu})=-\pi_{\ast}((-F_{\nabla^{\Mcal}})\wedge(-\tilde{\nu}))
\end{displaymath}
(opposite orientation on fibers) and apply Proposition \ref{prop:LOG-int}. We obtain
\begin{displaymath}
	d\LOG_{int}(\langle\ell,m\rangle\otimes\langle\ell',m'\rangle)=\frac{\nabla_{\langle\Lcal,\Mcal\rangle}^{int}\langle\ell,m\rangle}{\langle\ell,m\rangle}
	+\frac{\nabla_{\langle\Lcal^{c},\Mcal^{c}\rangle}^{int}\langle\ell',m'\rangle}{\langle\ell',m'\rangle}.
\end{displaymath}
Now we treat the second item when $\nabla_{\Xcal/S}^{\Mcal}$ is a Chern connection. We need to justify that
\begin{displaymath}
	\pi_{\ast}(F_{\nabla^{\Mcal}}\wedge\tilde{\nu})+\overline{\pi}_{\ast}(F_{\nabla^{\Mcal,c}}\wedge(-\tilde{\nu}))=0.
\end{displaymath}
Here, $\nabla^{\Mcal}$ is the global Chern connection attached to the smooth hermitian metric on $\Mcal$, and $\nabla^{\Mcal,c}$ is the conjugate connection. Therefore, the relation between their curvatures is
$F_{\nabla^{\Mcal}}=-F_{\nabla^{\Mcal,c}}$. 
The claim follows as in the flat case, \emph{i.e.} because the fibers of $\pi$ and of $\overline{\pi}$ have opposite orientation.
\end{proof}
\begin{corollary}
Given $(\Lcal,\nabla^{\Lcal}_{\Xcal/S})$, $(\Mcal,\nabla^{\Mcal}_{\Xcal/S})$,$(\Lcal^{c},\nabla^{\Lcal,c}_{\overline{\Xcal}/S})$, $(\Mcal^{c},\nabla^{\Mcal,c}_{\overline{\Xcal}/S})$ with flat connections and no assumption on rigidifications, the smooth line bundle $\langle\Lcal,\Mcal\rangle\otimes_{\Ccal^{\infty}_{S}}\langle\Lcal^{c},\Mcal^{c}\rangle$ has a canonically defined smooth logarithm, $\LOG_{int}$, that coincides with the previous construction in presence of a rigidification. Its attached connection is the tensor product of intersection connections.
\end{corollary}
\begin{proof}
Locally over $S$, we can find sections and rigidify our line bundles. We conclude by Proposition \ref{prop:non-rig-LOG}. 
\end{proof}
\begin{remark}
In the flat case, and if $S$ has dimension at least 1, the relation of $\LOG_{int}$ to the intersection connection shows that $\LOG_{int}$ is compatible with the symmetry of Deligne pairings, up to a constant. Of course this argument cannot be used when $S$ is reduced to a  point. We will show below that $\LOG_{int}$ is indeed symmetric (Proposition \ref{prop:int-log-sym}).
\end{remark}

\subsection{Intersection logarithm in the universal case}
An important geometric setting when an intersection logarithm can be defined is the ``universal" product situation. A study of this case will lead below to the proof of the symmetry of intersection logarithms.

Let $(X,\sigma)$ be a pointed Riemann surface and $M_{B}(X)$ the Betti moduli space of complex characters of $\pi_{1}(X,\sigma)$. Let $(\overline{X},\overline{\sigma})$ be the conjugate Riemann surface, and identify $M_{B}(\overline{X})$ with $M_{B}(X)$. We have relative curves $X\times M_{B}(X)\rightarrow M_{B}(X)$ and similarly for $\overline{X}$. There are universal rigidified holomorphic line bundles with relative flat connections $(\Lcal_{\chi},\nabla_{\chi})$ and $(\Lcal^{c}_{\chi},\nabla_{\chi}^{c})$, whose holonomy characters over a given $\chi\in M_{B}(X)$ are $\chi$ and $\chi^{-1}$ respectively. Observe on the conjugate surface $\overline{X}$ the character we use is not $\overline{\chi}$. This is important since we seek an intersection logarithm that depends holomorphically on $\chi$.  We take the Deligne pairing
\begin{displaymath}
	\langle\Lcal_{\chi},\Lcal_{\chi}\rangle\otimes_{\Ocal_{M_{B}(X)}}\langle\Lcal^{c}_{\chi},\Lcal^{c}_{\chi}\rangle.
\end{displaymath} 
On the associated smooth line bundle, a slight modification of the construction of $\LOG_{int}$ produces a well defined logarithm, still denoted $\LOG_{int}$. The only difference is that now we do not need to change the holomorphic structure on $M_{B}(X)$. It is proven in \cite[Sec. 5]{FreixasWentworth:15} that this $\LOG_{int}$ is actually a holomorphic logarithm. More generally, we may work over $S=M_{B}(X)\times M_{B}(X)$. On $X\times S$ and $\overline{X}\times S$ we consider the pairs of universal bundles $(\Lcal_{\chi_1},\Mcal_{\chi_2})$ and $(\Lcal^{c}_{\chi_1},\Mcal^{c}_{\chi_2})$. We also have a universal intersection logarithm $\LOG_{int}$ on 
\begin{displaymath}
	\langle\Lcal_{\chi_1},\Mcal_{\chi_2}\rangle\otimes_{\Ocal_S}\langle\Lcal^{c}_{\chi_1},\Mcal^{c}_{\chi_2}\rangle,
\end{displaymath}
whose connection is the sum of intersection connections.

It proves useful to establish the symmetry of general intersection logarithms:
\begin{proposition}\label{prop:int-log-sym}
The intersection logarithms for line bundles with relative flat connections are symmetric, \emph{i.e.} compatible with the symmetry of Deligne pairings.
\end{proposition}

\begin{proof}
This is a pointwise assertion. Deforming to $M_{B}(X)$, it is enough to deal with the universal situation parametrized by $S=M_{B}(X)\times M_B(X)$. Because the intersection connection is symmetric, and $S$ is connected, we see that the intersection logarithm is symmetric up to a constant. Now it is enough to specialize to the pair of trivial characters, when the intersection logarithm is indeed symmetric, by Weil's reciprocity law. This concludes the proof.
\end{proof}
\begin{corollary}\label{corollary:int-log-holom}
The intersection logarithm on the universal pairing 
\begin{displaymath}
	\langle\Lcal_{\chi_1},\Mcal_{\chi_2}\rangle\otimes_{\Ocal_S}\langle\Lcal^{c}_{\chi_1},\Mcal^{c}_{\chi_2}\rangle,
\end{displaymath}
parametrized by $M_{B}(X)\times M_{B}(X)$ is holomorphic.
\end{corollary}
\begin{proof}
The holomorphy along the diagonal $\chi_{1}=\chi_{2}$ holds, since the intersection connection is holomorphic there by \cite[Sec.\ 5.3]{FreixasWentworth:15}. For the general case, we reduce to the diagonal. First, the multiplication map $(\chi_{1},\chi_{2})\mapsto \chi_{1}\chi_{2}$ is holomorphic, and induces the identification
\begin{displaymath}
	\Lcal_{\chi_{1}\chi_{2}}=\Lcal_{\chi_{1}}\otimes\Lcal_{\chi_{2}},
\end{displaymath}
and similarly for $\Mcal_{\chi_{1}\chi_{2}}$, etc. Second, we have the ``polarization formula'',
\begin{displaymath}
	\langle\Lcal\otimes\Mcal,\Lcal\otimes\Mcal\rangle=\langle\Lcal,\Lcal\rangle\otimes\langle\Lcal,\Mcal\rangle\otimes\langle\Mcal,\Lcal\rangle\otimes\langle\Mcal,\Mcal\rangle.
\end{displaymath}
and the symmetry of intersection logarithms already proven. These observations and the proposition are enough to conclude the result.
\end{proof}
A variant concerns the pairing of the universal bundles with a fixed hermitian line bundle $\Mcal$ on $X$, trivially extended to $X\times M_{B}(X)$.
\begin{corollary}\label{corollary:int-log-holom-2}
Let $\Mcal$ be a line bundle on $X$, $\ov{\Mcal}$ its conjugate line bundle on $\ov{X}$, and suppose that they are both endowed with a Chern connection. Extend trivially these data to $X\times M_{B}(X)$ and $\ov{X}\times M_{B}(X)$ by pull-back through the first projection. Then the intersection logarithm on 
\begin{displaymath}
	\langle\Lcal_{\chi},\Mcal\rangle\otimes\langle\Lcal_{\chi}^{c},\ov{\Mcal}\rangle,
\end{displaymath}
parametrized by $M_{B}(X)$, is holomorphic and does not depend on the choices of Chern connections.
\end{corollary}
\begin{proof}
By Lemma \ref{lemma:chern-on-M}, we can suppose that (i) $\Mcal$ is of relative degree 0 and rigidified along $\sigma$ and (ii) its Chern connection is flat. Similarly, we can assume its conjugate line bundle comes with the conjugate connection. Therefore, there exists $\chi_{0}$ a unitary character and an isomorphism of rigidified line bundles with connections
\begin{displaymath}
	(\Lcal_{\chi_{0}},\nabla_{\chi_{0}})\isorightarrow(\Mcal,\nabla^{\Mcal}),\quad(\Lcal_{\chi_{0}}^{c},\nabla_{\chi_{0}}^{c})\isorightarrow(\ov{\Mcal},\nabla^{\ov{\Mcal}}).
\end{displaymath}
We conclude by Corollary \ref{corollary:int-log-holom} restricted to $\chi_{2}=\chi_{0}$.
\end{proof}

\subsection{Explicit construction for families}\label{subsec:real-families}
In view of arithmetic applications, it is important to exhibit natural geometric situations when the setting of Section \ref{subsection:conjugate-families} indeed obtains. With the notations therein, the difficulty is the existence of the invertible sheaf with connection $(\Lcal^{c},\nabla^{\Lcal,c}_{\overline{\Xcal}/\overline{S}})$. Even when the existence is granted, it would be useful to have at our disposal a general {\bf algebraic} procedure to build $(\Lcal^{c},\nabla^{\Lcal,c}_{\overline{\Xcal}/\overline{S}})$ from $(\Lcal,\nabla_{\Xcal/S}^{\Lcal})$. By algebraic procedure we mean a construction that can be adapted to the schematic (for instance the arithmetic) setting.

Let $\Xcal$ and $S$ be quasi-projective, smooth, connected algebraic varieties over $\CBbb$. We regard them as complex analytic manifolds. Let $\pi:\Xcal\rightarrow S$ be a smooth and proper morphism of relative dimension 1, with connected fibers. Let $\Lcal$ and $\Mcal$ be line bundles on $\Xcal$. We distinguish three kinds of relative flat connections on $\Lcal$ and $\Mcal$: real holonomies, imaginary holonomies, and the ``mixed'' case. When $S$ is reduced to a point, the mixed case is actually the general one. Furthermore, it is then possible to give an explicit description of the intersection logarithm.

\subsubsection{Real holonomies}\label{subsec:real-holonomies} Let $\Lcal$ and $\Mcal$ be invertible sheaves over $\Xcal$, that we see as holomorphic sheaves. Let $\nabla^{\Lcal}_{\Xcal/S}:\Lcal\rightarrow\Lcal\otimes\Acal^{1}_{\Xcal/S}$ and $\nabla^{\Mcal}_{\Xcal/S}:\Mcal\rightarrow\Mcal\otimes\Acal^{1}_{\Xcal/S}$ be relative flat connections, compatible with the holomorphic structures. We suppose here that the holonomies of $\nabla^{\Lcal}_{\Xcal/S}$ and $\nabla^{\Mcal}_{\Xcal/S}$ on fibers are real. On the conjugate variety $\overline{\Xcal}$, the conjugate line bundles $\overline{\Lcal}$ and $\overline{\Mcal}$ admit the complex conjugate connections to $\nabla^{\Lcal}_{\Xcal/S}$ and $\nabla^{\Mcal}_{\Xcal/S}$. Observe the families of holonomy representations do not change, because of the real assumption. We thus see that
\begin{displaymath}
	(\Lcal^{c},\nabla_{\Xcal/S}^{\Lcal,c})=(\overline{\Lcal}^{\vee},-\overline{\nabla}_{\Xcal/S}^{\Lcal}),\quad (\Mcal^{c},\nabla_{\Xcal/S}^{\Mcal,c})=(\overline{\Mcal}^{\vee},-\overline{\nabla}_{\Xcal/S}^{\Mcal}).
\end{displaymath}
The bar on the connections stands for complex conjugation.\\

\paragraph{\textbf{Explicit description of the intersection logarithm when $S=\Spec\CBbb$}} When the base scheme is a point, we write $X$, $p$, $L$, $M$, $\nabla_L$, $\nabla_M$ instead of $\Xcal$, $\sigma$, $\Lcal$, $\Mcal$, $\nabla_{\Xcal/S}^{\Lcal}$, $\nabla_{\Xcal/S}^{\Mcal}$. The first important remark is that since the connections $\nabla_{L}$ and $\nabla_{M}$ have real holonomy characters $\chi_L$ and $\chi_M$, they determine \emph{unique} real harmonic differential forms $\nu$ and $\vartheta$. Namely, harmonic differential forms on the Riemann surface $X$, invariant under the action of complex conjugation. The relation is
\begin{displaymath}
	\chi_{L}(\gamma)=\exp\left(\int_{\gamma}\nu\right),\quad\chi_{M}(\gamma)=\exp\left(\int_{\gamma}\vartheta\right),\quad\gamma\in\pi_{1}(X(\CBbb),p).
\end{displaymath}
Because $\nu$ and $\vartheta$ are real, we can write them
\begin{displaymath}
	\nu=\nu'+\overline{\nu}',\quad\vartheta=\vartheta'+\overline{\vartheta}',
\end{displaymath}
where $\nu'$ and $\vartheta'$ are holomorphic. In terms of these forms, we first provide the action of the naive logarithms on standard sections. Let $\ell$ and $m$ be rational sections of $L$ and $M$ on $X$, with disjoint divisors. After a choice of rigidification of $L$, we lift $\ell$ to a meromorphic function on the universal cover $\widetilde{X}$ (with its natural complex structure), transforming like $\chi_L$ under the action of $\pi_{1}(X,p)$. Also, we lift $\Div m$ to $\widetilde{\Div m}$. The naive logarithm for the complex structure on $X$ is determined by
\begin{displaymath}
	\LOG_{na}(\langle\ell,m\rangle)=\log(\tilde{\ell}(\widetilde{\Div m}))-\int_{\tilde{p}}^{\widetilde{\Div m}}\nu
	-\frac{i}{2\pi}\int_{X}\frac{\nabla m}{m}\wedge\nu.
\end{displaymath} 
Recall that the first two terms together do not change under a transformation $\nu\mapsto\nu+\theta$, for $\theta$ holomorphic. Using the relation with the Chern connections
\begin{displaymath}
	\nabla_{L}=\nabla_{L,ch}+2\nu',\quad\nabla_{M}=\nabla_{M,ch}+2\vartheta',
\end{displaymath}
we simplify the naive logarithm to
\begin{displaymath}
	\LOG_{na}(\langle\ell,m\rangle)=\log(\tilde{\ell}_{ch}(\widetilde{\Div m}))-\int_{\tilde{p}}^{\widetilde{\Div m}}(\overline{\nu}'-\nu')
	-\frac{i}{2\pi}\int_{X}\vartheta'\wedge\overline{\nu}'.
\end{displaymath}
We denoted $\tilde{\ell}_{ch}$ the lift of $\ell$ using the Chern connection $\nabla_{L,ch}$. Changing the holomorphic structure (and hence reversing the orientation in the last integral), the naive logarithm $\LOG_{na}^{c}$ computed with the conjugate sections $\overline{\ell}$ and $\overline{m}$ is
\begin{displaymath}
	\LOG_{na}^{c}(\langle\overline{\ell}^{\vee},\overline{m}^{\vee}\rangle)=\log(\overline{\tilde{\ell}_{ch}(\widetilde{\Div m})})-\int_{\tilde{p}}^{\widetilde{\Div m}}(\nu'-\overline{\nu}')
	+\frac{i}{2\pi}\int_{X}\overline{\vartheta}'\wedge\nu'.
\end{displaymath}
All in all, we find
\begin{displaymath}
	\begin{split}
	\LOG_{int}(\langle\ell,m\rangle\otimes\langle\overline{\ell}^{\vee},\overline{m}^{\vee}\rangle)
	&=\log|\tilde{\ell}_{ch}(\widetilde{\Div m})|^{2}-\frac{1}{\pi}\Imag\left(\int_{X(\CBbb)}\vartheta'\wedge\overline{\nu}'\right)\\
	&=\log\|\langle\ell,m\rangle\|^{2}+\frac{1}{\pi}\Imag\left(\int_{X(\CBbb)}\vartheta'\wedge\overline{\nu}'\right).
	\end{split}
\end{displaymath}
The norm on the Deligne pairing is the canonical one for pairings of degree 0 line bundles. As predicted by Proposition \ref{prop:non-rig-LOG}, the formula does not depend on the rigidification. Notice also this expression is real valued.

\subsubsection{Unitary connections} We suppose the holomorphic line bundles $\Lcal$, $\Mcal$ come with relative flat unitary connections $\nabla^{\Lcal}_{\Xcal/S}$, $\nabla^{\Mcal}_{\Xcal/S}$. Their family holonomy representations are thus unitary. For the complex conjugate family, it is therefore enough to take
\begin{displaymath}
	(\Lcal^{c},\nabla_{\Xcal/S}^{\Lcal,c})=(\overline{\Lcal},\overline{\nabla}_{\Xcal/S}^{\Lcal}),\quad (\Mcal^{c},\nabla_{\Xcal/S}^{\Mcal,c})=(\overline{\Mcal},\overline{\nabla}_{\Xcal/S}^{\Mcal}).
\end{displaymath}
Contrary to the real case, we do not need to dualize the complex conjugate line bundles.

In this case, the intersection logarithm $\LOG_{int}$ amounts to the logarithm of a smooth hermitian metric. If $\ell$ and $m$ are rational sections of $\Lcal$, $\Mcal$, with finite, \'etale, disjoint divisors over some Zariski open subset of $S$, then one easily sees
\begin{displaymath}
	\LOG_{int}(\langle\ell,m\rangle\otimes\langle\overline{\ell},\overline{m}\rangle)=\log\|\langle\ell,m\rangle\|^{2}.
\end{displaymath}
That is, the log of the square of the natural norm on the Deligne pairing.

\subsubsection{Mixed case} Suppose that $\Lcal$ is equipped with a flat relative connection (compatible with the holomorphic structure), with real holonomies, and $\Mcal$  with a relative flat unitary connection. Then  the tensor product of connections on $\Pcal=\Lcal\otimes\Mcal$ is no longer real nor unitary. Nevertheless, we can still define $\Pcal^{c}$ and $\nabla^{\Pcal,c}_{\Xcal/S}$ on the conjugate family:
\begin{displaymath}
	(\Pcal^{c}, \nabla^{\Pcal,c}_{\Xcal/S})=(\overline{\Lcal}^{\vee}\otimes\overline{\Mcal}, (-\overline{\nabla}^{\Lcal}_{\Xcal/S})\otimes\overline{\nabla}^{\Mcal}_{\Xcal/S}).
\end{displaymath}
In this case, we postpone the explicit description of the intersection logarithm to the next paragraph.
\subsection{The mixed case  over $\Spec\CBbb$}\label{subsec:mixed-case} 
Suppose now that the base scheme $S$ is a point. Therefore, we are dealing with a single compact Riemann surface $X$. We fix a base point $p\in X$. Let $P$ be a line bundle over $X$ with a connection
\begin{displaymath}
	\nabla_{P}\colon P\longrightarrow P\otimes\Omega^{1}_{X/\CBbb}.
\end{displaymath}
Let $\chi$ be the holonomy representation of $\nabla_{P}$. The absolute value $|\chi|$ is the holonomy representation of a line bundle $L$ on $X$ endowed with a holomorphic connection $\nabla_{L}$. We set
\begin{displaymath}
	M:=P\otimes L^{\vee},\quad\nabla_{M}=\nabla_{P}\otimes(-\nabla_{L}).
\end{displaymath}
Then $M$ is a line bundle with a flat unitary connection $\nabla_{M}$, and $P=L\otimes M$, $\nabla_{P}=\nabla_{L}\otimes\nabla_{M}$ are as in the mixed case.\\

\paragraph{\textbf{Explicit description of the intersection logarithm}} Let $L$ and $M$ be line bundles over $X$. Let $\nabla_{L}$ and $\nabla_{M}$ be arbitrary holomorphic connections on $L$ and $M$. We wish to describe the intersection logarithm. Taking into account the decomposition of $L$ and $M$ in terms of real/unitary holonomy flat bundles as above, the new case to study is when $\nabla_{L}$ has real holonomy $\chi_{L}$ and $\nabla_{M}$ is unitary (for the reverse case we invoke the symmetry of the intersection logarithm, Proposition \ref{prop:int-log-sym}).

Let $\ell,m$ be rational sections of $L,M$ respectively, with disjoint divisors. After trivializing $L$ at $p$, we lift $\ell$ to a meromorphic function $\tilde{\ell}$ on the universal cover $\widetilde{X}$, transforming like $\chi_{L}$ under the action of $\pi_{1}(X,p)$. We lift $\Div m$ to $\widetilde{\Div m}$. The naive logarithm for the natural complex structure on $X$ is determined by
\begin{equation}\label{eq:200}
	\begin{split}
	\LOG_{na}(\langle\ell,m\rangle)&=
	\log(\tilde{\ell}(\widetilde{\Div m}))-\int_{\tilde{p}}^{\widetilde{\Div m}}\nu-\frac{i}{2\pi}\int_{X}\frac{\nabla m}{m}\wedge\nu\\
	&=\log(\tilde{\ell}(\widetilde{\Div m}))-\int_{\tilde{p}}^{\widetilde{\Div m}}\nu.
	\end{split}
\end{equation}
We wrote $\nu$ for the real harmonic differential form determined by $\chi_{L}$. The second equality uses that $\nabla_{M}$ is a Chern connection. There is a similar expression for the naive logarithm $\LOG_{na}^{c}$. In the present case it takes the form
\begin{equation}\label{eq:201}
	\LOG_{na}^{c}(\langle\overline{\ell}^{\vee},\overline{m}\rangle)=
	-\log(\overline{\tilde{\ell}(\widetilde{\Div m})})-\int_{\tilde{p}}^{\widetilde{\Div m}}(-\nu).
\end{equation}
Adding \eqref{eq:200} and \eqref{eq:201} and simplifying, we find for the intersection logarithm
\begin{equation}\label{eq:202}
	\LOG_{int}(\langle\ell,m\rangle\otimes\langle\overline{\ell}^{\vee},\overline{m}\rangle)=2i\arg(\tilde{\ell}(\widetilde{\Div m})).
\end{equation}
This quantity is purely imaginary. Again, it does not depend on the trivialization of $L$, because $\Div m$ is a degree 0 divisor. The discussion is also valid if $M$ has arbitrary degree and is endowed with a hermitian metric. However, in this case the intersection logarithm depends on the rigidification of $L$.

\section{Logarithm for the Determinant of Cohomology}\label{section:LOG-det-coh}

\subsection{The Quillen logarithm}
In this section we proceed to define a logarithm which is analog to the so-called Quillen metric, on the determinant of the cohomology of a line bundle. Later on, in Section \ref{section:CM}, we relate our construction to the holomorphic Cappell-Miller torsion \cite{CM}. 

Let $(X,p)$ be a Riemann surface with a point. Fix a hermitian metric $h_{T_X}$ on the holomorphic tangent bundle $T_X$. Take a complex character $\chi:\pi_{1}(X,p)\rightarrow\mathbb{C}^{\times}$ and write $\Lcal_{\chi}$ for the canonically trivialized (at $p$) holomorphic line bundle with flat connection, whose holonomy representation is $\chi$. This depends on the base point $p$. On the conjugate Riemann surface $\overline{X}$ we consider the flat holomorphic line bundle $\Lcal^{c}_{\chi}$ attached to the character $\chi^{-1}$ of $\pi_{1}(\overline{X},\overline{p})=\pi_{1}(X,p)$. Consider the product of determinants of cohomology groups
\begin{displaymath}
	\det H^{\bullet}(X,\Lcal_{\chi})\otimes_{\mathbb{C}}\det H^{\bullet}(\overline{X},\Lcal^{c}_{\chi}).
\end{displaymath}
We will construct a canonical determinant on this complex line, but before we make several observations regarding these cohomology groups. To simplify the discussion, we assume that $L_\chi$ is not trivial. Then
\begin{displaymath}
	H^{0}(X,\Lcal_{\chi})=H^{0}(\overline{X},\Lcal^{c}_{\chi})=0.
\end{displaymath}
Their determinants are canonically isomorphic to $\mathbb{C}$, that affords the usual logarithm on $\mathbb{C}^{\times}$ (modulo $2\pi i\, \mathbb{Z}$). Let us examine the $H^{1}$'s. By Hodge theory and uniformization, there is a canonical isomorphism
\begin{displaymath}
	H^{1}(X,\Lcal_{\chi})\isorightarrow H^{0,1}(\chi),
\end{displaymath}
where $H^{0,1}(\chi)$ denotes the space of antiholomorphic differential forms on the universal cover $\widetilde{X}$ (with the complex structure compatible with $X$), with character $\chi$ under the action of $\pi_{1}(X,p)$ (anti-holomorphic Prym differentials). Similarly,
\begin{displaymath}
	H^{1}(\overline{X},\Lcal^{c}_{\chi})\isorightarrow H^{1,0}(\chi^{-1}),
\end{displaymath}
the space of holomorphic differential forms on $\widetilde{X}$, with character $\chi^{-1}$ (holomorphic Prym differentials). In this identification, we see anti-holomorphic differential forms on $\widetilde{\overline{X}}$ as holomorphic differential forms on $\widetilde{X}$. Given $\alpha\in H^{0,1}(\chi)$ and $\beta\in H^{1,0}(\chi^{-1})$, the differential form
$\beta\wedge\alpha$ is a $\pi_{1}(X,p)$ invariant $(1,1)$ differential form on $\widetilde{X}$. It thus descends to a smooth $(1,1)$ form on $X$.

Let us now consider a nonvanishing tensor
\begin{displaymath}
	\eta_{1}(\chi)\wedge\ldots\wedge\eta_{g-1}(\chi)\otimes \eta_{1}(\chi^{-1})\wedge\ldots\wedge\eta_{g-1}(\chi^{-1})
\end{displaymath}
 in the product of determinants. Up to a small caveat, we would like to define
 \begin{displaymath}
 	\begin{split}
 		\LOG_{L^2}(\eta_{1}(\chi)\wedge\ldots\wedge\eta_{g-1}(\chi)\otimes &\eta_{1}(\chi^{-1})\wedge\ldots\wedge\eta_{g-1}(\chi^{-1}))
	=\\
	&\log\det\left(\frac{i}{2\pi} \int_{X}\eta_{j}(\chi^{-1})\wedge\eta_{k}(\chi)\right)_{jk}\in\mathbb{C}/2\pi i\, \mathbb{Z}.
	\end{split}
 \end{displaymath}
By duality, we derive a $L^2$ logarithm on the determinant of cohomology. Notice, however, the determinant could vanish. We now follow terminology introduced by Fay \cite{Fay}. We define $V_0\subseteq M_{B}(X)$ as the locus of characters $\chi$ with $L_{\chi}$ trivial. This is equivalent to $\dim H^{0}(X,\Lcal_{\chi})\geq 1$, and by the semi-continuity theorem of coherent cohomology, shows that $V_0$ is a closed analytic subset. It is actually nonsingular of codimension $g$. We let $V$ be the locus of characters in $M_{B}(X)\setminus V_0$ for which the determinant vanishes. By Grauert's theorem, on $M_{B}(X)\setminus V_0$ we can locally choose cohomology bases that depend holomorphically on $\chi$, so that $V$ is a divisor in $M_{B}(X)\setminus V_0$. Only for $\chi\neq V$, the logarithm $\LOG_{L^2}$ can be defined. 
  
 We introduce another logarithm on the determinant of cohomology. We start with $\chi\not\in V_0$, $\chi\not\in V$. Then we put
 \begin{displaymath}
 	\LOG_{Q}=\LOG_{L^2}-\log T(\chi),
 \end{displaymath}
 where $T(\chi)$ is the complex valued analytic torsion introduced by Fay\footnote{It is necessary to normalize Fay's definition so that it coincides with the holomorphic analytic torsion on unitary characters. The normalization requires the introduction of a constant, depending only on the genus of $X$.}, and spectrally described in \cite[Sec.\ 5]{FreixasWentworth:15}. As a consequence of results by Fay \cite[Thm.\ 1.3]{Fay}, $\exp\circ\LOG_{Q}$ depends holomorphically on $\chi\not\in V$, and can be uniquely and holomorphically extended to $M_{B}(X)\setminus V_0$, with values in $\mathbb{C}$ (notice the possible vanishing!). See also \cite[Sec.\ 5]{FreixasWentworth:15} for an explicit expression of $T(\chi)$, that relates to the determinant of the matrix of Prym differentials, from which the claim follows as well. Therefore, contrary to $\LOG_{L^2}$, $\LOG_{Q}$ can be extended to $M_{B}(X)\setminus V_0$. An alternative approach to the spectral interpretation will follow later from Section \ref{section:CM}, in the comparison of $\LOG_{Q}$ with the holomorphic Cappell-Miller torsion.
 
 A final remark indicates the relation to the Quillen metric in the unitary case. If $\chi\not\in (V\cup V_0)$ is unitary, then we can chose the bases of Prym differentials so that
 \begin{displaymath}
 	\eta_{k}(\chi^{-1})=\overline{\eta_{k}(\chi)}.
 \end{displaymath}
 Also, $T(\chi)$ is the usual real valued analytic torsion in this case. We thus see that
$$
	\LOG_{Q}(\eta_{1}(\chi)\wedge\ldots\wedge\eta_{g-1}(\chi)\otimes \eta_{1}(\chi^{-1})\wedge\ldots\wedge\eta_{g-1}(\chi^{-1}))=
	\log\|\eta_{1}(\chi)\wedge\ldots\wedge\eta_{g-1}(\chi)\|^{2}_{Q}.
$$
 Hence, $\LOG_Q$ is a natural extension of the (log of the) Quillen metric in this case! The logarithm $\LOG_Q$ will be called \emph{Quillen logarithm}.
 
 We will also need to deal with the case of the trivial line bundle. In this case, we start by defining a logarithm on
 \begin{displaymath}
 	\det H^{0}(X,\Ocal_{X})\otimes\det H^{0}(\overline{X},\Ocal_{\overline{X}})=\CBbb.
 \end{displaymath} 
 This is done by assigning to 1 the volume of the normalized K\"ahler form on $X$: locally in a holomorphic coordinate $z$,
 \begin{displaymath}
 	\frac{i}{2\pi}h_{T_X}\left(\frac{\partial}{\partial z},\frac{\partial}{\partial z}\right)dz\wedge d\overline{z}.
 \end{displaymath}
 For $H^1$'s, as before Hodge theory provides canonical identifications
 \begin{displaymath}
 	\begin{split}
	& H^{1}(X,\Ocal_{X})\isorightarrow H^{0,1}(X)\quad\text{(anti-holomorphic differential forms on $X$)}\\
	& H^{1}(\overline{X},\Ocal_{\overline{X}})\isorightarrow H^{1,0}(X)\quad\text{(holomorphic differential forms on $X$)}.
	\end{split}
 \end{displaymath}
 Then, to a nonvanishing tensor like element
 \begin{displaymath}
 	 \alpha_1\wedge\ldots\wedge\alpha_g\otimes\beta_1\wedge\ldots\wedge\beta_g\in\det H^{1}(X,\Ocal_{X})\otimes\det H^{1}(\overline{X},\Ocal_{\overline{X}}),
 \end{displaymath}
 we associate the number
 \begin{displaymath}
 	\log\det\left(\frac{i}{2\pi}\int_{X}\beta_{j}\wedge\alpha_{k}\right)\in\CBbb/2\pi i\, \ZBbb.
 \end{displaymath}
 In the determinant does not vanish. It is always possible to choose $\alpha_{k}=\overline{\beta}_k$, and in this case the logarithm is univalued and takes values in $\RBbb$. The combination of both logarithms is again denoted $\LOG_{L^2}$. If $T(\Ocal_{X})$ is the analytic torsion of the trivial hermitian line bundle on $X$, then we put
 \begin{displaymath}
 	\LOG_{Q}=\LOG_{L^2}-\log T(\Ocal_{X}).
 \end{displaymath}
 Because $T(\Ocal_{X})$ is a strictly positive real number, we see that $\LOG_{Q}$ actually amounts to the (logarithm of the) Quillen metric.
 \subsection{The Deligne-Riemann-Roch isomorphism and logarithms}
Let $(X,p)$ be a pointed compact Riemann surface with conjugate $(\overline{X},\overline{p})$. Fix a hermitian metric $h_{T_X}$ on $T_X$. It defines a hermitian metric on $T_{\overline{X}}$, and also on $\omega_{X}$ and $\omega_{\overline{X}}$. Let $(\Lcal, \nabla)$ be a holomorphic line bundle, rigidified at $p$, with a flat compatible connection. Let $\chi$ be the associated holonomy character of $\pi_{1}(X,p)$. Hence we can identify $\Lcal$ and $\Lcal_{\chi}$, and also build a conjugate pair $(\Lcal^{c},\nabla^{c})$ on $\overline{X}$, corresponding to the character $\chi^{-1}$.  The notations employed here are customary, and unravel to
 \begin{equation} \label{eqn:det}
 	\lambda(\Lcal-\Ocal_{X})=\det H^{\bullet}(X,\Lcal)\otimes\det H^{\bullet}(X,\Ocal_{X})^{-1},
 \end{equation}
and similarly for $\lambda(\Lcal^{c}-\Ocal_{\overline{X}})$. The left hand side of the Deligne isomorphism carries the combination of Quillen logarithms detailed in the previous section. The right hand side is endowed with the intersection logarithm, where $\omega_{X}$ and $\omega_{\overline{X}}$ are endowed with the Chern connections for the choice we made of hermitian metrics. Observe that the logarithms on both sides depend on the rigidification of $\Lcal$ at $p$ (used to identify $\Lcal$ to $\Lcal_{\chi}$). Indeed, this is the case for $\LOG_{L^2}$ (although not of $T(\chi)$), and of the intersection logarithm of the pairing of $(\Lcal,\Lcal^{c})$ against $(\omega_X,\omega_{\overline{X}})$.

 \begin{theorem}\label{theorem:Iso-Deligne}
 Assume $\chi\in M_{B}(X)\setminus V_0$. Deligne's isomorphism is compatible with $\LOG_Q$ and $\LOG_{int}$ modulo $\pi i\,  \ZBbb$, that is
 \begin{displaymath}
 	\overline{\LOG}_{Q}=\overline{\LOG}_{int}\circ\Dcal.
 \end{displaymath}
 \end{theorem}
 \begin{remark}\label{remark:iso-Deligne}
 The explanation for the reduction modulo $\pi i\,  \ZBbb$ is that Deligne's isomorphism is only canonical up to a sign. Hence, at most we are able to show that the two logarithms correspond up to $\log(\pm 1)$, which is zero modulo $\pi i\,  \ZBbb$.
 \end{remark}
 \begin{proof}[Proof of Theorem \ref{theorem:Iso-Deligne}]
 We provide two arguments. Both exploit the universal product fibrations $\Xcal=X\times M_{B}(X)\rightarrow M_{B}(X)=S$ and $\Xcal^{c}=\overline{X}\times M_{B}(X)\rightarrow M_{B}(X)$ with the universal line bundles $\Lcal_{\chi}$ and $\Lcal_{\chi}^{c}$. There is a universal Deligne type isomorphism
 \begin{displaymath}
 	\Dcal_{\chi}\colon (\lambda(\Lcal_{\chi}-\Ocal_{\Xcal})\otimes_{\Ocal_S}\lambda(\Lcal^{c}_{\chi}-\Ocal_{\Xcal^{c}}))^{\otimes 2}
	\isorightarrow\langle\Lcal_{\chi},\Lcal_{\chi}\otimes\omega_{X}^{-1}\rangle\otimes_{\Ocal_S}\langle\Lcal^{c}_{\chi},\Lcal^{c}_{\chi}\otimes\omega_{\overline{X}}^{-1}\rangle.
 \end{displaymath}
 The first argument refers to \cite[Sec.\ 5]{FreixasWentworth:15}, where we showed that over the connected open subset $M_{B}(X)\setminus (V\cup V_{0})$, 
 \begin{displaymath}
 	d\LOG_{Q}=d\LOG_{int}\circ\Dcal_{\chi}.
 \end{displaymath}
 This implies the equality $\LOG_{Q}=\LOG_{int}\circ\Dcal$ holds, up to a constant $\kappa$, on $M_{B}(X)\setminus (V\cup V_0)$. By smoothness of the logarithms, the same is true over $M_{B}(X)\setminus V_0$. Now we need only to check that the constant is zero modulo $\pi i\,  \ZBbb$. For this, it is enough to evaluate at a suitable $\chi$ and suitable sections. Let us take $\chi$ unitary, $\chi\not\in (V\cup V_0)$. Since unitary characters lie in the complement of $V$, and it is enough to take it  to be nontrivial. Then, $\Lcal^{c}_{\chi}$ is complex conjugate to $\Lcal_{\chi}$. Let $\ell$, $m$ be rational sections of $\Lcal_{\chi}$ with disjoint divisors, that we see as holomorphic functions $\tilde{\ell}$, $\tilde{m}$ on the universal cover $\widetilde{X}$ with character $\chi$ under the action of $\Gamma$. The complex conjugate sections of $\ell$ and $m$ are defined by conjugating $\tilde{\ell}$ and $\tilde{m}$, and we write $\bar{\ell}$ and $\bar{m}$. They can also be obtained algebraically from $\ell$, $m$ by effecting the base change $\Spec\CBbb\rightarrow\Spec\CBbb$ induced by complex conjugation. Also, let $\theta$ be a meromorphic section of $\omega_{X}$, with complex conjugate $\overline{\theta}$. We may assume the divisor of $\theta$ is disjoint from $\Div\ell\cup\Div m$. Then, because the Deligne isomorphism commutes to base change, we have
 \begin{displaymath}
 	\langle\ell,m\otimes\theta^{-1}\rangle\otimes\langle\overline{\ell},\overline{m}\otimes\overline{\theta}^{-1}\rangle
 \end{displaymath}
 corresponds under $\Dcal^{-1}_{\chi}$ to a nonvanishing tensor of the form
 \begin{displaymath}
 	\pm (\overline{\eta}_1\wedge\ldots\wedge\overline{\eta}_{g-1}\otimes\eta_{1}\wedge\ldots\wedge\eta_{g-1})^{\otimes 2}
	\otimes (\overline{\alpha}_{1}\wedge\ldots\wedge\overline{\alpha}_{g}\otimes\alpha_{1}\wedge\ldots\wedge\alpha_{g})^{\otimes 2}.
 \end{displaymath}
 But for these sections, we have on the one hand
 \begin{align}
 \begin{split} \label{eq:19}
 	\LOG_{int}(\langle\ell,m\otimes\theta^{-1}\rangle\otimes\langle\overline{\ell},\overline{m}\otimes\overline{\theta}^{-1}\rangle)&=\log\|\langle\ell,m\otimes\theta^{-1}\rangle\|^{2}\\
	&=\log|\tilde{\ell}(\Div\tilde{m}-\widetilde{\Div\theta})|^{2},
	\end{split}
 \end{align}
 and on the other hand
 \begin{align}
		\LOG_{Q}&(\pm (\overline{\eta}_1\wedge\ldots\wedge\overline{\eta}_{g-1}\otimes\eta_{1}\wedge\ldots\wedge\eta_{g-1})^{\otimes 2}
	\otimes (\overline{\alpha}_{1}\wedge\ldots\wedge\overline{\alpha}_{g}\otimes\alpha_{1}\wedge\ldots\wedge\alpha_{g})^{\otimes (-2)}) \notag \\
	&=\log(\pm 1)+2\log\|\eta_{1}\wedge\ldots\wedge\eta_{g-1}\|_{Q}^{2}-2\log\|\alpha_{1}\wedge\ldots\wedge\alpha_{g}\|^{2}_{Q} \label{eq:20}
 \end{align}
 But now, $\log(\pm 1)=0$ modulo $\pi i\,  \ZBbb$ and the Deligne isomorphism is an isometry in the unitary case. Hence, the expressions \eqref{eq:19} and \eqref{eq:20} are equal modulo $\pi i\,  \ZBbb$. Therefore, the constant $\kappa$ is zero modulo $\pi i\,  \ZBbb$. 
 
 The second argument is similar, but replaces \cite[Sec.\ 5]{FreixasWentworth:15} by the following self-contained remarks. On the one hand, the intersection logarithm on the Deligne pairings is holomorphic on $M_{B}(X)$, after Corollaries \ref{corollary:int-log-holom}--\ref{corollary:int-log-holom-2}. On the other hand, the Quillen logarithm is holomorphic on its domain $M_{B}(X)\setminus( V_{0}\cup V)$, by construction. Finally, both logarithms coincide on the unitary locus of $M_{B}(X)\setminus( V_{0}\cup V)$, which is a maximal totally real subvariety. Hence, by a standard argument (\emph{cf.} \cite[Lemma 5.12]{FreixasWentworth:15}) both holomorphic logarithms must coincide on the whole $M_{B}(X)\setminus( V_{0}\cup V)$. In either approach, the proof of the theorem is complete.
 \end{proof}
 \begin{corollary}\label{corollary:iso-Deligne}
 The Quillen logarithm $\LOG_{Q}$, initially defined on $M_{B}(X)\setminus( V_{0}\cup V)$, uniquely extends to the whole $M_{B}(X)$. The extension is compatible with Deligne's isomorphism (modulo $\pi i\,  \ZBbb$) and is holomorphic.
 \end{corollary}
 \begin{proof}
 The extension of the Quillen logarithm, modulo $\pi i\,  \ZBbb$, follows from the theorem and the fact that the intersection logarithm is already defined on the whole $M_{B}(X)$. At the same time, this logarithm is defined modulo $2\pi i\, \ZBbb$ on a dense open subset. Both observations together imply that the Quillen logarithm can be extended everywhere. For the holomorphy, by the very construction of $\LOG_{Q}$, it is satisfied on a dense open subset of $M_{B}(X)$, and hence everywhere (for the extension). Indeed, this amounts to the vanishing on the well defined smooth $(0,1)$-form $\overline{\partial {\LOG}_{Q}}$, and it is enough to check the vanishing on a dense open subset.

 \end{proof}
 
It is important to notice that the theorem implies the compatibility of Deligne's isomorphism for general conjugate families modulo $\pi i\,  \ZBbb$, as well as the compatibility with the Quillen type connections and the intersection connections. 
\begin{corollary}
In the case of general K\"ahler fibrations\footnote{For a K\"ahler fibration, here we mean a smooth family of curves $\pi:\Xcal\rightarrow S$ with a choice of smooth hermitian metric on $\omega_{\Xcal/S}$. This is not the standard definition in higher relative dimensions.} and for a conjugate pair of data $\pi:\Xcal\rightarrow S$, $(\Lcal,\nabla^{\Lcal}_{\Xcal/S})$, and $\overline{\pi}:\overline{\Xcal}\rightarrow\overline{S}$, $(\Lcal^{c},\nabla^{\Lcal,c}_{\overline{\Xcal}/\overline{S}})$, with rigidifications along a given section, the Deligne isomorphism
\begin{align}
\begin{split} \label{eqn:D-iso}
	\Dcal:(\det R\pi_{\ast}(\Lcal-\Ocal_{\Xcal})&\otimes_{\Ccal^{\infty}_{S}}\det R\overline{\pi}_{\ast}(\Lcal^{c}-\Ocal_{\overline{\Xcal}}))^{\otimes 2} \\
&	\isorightarrow\langle\Lcal,\Lcal\otimes\omega_{\Xcal/S}^{-1}\rangle\otimes_{\Ccal^{\infty}_{S}}\langle\Lcal^{c},\Lcal^{c}\otimes\omega_{\overline{\Xcal}/\overline{S}}^{-1}\rangle
\end{split}
\end{align}
transforms $\LOG_{Q}$ into $\LOG_{int}$, modulo $\pi i\,  \ZBbb$. As a consequence, the Deligne isomorphism $\Dcal$ is parallel with respect 
to the connections $\nabla_{Q}$ on the left hand side of \eqref{eqn:D-iso} and $\nabla_{int}$ on the right.
\end{corollary}

\section{The Quillen logarithm and the Cappell-Miller torsion}\label{section:CM}
In this section, we prove Theorem \ref{thm:Quillen-CM}, to the effect that the construction of the Quillen logarithm is equivalent to a variant of holomorphic analytic torsion,  proposed by Cappell-Miller \cite{CM}. Our observation is that the Cappell-Miller torsion behaves holomorphically in holomorphic families of flat line bundles on a fixed Riemann surface. In the proof, we make essential use of Kato's theory of analytic perturbations of linear closed operators \cite[Chap.\ VII]{Kato}, which turns out to be particularly well-suited for these purposes. 

Let $X$ be a fixed compact Riemann surface with a smooth hermitian metric on $T_{X}$, $p$ a base point, and $(\ov{X},\ov{p})$ the conjugate datum. Let $M_{B}(X)$ be the space of characters of $\pi_{1}(X,p)$, and $\Lcal$, $\Lcal^{c}$ the holomorphic universal bundles on $\Xcal:=X\times M_{B}(X)$ and $\Xcal^{c}:=\ov{X}\times M_{B}(X)$. Recall the fibers $\Lcal_{\chi}$, $\Lcal^{c}_{\chi}$ at $\chi\in M_{B}(X)$, are canonically trivialized at the base point and have holonomy representations $\chi$ and $\chi^{-1}$, respectively. There are corresponding universal relative holomorphic connections. Write $\pi$ and $\pi^{c}$ for the projection maps onto $M_{B}(X)$.

Inspired by Quillen \cite{Quillen}, Bismut-Freed \cite{Bismut-Freed-1, Bismut-Freed-2} and Bismut-Gillet-Soul\'e \cite{BGS1, BGS2, BGS3}, we present the determinant of cohomology $\lambda(\Lcal)=\det R\pi_{\ast}(\Lcal)$ as the determinant of a truncated Dolbeault complex of finite dimensional \emph{holomorphic} vector bundles. The difference with the cited works is in the holomorphicity of these vector bundles. One can similarly proceed for $\Lcal^{c}$.

Introduce the relative Dolbeault complex of $\Lcal$, considered as a smooth complex line bundle with a $\dbar$-operator.
More precisely, this is the complex of sheaves of $\Ccal^{\infty}_{M_{B}(X)}$-modules 
\begin{displaymath}
	\Dcal_{\Xcal/M_{B}(X)}=\Dcal_{\Xcal/M_{B}(X)}(\Lcal)\colon 0\longrightarrow \Acal^{0,0}_{\Xcal/M_{B}(X)}(\Lcal)\overset{\ov{\partial}_{X}}{\longrightarrow} \Acal^{0,1}_{\Xcal/M_{B}(X)}(\Lcal)\longrightarrow 0.
\end{displaymath}
We have decorated the relative Dolbeault operator $\ov{\partial}_{X}$ with the index $X$ to emphasize the fact that we are in a product situation, and we are only differentiating in the $X$ direction. The cohomology sheaves of the complex  $\pi_{\ast}\Dcal_{\Xcal/M_{B}(X)}$ will be written $\Hcal^{0,p}_{\ov{\partial}_{X}}(\Lcal)$. After \cite[Thm.3.5]{BGS3}, there are canonical isomorphisms of sheaves of $\Ccal^{\infty}_{M_{B}(X)}$-modules
\begin{displaymath}
	\rho_{p}\colon R^{p}\pi_{\ast}(\Lcal)\otimes\Ccal^{\infty}_{M_{B}(X)}\isorightarrow \Hcal^{0,p}_{\ov{\partial}_{X}}(\Lcal).
\end{displaymath}
By Proposition 3.10 of \emph{loc. cit.}, there is a natural holomorphic structure on $\Hcal^{0,p}_{\ov{\partial}_{X}}(\Lcal)$, defined in terms of both the relative and the global Dolbeault complexes of $\Lcal$. For the sake of brevity, we refer to it as the \emph{holomorphic structure of Bismut-Gillet-Soul\'e}. They prove their structure coincides with the holomorphic structure on the coherent sheaves $R^{p}\pi_{\ast}(\Lcal)$, through the isomorphism $\rho_{p}$. Finally, in  \cite[Lemma 3.8]{BGS3}
it is shown that $\pi_{\ast}\Dcal_{\Xcal/M_{B}(X)}$ ($\mathscr{E}^{\bullet}$ in the notation of the cited paper) is a perfect complex in the category of sheaves of $\Ccal^{\infty}_{M_{B}(X)}$-modules. 
As a result, to compute higher direct images and the determinant of cohomology, we can equivalently work with the complex $\pi_{\ast}\Dcal_{\Xcal/M_{B}(X)}$ and the holomorphic structure of Bismut-Gillet-Soul\'e.

Associated to the relative connection on $\Lcal$ and the hermitian metric on $T_{X}$, there are non-self-adjoint Laplace operators $\Delta^{0,p}=(\ov{\partial}_{X}+{\ov\partial}_{X}^{\sharp})^2$ on $\pi_{\ast}\Dcal_{\Xcal/M_{B}(X)}$. Fiberwise, they restrict to the Laplace type operators of Cappell-Miller. We use the notation $\Delta^{0,p}_{\chi}$ for the restriction to the fiber above $\chi$, and similarly for other operators. Let us explicitly describe them. Let $\widetilde{X}$ be the universal cover of $X$, with fundamental group $\Gamma=\pi_{1}(X,p)$ and the complex structure induced from $X$. The Dolbeault complex of $\Lcal_{\chi}$ is isomorphic to the Dolbeault complex 
\begin{displaymath}
	A^{0,0}(\widetilde{X},\chi)\overset{\ov{\partial}}{\longrightarrow} A^{0,1}(\widetilde{X},\chi),
\end{displaymath}
where $A^{0,p}(\widetilde{X},\chi)$ indicates the smooth differential $\chi$-equivariant forms of type $(0,p)$, and $\ov{\partial}$ is the standard Dolbeault operator on functions on $\widetilde{X}$. In the identification, we are implicitly appealing to the canonical trivialization of $\Lcal_{\chi}$ at the base point $p$. The metric on $T_{X}$ induces a metric on $T_{\widetilde{X}}$ and a formal adjoint $\ov{\partial}^{\ast}$, defined as usual in terms of the Hodge $\ast$ operator. Let $D^{0,p}=(\ov{\partial}+\ov{\partial}^{\ast})^{2}$. Then, the Dolbeault complex of $\Lcal_{\chi}$ and $\Delta^{0,\bullet}_{\chi}$ are identified to $(A^{0,\bullet}(\widetilde{X},\chi),\ov{\partial}, D^{0,\bullet})$. To make the holomorphic dependence on $\chi$ explicit, we parametrize $M_{B}(X)$ by $H^{1}_{dR}(X,\CBbb)$, and further identify cohomology classes with harmonic representatives. In particular, let $\nu$ be a harmonic representative for $\chi$. Define the invertible function
\begin{displaymath}
	G_{\nu}(z)=\exp\left(\int_{\tilde{p}}^{z}\nu\right).
\end{displaymath}
We build the isomorphism of complexes
\begin{displaymath}
	\xymatrix{
		A^{0,0}(\widetilde{X},\chi)\ar[r]^{\ov{\partial}}\ar[d]_{G_{\nu}^{-1}\cdot}	&A^{0,1}(\widetilde{X},\chi)\ar[d]^{G_{\nu}^{-1}\cdot}\\
		A^{0,0}(\widetilde{X})^{\Gamma}\ar[r]^{\ov{\partial}-\nu''}		&A^{0,1}(\widetilde{X})^{\Gamma}.
	}
\end{displaymath}
Accordingly, the operators $\ov{\partial}^{\ast}$ and $D^{0,p}$ can be transported to the new complex, through conjugation by $G_{\nu}$. We indicate with an index $\nu$ the new conjugated operators, so that for instance $\ov{\partial}_{\nu}=\ov{\partial}-\nu''$, and similarly for $\ov{\partial}^{\ast}_{\nu}$ and $D^{0,p}_{\nu}$. After all these identifications, we see that $\ov{\partial}^{\sharp}_{\chi}$ will correspond to $\ov{\partial}^{\ast}_{\nu}$ and $\Delta^{0,p}_{\chi}$ will correspond to $D^{0,p}_{\nu}$. 
\begin{lemma}\label{lemma:family-A}
$\hbox{}$
\begin{enumerate}
	\item The operators $D_{\nu}^{0,p}$ form a holomorphic family of type (A) in the sense of Kato \cite[Chap.\ VII, Sec.\ 2]{Kato}: i) they all share the same domain $A^{0,p}(X)$ and are closed with respect to the $L^{2}$ structure induced by the choice of hermitian metric on $T_{X}$ and ii) they depend holomorphically in $\nu$.
	\item The operators $D_{\nu}^{0,p}$ have compact resolvent, and spectrum bounded below and contained in a ``horizontal" parabola.
\end{enumerate}
\end{lemma}
\begin{proof}
For the first item, we notice that the $D_{\nu}^{0,p}$ are differential operators of order $2$ and share the same principal symbol with $D^{0,p}$, hence they are elliptic since the latter is. This also implies that the $D_{\nu}^{0,p}$ are closed as unbounded operators acting on $A^{0,p}(X)$ and with respect to the $L^{2}$ structure. We have thus checked the first condition in Kato's definition. For the holomorphicity, introduce a basis of holomorphic differentials $\lbrace\omega_{i}\rbrace$ of $X$ and write
\begin{displaymath}
	\nu=\sum_{i}(s_{i}\omega_{i}+t_{i}\ov{\omega}_{i}).
\end{displaymath}
The holomorphic dependence on $\nu$ amounts to the holomorphic dependence on the parameter $s_{i},t_{j}$, which is obvious from the construction of $D_{\nu}^{(0,p)}$ by conjugation by $G_{\nu}$: given $\theta\in A^{0,p}(X)$, the differential form $D_{\nu}^{(0,p)}\theta$ is holomorphic in the parameters $s_{i},t_{i}$. This establishes the second condition, so the first claim.

For the compact resolvent property, this is done in \cite{Fay} (especially p.\ 111), where Fay explicitly constructs the Green kernel for $(\Delta^{0,p}_{\chi}-s(1-s))^{-1}$ (see also the remark below). The spectrum assertion is an observation of Cappell-Miller \cite[Lemma 4.1]{CM}.
\end{proof}
\begin{remark}
Actually, for holomorphic families of type (A) in a parameter $\chi$ on a domain, compactness of the resolvent for all $\chi$ follows from the compactness of the resolvent at a given $\chi_{0}$ \cite[Thm.\ 2.4]{Kato}. Therefore, the compactness asserted by the lemma is automatic from the classical compactness in the unitary and self-adjoint case, for instance when $\nu=0$. 
\end{remark}

We now look at a given $\chi_{0}\in M_{B}(X)$. Let $b>0$ not in the spectrum of $\Delta^{0,p}_{\chi_{0}}$. By the Lemma \ref{lemma:family-A} and \cite[Chap.\ VII, Thm.\ 1.7]{Kato}, there exists a neighborhood $U_{\chi_{0}}$ of $\chi_{0}$ such that the same property still holds for $\Delta^{0,p}_{\chi}$, if $\chi\in U_{\chi_{0}}$. Hence, the set $U_{b}$ of those $\chi\in M_{B}(X)$ such that $b$ is not the real part of any generalized eigenvalue of $\Delta^{0,p}_{\chi}$, forms an open set. Because $b>0$, it is easy to see that this open set does not depend on whether we work with $\Delta^{0,0}_{\chi}$ or $\Delta^{0,1}_{\chi}$: it is the same for both. Such open subsets $U_{b}$ form an open cover of $M_{B}(X)$. We define $\mathscr{V}^{0,p}_{b,\chi}\subset A^{0,p}(\Lcal_{\chi})$ the subspace spanned by generalized eigenfunctions of $\Delta_{\chi}^{0,p}$, of generalized eigenvalue $\lambda$ with $\Real(\lambda)<b$. If $c>b>0$ are not the real parts of the eigenvalues at some $\chi_{0}$, we can similarly introduce $\mathscr{V}_{(b,c),\chi}^{0,p}$ on $U_{b}\cap U_{c}$, by consideration of generalized eigenfunctions with eigenvalues whose real part is in the open interval $(b,c)$.
\begin{proposition}
For $\chi\in U_{b}$ (resp. $U_{b}\cap U_{c}$), the vector spaces $\mathscr{V}^{0,p}_{b,\chi}$ (resp. $\mathscr{V}_{(b,c),\chi}^{0,p}$) define a holomorphic vector bundle on $U_{b}$ (resp. $U_{b}\cap U_{c}$) with locally finite ranks.
\end{proposition}
\begin{proof} In view of Lemma \ref{lemma:family-A}, this is a reformulation of \cite[Chap.\ VII, Thm.\ 1.7]{Kato}. The proof of \emph{loc. cit.} provides an illuminating construction by means of spectral projectors.
\end{proof}
We denote by
\begin{displaymath}
	\mathscr{V}^{0,p}_{b}=\mathscr{V}^{0,p}_{b}(\Lcal)\subset \pi_{\ast}\Acal^{0,p}_{\Xcal/M_{B}(X)}(\Lcal)\bigr|_{U_{b}}
\end{displaymath}
the holomorphic bundle on $U_{b}$ thus defined. The differential on the Dolbeault complex $\pi_{\ast}\Dcal_{\Xcal/M_{B}(X)}$ induces a differential on $\mathscr{V}_{b}^{0,p}$, and $\ov{\partial}_{X}(\mathscr{V}^{0,0}_{b})\subset\mathscr{V}^{0,1}_{b}$. Indeed, the relative $\ov{\partial}$ operator of $\Lcal$ commutes with the operators $\Delta^{0,p}_{\chi}$. We introduce similar notation for eigenspaces with real parts in $(b,c)$.
\begin{proposition} $\hbox{}$
\begin{enumerate}	
	\item The inclusion of complexes
	\begin{equation}\label{eq:inclusion-complexes}
		(\mathscr{V}^{0,\bullet}_{b}\otimes\Ccal^{\infty}_{U_b},\ov{\partial}_{X})\hookrightarrow \pi_{\ast}\Dcal_{\Xcal/M_{B}(X)}\bigr|_{U_{b}}
	\end{equation}
	is a quasi-isomorphism. Therefore, the complex $\mathscr{V}^{0,\bullet}_{b}\otimes\Ccal^{\infty}_{U_{b}}$ computes $\Hcal^{0,p}_{\ov{\partial}_{X}}(\Lcal)$ restricted to $U_{b}$.
	\item The cohomology sheaves of $\mathscr{V}^{0,\bullet}_{b}$ have natural structures of coherent sheaves on $U_{b}$, compatible with the holomorphic structures of Bismut-Gillet-Soul\'e on $\Hcal^{0,p}_{\ov{\partial}_{X}}(\Lcal)$. Therefore, the complex $\mathscr{V}^{0,\bullet}_{b}$ computes $R\pi_{\ast}(\Lcal)$ restricted to $U_{b}$.
	\item The complex $\mathscr{V}^{0,\bullet}_{(b,c)}$ is acyclic.
\end{enumerate}
\end{proposition}
\begin{proof}
First, by \cite[Lemma 3.8]{BGS3} we know that the relative Dolbeault complex is perfect as a complex of $\Ccal^{\infty}_{M_{B}(X)}$-modules, and its cohomology is bounded and finitely generated. Second, Cappell-Miller show that \eqref{eq:inclusion-complexes} is fiberwise a quasi-isomorphism \cite[top of p.\ 151]{CM}. Finally, the $\mathscr{V}^{0,\bullet}_{b}\otimes\Ccal_{U_{b}}^{\infty}$ are vector bundles, hence projective objects in the category of sheaves of $\Ccal_{U_{b}}^{\infty}$-modules. The three assertions together are enough to conclude the first assertion.

That the cohomology of $\mathscr{V}^{0,\bullet}_{b}$ is formed by coherent sheaves is immediate, being the cohomology sheaves of a complex of finite rank holomorphic vector bundles. For the compatibility of holomorphic structures, taking into account the construction of Bismut-Gillet-Soul\'e, it is enough to observe the following. Assume $\theta$ is a local holomorphic section of $\mathscr{V}^{0,p}_{b}$. Hence, it depends holomorphically on $\chi$ and $\ov{\partial}_{X}\theta=0$. Because $\Xcal=X\times M_{B}(X)$ is a product, we can assume that $\theta$ is a global $(0,p)$ form, with $\ov{\partial}_{X}\theta=0$ and depending holomorphically on $\chi$. By the very construction of the universal bundle $\Lcal$, this is tantamount to saying $\ov{\partial}_{\Lcal}\theta=0$. Here $\ov{\partial}_{\Lcal}$ is the Dolbeault operator of $\Lcal$ on $\Xcal$.
But now $\ov{\partial}_{\Lcal}\theta=0$ is exactly the condition defining the holomorphic structure of Bismut-Gillet-Soul\'e \cite[p.\ 346]{BGS3} in our case.

The last assertion is left as an easy exercise.
\end{proof}
Let us graphically summarize the proposition with a diagram:
	\begin{equation}\label{eq:diagr-compl-str}
		\xymatrix{
			&	&\Hcal^{p}(\mathscr{V}_{b}^{0,\bullet},\ov{\partial}_{X})\otimes\Ccal_{U_{b}}^{\infty}\ar[d]_{\alpha_{p,b}}^{\begin{sideways}$\sim$\end{sideways}}\ar@{-->}[ld]_{\quad\beta_{p,b}}
			\\
			&R^{p}\pi_{\ast}(\Lcal)\otimes \Ccal_{U_{b}}^{\infty}\ar[r]_{\rho_{p}}^{\hbox{$\sim$}}&\Hcal^{0,p}_{\ov{\partial}_{X}}(\Lcal)\mid_{U_{b}}.
		}
	\end{equation}
	The complex structures on $\Hcal^{0,p}_{\ov{\partial}_{X}}(\Lcal)\bigr|_{U_{b}}$ induced by $\rho_{p}$ and $\alpha_{p,b}$ are compatible by the proposition, and hence $\beta_{p,b}$ is induced by an isomorphism of coherent sheaves. There are corresponding arrows between determinants of cohomologies, that we indicate $\rho$, $\alpha_{b}$ and $\beta_{b}$. In particular, by an abuse of notation the isomorphism $\beta_{b}$ can be identified with an isomorphism of holomorphic line bundles
\begin{displaymath}
	\beta_{b}\colon \det(\mathscr{V}_{b}^{0,\bullet})\isorightarrow\det R\pi_{\ast}(\Lcal)\mid_{U_b}.
\end{displaymath}
Here, we used the canonical isomorphism between the determinant of cohomology of $\mathscr{V}_{b}^{0,\bullet}$ and the determinant of its cohomology. A parallel digression applies to $\Lcal^{c}$, and we use the index $c$ for the corresponding objects. There is also a variant that applies to $\Lcal\otimes\omega_{X}$ and $\Lcal^{c}\otimes\omega_{\ov{X}}$, where we incorporate the Chern connections on $\omega_{X}$ and $\omega_{\ov{X}}$, with respect to the fixed hermitian metric. We leave the details to the reader. We introduce the notations $\mathscr{V}^{0,p}_{b}(\Lcal\otimes\omega_{X})$, etc. when confusions can arise. We now have a fundamental duality phenomenon.
\begin{proposition}
The operator $\ov{\partial}_{X}^{\sharp}$ induces a homological complex of holomorphic vector bundles on $U_{b}$
\begin{displaymath}
	\mathscr{V}^{0,1}_{b}(\Lcal)
	\stackrel{\ov{\partial}^{\sharp}_{X}}{\xrightarrow{\hspace*{.75cm}}}
	\mathscr{V}^{0,0}_{b}(\Lcal).
\end{displaymath}
This complex is $\Ocal_{U_{b}}$-isomorphic (\emph{i.e.} holomorphically) to the cohomological complex
\begin{displaymath}
	\mathscr{V}^{0,0}_{b}((\Lcal^{c})^{\vee}\otimes\omega_{\ov{X}})
	\stackrel{\ov{\partial}^{\sharp}_{\ov X}}{\xrightarrow{\hspace*{.75cm}}}
	\mathscr{V}^{0,1}_{b}((\Lcal^{c})^{\vee}\otimes\omega_{\ov{X}}).
\end{displaymath}
Therefore, there is a canonical isomorphism of holomorphic line bundles
\begin{displaymath}
	\det(\mathscr{V}^{0,\bullet}_{b})
	\stackrel{\beta^{c}_{b}}{\xrightarrow{\hspace*{.75cm}}}
	\det R\pi_{\ast}^{c}((\Lcal^{c})^{\vee}\otimes\omega_{\ov{X}})^{\vee}.
\end{displaymath}
\end{proposition}
\begin{proof}
The first assertion follows because $\ov{\partial}_{X}^{\sharp}$ commutes with $\Delta^{0,p}$. The coincidence
\begin{displaymath}
	\mathscr{V}^{0,1}_{b}=\mathscr{V}^{0,0}_{b}((\Lcal^{c})^{\vee}\otimes\omega_{\ov{X}})
\end{displaymath}
as holomorphic vector bundles is easily seen. Notice the natural appearance of $(\Lcal^{c})^{\vee}$, which has same holonomy characters as $\Lcal$, but the opposite holomorphic structure fiberwise. Observe that the base point and the trivialization of the universal bundles at $p$ is implicit in the identification. Moreover, there is an isomorphism of holomorphic vector bundles given by the Hodge star operator followed by conjugation, that following \cite{CM} we write $\hat{\star}$:
\begin{displaymath}
	\hat{\star}\colon \mathscr{V}^{0,0}_{b}(\Lcal)\isorightarrow\mathscr{V}^{1,0}_{b}((\Lcal^{c})^{\vee}\otimes\omega_{\ov{X}}).
\end{displaymath}
Observe that $\hat{\star}$ is complex linear, and this is necessary if we want to preserve holomorphy. The compatibilities with the differentials are readily checked from the definitions. This concludes the first assertion. For the second, we just need to stress that the determinant of $\mathscr{V}^{0,\bullet}_{b}$ as a cohomological complex is dual to the determinant of $\mathscr{V}^{0,\bullet}_{b}$ as a homological complex.
\end{proof}
\begin{corollary}\label{cor:local-triv-CM}
$\hbox{}$
\enumerate
	\item There is a diagram of isomorphisms of holomorphic line bundles on $U_{b}$
\begin{displaymath}
	\xymatrix{
		\det(\mathscr{V}^{0,\bullet}_{b})\ar[r]^{id}\ar[d]_{\beta_{b}}	&\det(\mathscr{V}^{0,\bullet}_{b})\ar[d]^{\beta^{c}_{b}}\\
		\det R\pi_{\ast}(\Lcal)\ar[r]^{\hbox{$\sim$}\qquad\quad}	&\det R\pi_{\ast}^{c}((\Lcal^{c})^{\vee}\otimes\omega_{\ov{X}})^{\vee}.
	}
\end{displaymath}
	It induces a holomorphic trivialization $\tau(b)$ of $\det R\pi_{\ast}(\Lcal)\otimes \det R\pi_{\ast}^{c}(\Lcal^{c})$ on $U_{b}$.
	\item Let $c>b>0$. On $U_{b}\cap U_{c}$, the relation between $\tau(b)$ and $\tau(c)$ is given by
	\begin{displaymath}
		\tau(b)=\tau(c)\prod_{j=1}^{m}\det\Delta^{0,1}_{(b,c)},
	\end{displaymath}
	where $\Delta^{0,1}_{(b,c)}$ is the endomorphism of the holomorphic vector bundle $\mathscr{V}^{0,1}_{(b,c)}$ defined by the laplacians $\Delta^{0,1}_{\chi}$, $\chi\in U_{b}\cap U_{c}$.
\end{corollary}
\begin{proof}
The first item is a reformulation of the proposition, together with the canonical Serre duality $\det R\pi_{\ast}((\Lcal^{c})^{\vee}\otimes\omega_{\ov{X}})\simeq\det R\pi_{\ast}(\Lcal^{c})$. For the second item, it is enough to check this equality pointwise and use that the determinant of a holomorphic bundle endomorphism is a holomorphic function. The pointwise relation follows from \cite[Eq.\ (3.6)]{CM}.
\end{proof}
\begin{remark}
The holomorphic function $\det\Delta^{0,1}_{(b,c)}$ is to be thought as a trivialization of the holomorphic line bundle $\det\Hcal^{\bullet}(\Vcal^{0,\bullet}_{(b,c)})$.
\end{remark}

For a given $\chi\in U_{b}$ and $b>0$, let us denote $P_{b}$ the spectral projector on generalized eigenfunctions of $\Delta^{0,1}_{\chi}$ of eigenvalues with real part $<b$. We put $Q_{b}=1-P_{b}$, and define the spectral zeta function of $Q_{b}\Delta^{0,1}_{\chi}$, as usual to be the Mellin transform of the heat operator $e^{-tQ_{b}\Delta^{0,1}_{\chi}}$. This depends on the auxiliary choice of an Agmon angle. Let this function be $\zeta_{b,\chi}(s)$. It is a meromorphic function on $\CBbb$, regular at $s=0$. The bases for these definitions and claims are due to Cappell-Miller, and rely on Seeley's methods \cite{Seeley}.  Furthermore, the special value $\exp(\zeta^{\prime}_{b,\chi}(0))$ does not depend on the choice of Agmon angle.

\begin{lemma}
The expression $\exp(\zeta^{\prime}_{b,\chi}(0))$ defines a holomorphic function in $\chi\in U_{b}$.
\end{lemma}
\begin{proof}
We adapt the proof of the smoothness property for unitary $\chi$, in the lines of Bismut-Freed \cite[Sec. g)]{Bismut-Freed-1}. Let us explain the main lines. The holomorphicity is a local property, and hence we can restrict to a small neighborhood $\Omega$ of a fixed $\chi_{0}\in U_{b}$, where a uniform choice of Agmon angle is possible. We then address the holomorphicity of $\zeta^{\prime}_{b,\chi}(0)$ for $\chi\in\Omega$, for this uniform choice of Agmon angle. 

First of all, the operators $\Delta^{0,1}_{\chi}$ define an endomorphism of the \emph{finite rank} holomorphic vector bundle $\mathscr{V}_{b}^{0,p}$ on $\Omega$. Hence, the operators $e^{-t P_{b} \Delta^{0,1}_{b}}$ are obviously trace class and 
\begin{displaymath}
	\tr(e^{-t P_{b} \Delta^{0,1}_{\chi}})
\end{displaymath}
is an entire function both in $t$ and $\chi$. Second, after possibly restricting $\Omega$, we can lift $\chi$ to harmonic representatives $\nu=\nu(\chi)$, depending holomorphically in $\chi$, as in the beginning of this section. Then the operators $\Delta_{\chi}^{0,1}$ are conjugate to the operators $D_{\nu}^{0,1}$ acting on $A^{0,1}(X)$, as in Lemma \ref{lemma:family-A}. These constitute a holomorphic family of differential operators of order 2. They differ from the fixed self-adjoint Dolbeault laplacian $D_{0}^{0,1}$ by differential operators of order 1. In particular, the theory of Seeley \cite{Seeley} and Greiner \cite[Sec.\ 1]{Greiner} applies. From the latter one sees there is an asymptotic expansion as $t\to 0$
\begin{displaymath}
	\tr(e^{-t D^{0,1}_{\nu}})=\sum_{k=0}^{N} t^{-1+k/2}a_{k}(\nu) + o(t^{-1+N/2}),
\end{displaymath}
where the $a_{k}(\nu)$ are holomorphic functions in $\nu$, and the remainder is uniform in $\nu$ (after possibly shrinking $\Omega$). Hence, one concludes with a similar property for $\tr(e^{-t Q_{b} \Delta^{0,1}_{\chi}})$. Finally, for the large time asymptotics, one can adapt the methods of Seeley to show
\begin{displaymath}
	\tr(e^{-t Q_{b} \Delta^{0,1}_{\chi}})=O(e^{-tb}),
\end{displaymath}
with a uniform $O$ term on $\Omega$ (after again possibly restricting). This makes use of the resolvent kernel, as constructed by Seeley. These considerations, combined with the explicit expression for $\zeta^{\prime}_{b,\chi}(0)$ provided by \cite[Thm. 11.1]{CM}, prove the statement of the lemma.
\end{proof}

\begin{proposition}\label{prop:CM-torsion}
Let $c>b>0$. We have an equality of holomorphic sections on $U_{b}\cap U_{c}$ 
\begin{displaymath}
	\tau(b)\exp(-\zeta^{\prime}_{b}(0))=\tau(c)\exp(-\zeta^{\prime}_{c}(0)).
\end{displaymath}
\end{proposition}
Hence, such expressions can be glued into a single holomorphic trivialization $\tau$ of $\det R\pi_{\ast}(\Lcal)\otimes\det R\pi_{\ast}(\Lcal^{c})$ on $M_{B}(X)$.
\begin{proof}
The proof is direct after Corollary \ref{cor:local-triv-CM} and the very definition of the spectral zeta functions.
\end{proof}
The proposition motivates the following terminology.
\begin{definition}
\begin{enumerate}
	\item The holomorphic trivialization $\tau$ of $\lambda(\Lcal)\otimes\lambda(\Lcal^{c})$ defined by Proposition \ref{prop:CM-torsion} is called the holomorphic Cappell-Miller torsion. 
	\item The holomorphic logarithm $\LOG$ of $\lambda(\Lcal)\otimes\lambda(\Lcal^{c})$ attached to the holomorphic Cappell-Miller torsion is called the Cappell-Miller logarithm, and written $\LOG_{CM}$.
\end{enumerate}
\end{definition}
\begin{remark}
	By construction, at a given $\chi$, the section $\tau$ coincides with the construction of Cappell-Miller. To sum up, our task so far has been to establish that the Cappell-Miller construction can be put into holomorphic families.
\end{remark}
We can now state the main theorem of this section.
\begin{theorem}\label{thm:Quillen-CM}
The Quillen and the Cappell-Miller logarithms on $\lambda(\Lcal)\otimes_{\CBbb}\lambda(\Lcal^{c})$ coincide.
\end{theorem}
\begin{proof}
First of all, both logarithms are holomorphic. Second, by construction of the Cappell-Miller torsion and the Quillen logarithm, both coincide on the unitary locus of $M_{B}(X)$, which is a maximal totally real subvariety. Then by a standard argument, they coincide on all of $M_{B}(X)$.
\end{proof}
\begin{remark}
\begin{enumerate}
	\item A consequence of the theorem, together with Theorem \ref{theorem:Iso-Deligne} and Corollary \ref{corollary:iso-Deligne}, is that in dimension 1 and rank 1, the Cappell-Miller torsion enjoys of analogous properties to the holomorphic analytic torsion of Bismut-Gillet-Soul\'e, regarding the Riemann-Roch formula. This answers affirmatively a question of these authors.
	\item From now on, we will refer to $\LOG_{Q}$ as the Quillen-Cappell-Miller logarithm.
\end{enumerate}
\end{remark}

\section{Arithmetic Intersection Theory for Flat Line Bundles}\label{section:AIT}

\subsection{Conjugate pairs of line bundles with logarithms on $\Spec\Ocal_K$}
Let $K$ be a number field with ring of integers $\Ocal_{K}$. We write $S=\Spec\Ocal_K$. An invertible sheaf (or line bundle) $\Lcal$ over $S$ can be equivalently seen as a projective $\Ocal_K$ module of rank 1. We will not make any distinction between both points of view, in order to ease notations. This particularly concerns base change and tensor product.

\begin{definition}
A \emph{conjugate pair of line bundles with logarithms}, or simply a conjugate pair, on $S$ consists in the following data:
\begin{enumerate}
	\item a pair of line bundles $\Lcal$ and $\Lcal^{c}$ over $S$;
	\item for every embedding $\tau\colon K\hookrightarrow\CBbb$, a logarithm $\LOG_{\tau}$ on the one dimensional complex vector space $\Lcal_{\tau}\otimes_{\CBbb}\Lcal^{c}_{\overline{\tau}}$. 
\end{enumerate}
We introduce the notation $\Lcal^{\sharp}$ for the data $(\Lcal,\Lcal^{c},\lbrace\LOG_{\tau}\rbrace_{\tau\colon K\hookrightarrow\CBbb})$.
\end{definition}
Given conjugate pairs $\Lcal^{\sharp}$ and $\Mcal^{\sharp}$, an isomorphism $\varphi^{\sharp}:\Lcal^{\sharp}\rightarrow\Mcal^{\sharp}$ is a pair $(\varphi,\varphi^{c})$ of isomorphisms, $\varphi\colon\Lcal\rightarrow\Mcal$ and $\varphi^{c}\colon\Lcal^{c}\rightarrow\Mcal^{c}$, such that for every $\tau\colon K\hookrightarrow\CBbb$, $\varphi_{\tau}\otimes\varphi^{c}_{\overline{\tau}}$ preserves logarithms. There are standard constructions on conjugate pairs with logarithms, notably tensor product and duality. 
\begin{definition}
The groupoid of conjugate pairs of line bundles with logarithms, denoted $\PIC^{\sharp}(S)$, is defined by:
\begin{enumerate}
	\item[\textbullet]\emph{objects}: conjugate pairs of line bundles with logarithms;
	\item[\textbullet]\emph{morphisms}: isomorphisms of pairs of line bundles with logarithms.
\end{enumerate}
It has the structure of a Picard category. The group of isomorphisms classes of objects is denoted by $\Pic^{\sharp}(S)$ and is called the \emph{arithmetic Picard group of conjugate pairs of line bundles with logarithms}.
\end{definition}

\noindent\textbf{Arithmetic degree.} We proceed to construct an \emph{arithmetic degree map} on $\Pic^{\sharp}(S)$,
\begin{displaymath}
	\deg^{\sharp}:\Pic^{\sharp}(S)\longrightarrow\CBbb/\pi i\,  \ZBbb.
\end{displaymath}
We emphasize that the target group is not $\CBbb/2\pi i\, \ZBbb$, but $\CBbb/\pi i\,  \ZBbb$. Let $\Lcal^{\sharp}$ be a conjugate pair. Given nonvanishing elements $\ell\in\Lcal_{K}$, $\ell^{c}\in\Lcal_{K}^{c}$, the quantity
\begin{displaymath}
	\sum_{\pfrak}\ord_{\pfrak}(\ell\otimes\ell^{c})\log(N\pfrak)-\sum_{\tau:K\hookrightarrow\CBbb}\LOG_{\tau}(\ell_{\tau}\otimes\ell^{c}_{\overline{\tau}})
\end{displaymath}
taken in $\CBbb/\pi i\,  \ZBbb$ does not depend on the choices $\ell$, $\ell^{c}$. Indeed, for $\lambda,\mu\in\ K^{\times}$, the following relations hold in $\CBbb/\pi i\,  \ZBbb$:
 \begin{align}
 	\begin{split}
	\sum_{\pfrak}\ord_{\pfrak}(\lambda\mu)\log(N\pfrak)&-\sum_{\tau:K\hookrightarrow\CBbb}\log(\tau(\lambda)\overline{\tau}(\mu))=\\
	&-\log\left(\prod_{\pfrak}|\lambda|_{\pfrak}\prod_{\tau:K\hookrightarrow\CBbb}\tau(\lambda)\right)
	-\log\left(\prod_{\pfrak}|\mu|_{\pfrak}\prod_{\tau:K\hookrightarrow\CBbb}\overline{\tau}(\mu)\right)\\
	&=-\log(\pm 1)-\log(\pm 1)=0\ .
	\end{split}
\end{align}
 We then conclude by the very definition of logarithm: modulo $2\pi i\, \ZBbb$, and hence modulo $\pi i\,   \ZBbb$, $\LOG_{\tau}$ satisfies
 \begin{displaymath}
 	\LOG_{\tau}((\lambda\ell)_{\tau}\otimes(\mu\ell^{c})_{\overline{\tau}})=\LOG_{\tau}(\tau(\lambda)\overline{\tau}(\mu)\ell_{\tau}\otimes\ell^{c}_{\overline{\tau}})=\log(\tau(\lambda)\overline{\tau}(\mu))+\LOG_{\tau}(\ell_{\tau}\otimes\ell^{c}_{\overline{\tau}}).
 \end{displaymath}

 \begin{remark}
 \begin{enumerate}
 	\item When the field $K$ cannot be embedded into $\RBbb$, the arithmetic degree is well defined in $\CBbb/2\pi i\, \ZBbb$, and the argument in $\RBbb/2\pi\ZBbb$.
	\item In general, to obtain an arithmetic degree with values in $\CBbb/2\pi i\, \ZBbb$, one needs to add to conjugate pairs a positivity condition at real places (or equivalently, an orientation). However, our main goal is to prove an arithmetic Riemann-Roch formula, which relies on the Deligne isomorphism through Theorem \ref{theorem:Iso-Deligne}. As we point out in Remark \ref{remark:iso-Deligne}, this introduces a $\log(\pm 1)$ ambiguity. This is why we do not impose any positivity conditions in this article.
 \end{enumerate}
 \end{remark}
 \begin{example}
 Because a $\ZBbb$ module of rank 1 admits a basis, which is unique up to sign, one proves with ease that the arithmetic degree on $\Pic^{\sharp}(\Spec\ZBbb)$ is an isomorphism:
 \begin{displaymath}
 	\deg^{\sharp}\colon\Pic^{\sharp}(\Spec\ZBbb)\isorightarrow\CBbb/\pi i\,  \ZBbb.
 \end{displaymath}
 \end{example}

 We will need the following functorialities for the Picard groups and the arithmetic degree.
 \begin{proposition}\label{proposition:functorialities-deg}
 Let $F$ be a finite extension of $K$ and put $\Tcal=\Spec\Ocal_{F}$. With respect to the morphism $\pi\colon\Tcal\to S$, the arithmetic Picard groups satisfy covariant and contravariant functorialities:
 \begin{enumerate}
 	\item (Inverse images or pull-backs) Tensor product with $\Ocal_{F}$ induces a morphism
	\begin{displaymath}
		\pi^{\ast}\colon\Pic^{\sharp}( S)\longrightarrow\Pic^{\sharp}(\Tcal).
	\end{displaymath}
	\item (Direct images or push-forwards) The norm down to $\Ocal_{K}$ of a projective $\Ocal_{F}$-module induces a morphism
	\begin{displaymath}
		\pi_{\ast}\colon\Pic^{\sharp}( S)\longrightarrow\Pic^{\sharp}(\Tcal).
	\end{displaymath}
	The arithmetic degree on $\Pic^{\sharp}(\Ocal_{K})$ factors through the push-forward to $\Pic^{\sharp}(\ZBbb)$.
	\item The composition $\pi_{\ast}\pi^{\ast}$ acts as multiplication by $[F:K]$.
 \end{enumerate}
 \end{proposition}
 \begin{proof}
 The proof is elementary.
 \end{proof}
 
 
 \subsection{Conjugate pairs of line bundles with connections}
For the rest of this section, we fix a square root of $-1$, $i=\sqrt{-1}\in\CBbb$. Let $\Xcal\rightarrow S$ be an arithmetic surface. By this we mean a regular, irreducible and flat projective scheme over $ S$, with geometrically connected generic fiber $\Xcal_{K}$ of dimension 1. We fix some conventions on complex structures.\\

\paragraph{\textbf{Conventions on complex structures}}
\begin{enumerate}
\item Given an embedding $\tau:K\hookrightarrow\CBbb$, we write $\Xcal_{\tau}$ for the base change of $\Xcal$ to $\CBbb$ through $\tau$. After the choice we made of $\sqrt{-1}$, the set of complex points $\Xcal_{\tau}(\CBbb)$ has a complex structure and is thus a Riemann surface. We call this complex structure the \emph{natural} one. The other complex structure (corresponding to $-i$) is called the \emph{reverse}, \emph{opposite} or \emph{conjugate} one, and as usual we indicate this with a bar: $\overline{\Xcal_{\tau}(\CBbb)}$. With these notations, if $\tau$ is a complex, nonreal, embedding, then $\Xcal_{\overline{\tau}}(\CBbb)$ is canonically biholomorphic to $\overline{\Xcal_{\tau}(\CBbb)}$. 

\item If $\tau$ is a real embedding, we put $\Xcal_{\overline{\tau}}(\CBbb)=\overline{\Xcal_{\tau}(\CBbb)}$ (although $\tau=\overline{\tau}$!). For the natural complex structure on $\Xcal_{\overline{\tau}}(\CBbb)$ we then mean the reverse structure on $\Xcal_{\tau}(\CBbb)$.

\item The same conventions will apply to holomorphic line bundles, and sections of such, over $\Xcal$. For instance, if $\Lcal$ is a line bundle over $\Xcal$ and $\tau$ is a complex, nonreal, embedding, the holomorphic line bundles $\Lcal_{\overline{\tau}}$ on $\Xcal_{\overline{\tau}}(\CBbb)$ and $\overline{\Lcal}_{\tau}$ on $\overline{\Xcal_{\tau}(\CBbb)}$ can be identified, after the identification of $\Xcal_{\overline{\tau}}(\CBbb)$ with $\overline{\Xcal_{\tau}(\CBbb)}$. If $\tau$ is real, then the convention is that $\Lcal_{\overline{\tau}}=\overline{\Lcal}_{\tau}$ on $\Xcal_{\overline{\tau}}(\CBbb)=\overline{\Xcal_{\tau}(\CBbb)}$.
\end{enumerate}
\begin{definition} \label{def:conjugate-pair}
A \emph{conjugate pair of line bundles with connections} on $\Xcal$ consists in the following data:
\begin{enumerate}
	\item two line bundles $\Lcal,\Lcal^{c}$ on $\Xcal$;
	\item holomorphic connections $\nabla_{\tau}$ on the holomorphic line bundles $\Lcal_{\tau}$, with respect to the natural complex structure on $\Xcal_{\tau}(\CBbb)$;
	\item holomorphic connections $\nabla_{\overline{\tau}}^{c}$ on the holomorphic line bundles $\Lcal_{\overline{\tau}}$, with respect to the natural complex structure on $\Xcal_{\overline{\tau}}(\CBbb)$. Observe that by the previous conventions, if $\tau$ is a real embedding, then $\nabla_{\overline{\tau}}^{c}$ is a holomorphic connection on the holomorphic line bundle $\overline{\Lcal}_{\tau}^{c}$ on $\overline{\Xcal_{\tau}(\CBbb)}$.
	\item we impose the following relation: if $\chi_{\tau}$ is the holonomy character of $\pi_{1}(\Xcal_{\tau}(\CBbb),\ast)$ associated to $(\Lcal_{\tau},\nabla_{\tau})$, and $\chi_{\overline{\tau}}^{c}$ is the character associated to $(\Lcal_{\overline{\tau}}^{c},\nabla_{\overline{\tau}}^{c})$, then $\chi_{\overline{\tau}}^{c}=\chi_{\tau}^{-1}$.
\end{enumerate}
We introduce the notation $\Lcal^{\sharp}=((\Lcal,\nabla),(\Lcal^{c},\nabla^{c}))$, with $\nabla=\lbrace\nabla_{\tau}\rbrace_{\tau}$, $\nabla^{c}=\lbrace\nabla_{\tau}^{c}\rbrace_{\tau}$. 
\end{definition}
\begin{remark}
In the definition we do not impose any relationship between $\chi_{\tau}$ and $\chi_{\overline{\tau}}$, in contrast to classical Arakelov geometry. Moreover, we required $\chi_{\overline{\tau}}^{c}=\chi_{\tau}^{-1}$, and not $\chi_{\overline{\tau}}^{c}=\overline{\chi}_{\tau}$. The latter condition only happens in the unitary case, which is the range of application of classical Arakelov geometry.
\end{remark}
There is an obvious notion of isomorphism of conjugate pairs of line bundles with connections. There are also standard operations that can be performed, such as tensor products and duals. Base change is possible as well, for instance by unramified extensions of $K$ (in order to preserve the regularity assumption for arithmetic surfaces).
\begin{definition}
We denote by $\PIC^{\sharp}(\Xcal)$ the groupoid of conjugate pairs of line bundles with connections. It is a Picard category. The group of isomorphism classes is denoted $\Pic^{\sharp}(\Xcal)$ and is called the \emph{Picard group of conjugate pairs of line bundles with connections}.
\end{definition}
Let us now suppose there is a section $\sigma: S\rightarrow\Xcal$. A rigidification along $\sigma$ of a conjugate pair of line bundles with connections $\Lcal^{\sharp}$, is a choice of isomorphisms $\sigma^{\ast}\Lcal\isorightarrow\Ocal_{ S}$ and $\sigma^{\ast}\Lcal^{c}\isorightarrow\Ocal_{ S}$. The previous definitions have obvious counterparts in this setting.
\begin{definition}
Given a section $\sigma: S\rightarrow\Xcal$, we denote by $\PICRIG^{\sharp}(\Xcal,\sigma)$ the groupoid of conjugate pairs of line bundles with connections, rigidified along $\sigma$.
\end{definition}
\begin{remark}\label{remark:rigid}
\begin{enumerate}
	\item Observe that a rigidification of $\Lcal^{\sharp}$ induces rigidifications of $\Lcal_{\tau}$ at $\sigma_{\tau}$ and $\Lcal^{c}_{\overline{\tau}}$ at $\sigma_{\overline{\tau}}$, for $\tau\colon K\hookrightarrow\CBbb$.
	\item A rigidification is unique up to $\Ocal_{K}^{\times}$. Because the norm down to $\QBbb$ of a unit is $\pm 1$, the arithmetic degree is not sensitive to the particular choice of rigidification. 
	\item The Hilbert class field $H$ of $K$ is the maximal unramified abelian extension of $K$, and has the property that any invertible $\Ocal_{K}$-module  becomes trivial after base change to $\Ocal_{H}$. Therefore, after possibly extending the base field to $H$, a rigidification always exists. 
\end{enumerate}
\end{remark}
\paragraph{\textbf{Arithmetic intersection product}} The Deligne pairing and the intersection logarithm constructions allow to define a symmetric bilinear pairing
\begin{displaymath}
 	\PIC^{\sharp}(\Xcal)\times\PIC^{\sharp}(\Xcal)\longrightarrow\PIC^{\sharp}( S).
\end{displaymath}
The construction works as follows. Let $\Lcal^{\sharp}$ and $\Mcal^{\sharp}$ be conjugate pairs of line bundles with connections. We consider the Deligne pairings
$\langle\Lcal,\Mcal\rangle$, $\langle\Lcal^{c},\Mcal^{c}\rangle$.
For every complex embedding $\tau\colon K\hookrightarrow\CBbb$, 
\begin{displaymath}
	\langle\Lcal,\Mcal\rangle_{\tau}\otimes_{\CBbb}\langle\Lcal^{c},\Mcal^{c}\rangle_{\overline{\tau}}
	=\langle\Lcal_{\tau},\Mcal_{\tau}\rangle\otimes_{\CBbb}\langle\Lcal^{c}_{\overline{\tau}},\Mcal^{c}_{\overline{\tau}}\rangle
\end{displaymath}
carries an intersection logarithm $\LOG_{int,\tau}$, build up from the connections defining $\Lcal^{\sharp}$, $\Mcal^{\sharp}$ and intermediate choices of rigidifications (we proved the construction is independent of these choices). We obtain this way a conjugate pair of line bundle with logarithms on $ S$, that we denote $\langle\Lcal^{\sharp},\Mcal^{\sharp}\rangle$. The bilinearity of this pairing is clear, and the symmetry is a consequence of Proposition \ref{prop:int-log-sym}. In terms of this pairing, the \emph{arithmetic intersection product of} $\Lcal^{\sharp}$ and $\Mcal^{\sharp}$ is obtained by taking the arithmetic degree:
\begin{displaymath}
	(\Lcal^{\sharp},\Mcal^{\sharp})=\deg^{\sharp}\langle\Lcal^{\sharp},\Mcal^{\sharp}\rangle\in\CBbb/\pi i\,  \ZBbb.
\end{displaymath}
One of the aims of this section is to prove an arithmetic Riemann-Roch formula that accounts for these arithmetic intersection numbers.\\

\paragraph{\textbf{Argument of the Deligne pairing}} Let $\Lcal^{\sharp}$ and $\Mcal^{\sharp}$ be conjugate pairs of line bundles with connections. By the \emph{argument of the Deligne pairing of $\Lcal^{\sharp}$ and $\Mcal^{\sharp}$} we mean the imaginary part of the intersection product:
\begin{displaymath}
	\arg^{\sharp}\langle\Lcal^{\sharp},\Mcal^{\sharp}\rangle=\Imag(\Lcal^{\sharp},\Mcal^{\sharp})\in\RBbb/\pi\ZBbb.
\end{displaymath} 

\subsection{Mixed arithmetic intersection products}\label{subsec:mixed-prod}
The classical arithmetic Picard group in Arakelov geometry classifies smooth hermitian line bundles, and is denoted $\widehat{\Pic}(\Xcal)$. There is an obvious groupoid version that we denote $\widehat{\PIC}(\Xcal)$. We constructed intersection logarithms between conjugate pairs of rigidified line bundles with connections and hermitian line bundles. With this, we can define a pairing
\begin{displaymath}
	\PICRIG^{\sharp}(\Xcal)\times\widehat{\PIC}(\Xcal)\longrightarrow\PIC^{\sharp}( S)
\end{displaymath}
simply as follows. Given a conjugate pair of line bundles with connections $\Lcal^{\sharp}$, rigidified along $\sigma$, and a hermitian line bundle $\overline{\Mcal}$ on $\Xcal$, we define the Deligne pairing
\begin{displaymath}
	\langle\Lcal^{\sharp},\overline{\Mcal}\rangle
	=(\langle\Lcal,\Mcal\rangle, \langle\Lcal^{c},\Mcal\rangle,\lbrace\LOG_{int,\tau}\rbrace_{\tau}).
\end{displaymath}
We denoted $\LOG_{int,\tau}$ the intersection logarithm on the base change
\begin{displaymath}
	\langle\Lcal,\Mcal\rangle_{\tau}\otimes_{\CBbb}\langle\Lcal^{c},\Mcal\rangle_{\overline{\tau}}
	=\langle\Lcal_{\tau},\Mcal_{\tau}\rangle\otimes_{\CBbb}\langle\Lcal^{c}_{\overline{\tau}},\Mcal_{\overline{\tau}}\rangle,
\end{displaymath}
build up using the connections defining $\Lcal^{\sharp}$ at $\tau$, the rigidifications, and the hermitian metric on $\Mcal$. In terms of this Deligne pairing, we define the mixed arithmetic intersection product
\begin{displaymath}
	(\Lcal^{\sharp},\overline{\Mcal})=\deg^{\sharp}\langle\Lcal^{\sharp},\overline{\Mcal}\rangle\in\CBbb/\pi i\,  \ZBbb.
\end{displaymath}
Because a rigidification is unique up to $\Ocal_{K}^{\times}$, this quantity does not depend on the particular choice of rigidification, but in general it depends on the section.\\

\paragraph{\textbf{Variant in the absence of rigidification}} When a section $\sigma$ is given, but we do not have a rigidification, we may follow the observation made in Remark \ref{remark:rigid} and base change to the Hilbert class field $H$. Observe the base change $\Xcal_{\Ocal_{H}}$ is still an arithmetic surface: because the Hilbert class field $H$ is unramified, the regularity of the scheme is preserved. Let us indicate base changed objects with a prime symbol. Given $\Lcal^{\sharp}$, the base change $\Lcal^{\sharp\prime}$ admits a rigidification, which is unique up to unit. Then, the arithmetic intersection number
\begin{displaymath}
	(\Lcal^{\sharp\prime},\ov{\Mcal}^{\prime})\in\CBbb/\pi i\,  \ZBbb
\end{displaymath}
is defined. Taking into account the functoriality properties of the arithmetic degree (Proposition \ref{proposition:functorialities-deg}), it is more natural to normalize this quantity by $[H:K]$, that is the class number $h_{K}$. We then write
\begin{displaymath}
	(\Lcal^{\sharp},\ov{\Mcal}):=\frac{1}{h_{K}}	(\Lcal^{\sharp\prime},\ov{\Mcal}^{\prime})\in\CBbb/\pi i\,  \ZBbb[1/h_{K}].
\end{displaymath}
In particular, when $K=\QBbb$, or more generally when $h_K=1$, the mixed arithmetic intersection number with values in $\CBbb/\pi i\,  \ZBbb$ is always defined, without any reference to the rigidification (but always depending on the section).

\subsection{Variants over $\RBbb$ and $\CBbb$, argument and periods}\label{subsection:argument} While classical Arakelov geometry over $\RBbb$ or $\CBbb$ cannot produce any interesting numerical invariants (only zero), the present theory  has a nontrivial content over these fields. Let us discuss the case of the base field $\CBbb$. We saw we can still define $\Pic^{\sharp}(\Spec\CBbb)$, and an arithmetic degree $\deg^{\sharp}$, now with values in $i\RBbb/2\pi i\,  \ZBbb$. In the construction, one has to take into account the identity and conjugation embeddings $\CBbb\rightarrow\CBbb$. We denote the imaginary part of $\deg^{\sharp}$ by $\arg^{\sharp}$:
\begin{displaymath}
	\arg^{\sharp}\colon\Pic^{\sharp}(\Spec\CBbb)\longrightarrow\RBbb/2\pi\ZBbb.
\end{displaymath}
Let $X$ be a smooth, proper and geometrically irreducible curve over $\CBbb$. We can also define $\PIC^{\sharp}(X)$ and a Deligne pairing. The argument of the Deligne pairing is still defined:
\begin{displaymath}
	\arg^{\sharp}\langle\Lcal^{\sharp},\Mcal^{\sharp}\rangle\in\RBbb/2\pi\ZBbb.
\end{displaymath}
Similarly there is a well-defined argument of the mixed arithmetic intersection product, between $\PICRIG^{\sharp}(X)$ and $\widehat{\PIC}(X)$.\\

\paragraph{\textbf{Interpretation of the argument}}
Let $X$ be a smooth, projective and irreducible curve over $\CBbb$. To apply the formalism above, we stress $\CBbb$ has to be considered with its identity and conjugation embeddings. Let $L$ be a line bundle on $X$ and $\overline{L}$ the conjugate line bundle on $\overline{X}$. We suppose given holomorphic connections $\nabla_{L}\colon L\rightarrow L\otimes\Omega_{X/\CBbb}^{1}$ and $\nabla_{\overline{L}}\colon\overline{L}\rightarrow\overline{L}\otimes\Omega_{\overline{X}/\CBbb}^{1}$ with \emph{real} holonomy characters. We do not impose any further condition. We choose $L^{c}=L^{\vee}$, and we endow $L^{c}$ and $\overline{L}^{c}$ with the dual connections to $\nabla_{L}$, $\nabla_{\overline{L}}$. This provides an example of conjugate pair of line bundles with connections on $X$, that we write $L^{\sharp}$. Let $M$ be a degree 0 line bundle on $X$, that we endow with its unitary connection. On $\overline{M}$ we put the conjugate connection. In this case we take $M^{c}=M$, with same connections. We proceed to describe
\begin{displaymath}
	\arg^{\sharp}\langle L^{\sharp},M^{\sharp}\rangle\in\RBbb/2\pi\ZBbb.
\end{displaymath} 
We fix a base point $p\in X$ and a trivialization of $L$. Let $\ell$ and $m$ be rational sections of $L$ and $M$. Using the connection $\nabla_{L}$, we lift as usual $\ell$ to $\tilde{\ell}$, on the universal covering. We also lift $\Div m$ to $\widetilde{\Div m}$. For the conjugate datum, we lift $\overline{\ell}$ to $\tilde{\overline{\ell}}$ and $\Div\overline{m}$ to $\widetilde{\Div\overline{m}}$. We will appeal to the explicit description of the intersection logarithm in Section \ref{subsec:mixed-case}, in particular formula \eqref{eq:202}. Because we didn't impose any relation between $\overline{\nabla}_{L}$ and $\nabla_{\overline{L}}$, we cannot conclude with
\begin{displaymath}
	\overline{\tilde{\ell}(\widetilde{\Div m})}=\tilde{\overline{\ell}}(\widetilde{\Div\overline{m}}).
\end{displaymath}
In words, in general ``conjugation does not commute with tilde''. There exists a holomorphic differential form on $\overline{X}$, that we present as $\overline{\theta}'$ for some holomorphic form $\theta'$ on $X$, such that
$\nabla_{\overline{L}}=\overline{\nabla}_{L}+\overline{\theta}'$.
Because both connections are supposed to have real holonomy characters, we see that
\begin{displaymath}
	\exp\left(\int_{\gamma}\overline{\theta}'\right)=\exp\left(\int_{\gamma}\theta'\right).
\end{displaymath}
Hence, the harmonic differential form $\theta=\theta'-\overline{\theta}'$ has periods in $2\pi i\, \ZBbb$. Such differential forms are of course parametrized by $H^{1}(X,2\pi i\, \ZBbb)$, which is a rank $2g$ $\ZBbb$-module. In terms of $\overline{\theta}'$ we have
\begin{displaymath}
	\overline{\tilde{\ell}(\widetilde{\Div m})}=\tilde{\overline{\ell}}(\widetilde{\Div\overline{m}})\exp\left(\int_{\tilde{p}}^{\widetilde{\Div m}}\overline{\theta}'\right).
\end{displaymath}
From this and equation \eqref{eq:202}, we conclude that
\begin{displaymath}
	\arg^{\sharp}\langle L^{\sharp},M^{\sharp}\rangle=-2\Imag\left(\int_{\tilde{p}}^{\widetilde{\Div m}}\overline{\theta}'\right)=\Imag\left(\int_{\tilde{p}}^{\widetilde{\Div m}}\theta\right).
\end{displaymath}
Because $\theta$ has periods in $2\pi i\, \ZBbb$, this quantity does not depend on the choice of lifting $\widetilde{\Div m}$, modulo $2\pi\ZBbb$. Moreover, modulo $2\pi\ZBbb$ it only depends on the rational equivalence class of $\Div m$, namely $M$ itself. And this is again because $\theta$ has periods in $2\pi i\, \ZBbb$. It is also independent of the base point, because $M$ has degree 0. Finally, the connection on $M$ played no role. This is of course in agreement with the properties of the intersection pairings. Therefore, given a degree 0 Weil divisor $D$ on $X$, we have a well defined argument
\begin{displaymath}
	\arg^{\sharp}\langle L^{\sharp},\Ocal(D)\rangle=\Imag\left(\int_{\tilde{p}}^{D}\theta\right)\in\RBbb/2\pi\ZBbb.
\end{displaymath}
Let us write $\theta_{L^\sharp}$ for the harmonic differential form above. We thus have a pairing
\begin{displaymath}
	\begin{split}
		\arg^{\sharp}\colon\PIC^{\sharp}(X)_{re}\times\Pic^{0}(X)(\CBbb)&\longrightarrow\RBbb/2\pi\ZBbb\\
		(L^{\sharp},\Ocal(D))&\longmapsto\arg^{\sharp}\langle L^{\sharp},\Ocal(D)\rangle=\Imag\left(\int_{\tilde{p}}^{D}\theta_{L^{\sharp}}\right),
	\end{split}
\end{displaymath}
where the subscript \emph{re} indicates we restrict to conjugate pairs with real holonomy connections. The values of this pairing are imaginary parts of integer combinations of periods!

There is a variant of this pairing when $M=\Ocal(D)$ has arbitrary degree. In this case one needs to equip $L^{\sharp}$ with a rigidification. Because $L^{c}=L^{\vee}$, it is enough to fix a rigidification for $L$. For the argument, one needs to fix a hermitian metric on $M$ and use the mixed intersection pairing. The final formula looks exactly the same. While the result will not depend on the metric on $M$, it depends on the base point (since $\deg D\neq 0$). If we had chosen unrelated rigidifications for $L$ and $L^{c}$, the result would have depended on these choices, as well.

\begin{remark}
There is no simple formula for the general case of an arbitrary conjugate pair $L^{\sharp}$. 
\end{remark}

\subsection{Arithmetic Riemann-Roch theorem}
Let $\pi\colon\Xcal\rightarrow\Spec\Ocal_K$ be an arithmetic surface with a section $\sigma\colon S\rightarrow\Xcal$. Let $\Lcal^{\sharp}$ be a rigidified pair of conjugate line bundles with connections. Recall the notation $\lambda(\Lcal)$ for $\det R\pi_{\ast}(\Lcal)$. It is compatible with base change. Following the construction of Section \ref{section:LOG-det-coh}, for every $\tau$ there is a Quillen-Cappell-Miller logarithm $\LOG_{Q,\tau}$ on
\begin{displaymath}
	\lambda(\Lcal_{\tau})\otimes_{\CBbb}\lambda(\Lcal^{c}_{\overline{\tau}})=\det H^{\bullet}(\Xcal_{\tau}(\CBbb),\Lcal_{\tau})\otimes\det H^{\bullet}(\Xcal_{\overline{\tau}}(\CBbb),\Lcal^{c}_{\overline{\tau}}).
\end{displaymath}
We introduce the conjugate pair of line bundles with logarithms on $ S$
\begin{displaymath}
	\lambda(\Lcal^{\sharp})_{Q}=(\lambda(\Lcal),\lambda(\Lcal^{c}),\lbrace\LOG_{Q,\tau}\rbrace_{\tau}).
\end{displaymath}
Notice the construction of the Quillen-Cappell-Miller logarithm requires the rigidification, in order to identify $\Lcal_{\tau}$ to $\Lcal_{\chi_{\tau}}$ and $\Lcal_{\overline{\tau}}^{c}$ to $\Lcal^{c}_{\chi_{\tau}}$.
\begin{theorem}\label{theorem:arithmetic-RR}
Let us endow the relative dualizing sheaf $\omega_{\Xcal/S}$ with a smooth hermitian metric. Let $\Lcal^{\sharp}$ be a rigidified conjugate pair of line bundles with connections. There is an equality in $\CBbb/\pi i\,  \ZBbb$
\begin{align}
	\begin{split} \label{eq:ARR-flat}
	12\deg^{\sharp}\lambda(\Lcal^{\sharp})_{Q}-2\delta
	&=2(\overline{\omega}_{\Xcal/S},\overline{\omega}_{\Xcal/S})+6(\Lcal^{\sharp},\Lcal^{\sharp})
	-6(\Lcal^{\sharp},\overline{\omega}_{\Xcal/S})
\\
	&	-(4g-4)[K:\QBbb]
	\left(\frac{\zeta'(-1)}{\zeta(-1)}+\frac{1}{2}\right),
	\end{split}
\end{align}
where $\delta=\sum_{\pfrak}n_{\pfrak}\log (N\pfrak)$ is the ``Artin conductor" measuring the bad reduction of $\Xcal\rightarrow\Spec\Ocal_K$. If $K$ does not admit any real embeddings, then the equality already holds in $\CBbb/2\pi i\, \ZBbb$.
\end{theorem}
\begin{remark}
The mixed arithmetic intersection product $(\Lcal^{\sharp},\overline{\omega}_{\Xcal/S})$ involves the rigidification, and depends on it. This is in agreement with the dependence of the Quillen logarithm on the rigidification. Nevertheless, it does not depend on the choice of metric on $\omega_{\Xcal/S}$, by Lemma \ref{lemma:chern-on-M}. Therefore, on the right hand side of the formula, the dependence in the metric on $\omega_{\Xcal/S}$ comes from $(\overline{\omega}_{\Xcal/S},\overline{\omega}_{\Xcal/S})$.
\end{remark}
\begin{proof}[Proof of Theorem \ref{theorem:arithmetic-RR}]
The theorem is derived as a combination of the following statements:
\begin{enumerate}
	\item the Deligne isomorphism applied to $\Xcal\rightarrow S$, $\Lcal$, $\Lcal^{c}$ and $\Ocal_{\Xcal}$, and its compatibility to base change under $\tau:K\hookrightarrow \CBbb$;
	\item the arithmetic Riemann-Roch theorem of Gillet-Soul\'e \cite{GS:ARR} applied twice to $\Ocal_{\Xcal}$ in Deligne's functorial formulation \cite{Deligne:87, Soule}, which guarantees a quasi-isometry
	\begin{displaymath}
		\lambda(\Ocal_{\Xcal})_{Q}^{\otimes 12}\otimes\Ocal(-\Delta)\isorightarrow
		\langle\overline{\omega}_{\Xcal/ S},\overline{\omega}_{\Xcal/ S}\rangle,
	\end{displaymath}
	with norm $\exp((2g-2)(\zeta'(-1)/\zeta(-1)+1/2))$. The index $Q$ stands for the Quillen metric (for the trivial hermitian line bundle in this case), $\Delta$ is the so-called Deligne discriminant supported on finite primes, and $\Ocal(\Delta)$ is endowed with the trivial metric (then $\delta$ is the arithmetic degree of $\Ocal(\Delta)$). It is related to Artin's conductor through work of T. Saito \cite{Saito};
	\item the fact that our definition of $\LOG_Q$ for the trivial hermitian line bundle amounts to the Quillen metric;
	\item Theorem \ref{theorem:Iso-Deligne} applied to $\Xcal_{\tau}(\CBbb)$, $\Lcal_{\chi_{\tau}}$, $\Lcal_{\chi_{\tau}}^{c}$;
	\item the use of the connections and rigidifications in order to identify $\Lcal_{\tau}$ to $\Lcal_{\chi_{\tau}}$ and $\Lcal_{\overline{\tau}}^{c}$ to $\Lcal_{\chi_{\tau}}^{c}$, plus the compatibility of Deligne's isomorphism to isomorphisms of line bundles.
\end{enumerate}
This provides a statement in a finer form, at the level of $\PIC^{\sharp}( S)$. We conclude by applying the arithmetic degree $\deg^{\sharp}$. For the last claim, it is enough to observe first that the arithmetic intersection numbers are well defined in $\CBbb/2\pi i\, \ZBbb$, and that the sign ambiguity in Deligne's isomorphism disappears, since there is an even number of different embeddings from $K$ into $\CBbb$.
\end{proof}

\paragraph{\textbf{Variant in the absence of rigidification}} In practical situations, while a section $\sigma$ of $\pi\colon\Xcal\to S$ may be given, a natural choice of rigidification may not. As we explained in Remark \ref{remark:rigid} and in Section \ref{subsec:mixed-prod}, this can be remedied by base changing to the Hilbert class field of $K$. For instance, we justified that mixed intersection products $(\Lcal^{\sharp},\ov{\Mcal})$ are naturally defined in $\CBbb/\pi i\,  \ZBbb[1/h_{K}]$. For the determinant of cohomology $\lambda(\Lcal^{\sharp})$ it is even simpler, since the rigidification is only needed in the construction of the logarithms, which happen on the archimedean places. Clearly, $\lambda(\Lcal^{\sharp})$ can be defined over $\Ocal_{K}$ if it is defined after base change to $\Ocal_{H}$.
\begin{corollary}
Let $\Xcal\to S$ be an arithmetic surface with $\sigma\colon S\to\Xcal$ a given section. Fix a hermitian metric on $\omega_{\Xcal/ S}$. Let $\Lcal^{\sharp}$ be a conjugate pair of line bundles with connections. Then, the formula \eqref{eq:ARR-flat} holds with values in $\CBbb/\pi i\,  \ZBbb[1/h_{K}]$, where $h_K$ is the class number of $K$.
\end{corollary}
\begin{proof}
After Theorem \ref{theorem:arithmetic-RR}, it is enough to base change to the Hilbert class field, and use the functoriality of the arithmetic degree and the compatibility of the determinant of cohomology with base change.
\end{proof}

\paragraph{\textbf{Variant over $\Spec\CBbb$}} There is an interesting version of Theorem \ref{theorem:arithmetic-RR} when the base scheme $\Spec\CBbb$, when the argument is still well defined and with values in $\RBbb/2\pi\ZBbb$. The formula dramatically simplifies:
\begin{theorem}[{\sc Argument of arithmetic Riemann-Roch}]\label{theorem:argument-ARR}
When the base scheme is $\Spec\CBbb$, there is the following equality of arguments in $\RBbb/2\pi\ZBbb$:
\begin{displaymath}
	12\arg^{\sharp}\lambda(\Lcal^{\sharp})_{Q}=6\arg^{\sharp}\langle\Lcal^{\sharp},\Lcal^{\sharp}\rangle-6\arg^{\sharp}\langle\Lcal^{\sharp},\overline{\omega}_{\Xcal/S}\rangle.
\end{displaymath}
\end{theorem}
\begin{example}
Let $X$ be a compact Riemann surface with a fixed base point $p$. Let $L^{\sharp}$ be a conjugate pair of rigidified line bundles with connections. Assume the connections have real holonomies, that $L^{c}=L^{\vee}$ and the rigidification is induced by a trivialization of $L$ alone. Because we are in the real holonomy case, the explicit description of the intersection logarithm in Section \ref{subsec:real-holonomies} shows that $\arg^{\sharp}\langle L^{\sharp}, L^{\sharp}\rangle=0$.
For the other intersection product, recall we saw in Section \ref{subsection:argument} that $L^{\sharp}$ determines a harmonic differential form $\theta_{L^{\sharp}}$ with periods in $2\pi i\, \ZBbb$. Then, if $\omega_{X/\CBbb}=\Ocal(K)$ for some canonical divisor $K$, we have
\begin{displaymath}
	\arg^{\sharp} \langle L^{\sharp},\omega_{X/\CBbb}\rangle=\Imag\left(\int_{\tilde{p}}^{K}\theta_{L^{\sharp}}\right).
\end{displaymath}
Now the argument of the arithmetic Riemann-Roch theorem in this particular case specializes to
\begin{displaymath}
	12\arg^{\sharp}\lambda(L^{\sharp})_{Q}=-6\Imag\left(\int_{\tilde{p}}^{K}\theta_{L^{\sharp}}\right)\quad\text{in }\RBbb/2\pi\ZBbb.
\end{displaymath}
This can be seen as an anomaly formula for the imaginary part of the Quillen-Cappell-Miller logarithm, under a change of connection (within the real holonomy assumption).
\end{example}

\bibliographystyle{amsplain}
\bibliography{biblio}{}

 \end{document}